\documentclass[10pt,leqno,twoside]{amsart}

\usepackage{enumerate}

\usepackage{amssymb}

\usepackage{amssymb,amsthm,amsmath,eucal,mathrsfs}
\usepackage{bm}

\newcommand*{\QEDB}{\hfill\ensuremath{\square}}%

\usepackage{tikz, subfigure, xcolor} 

\usepackage{pgfplots}

\usepackage{cite}

\usepackage{amsmath,verbatim}

\usepackage{amsthm}

\usepackage{amssymb}

\usepackage{amsfonts}

\usepackage{dsfont}

\usepackage{hyperref}

\newtheorem{theorem}{Theorem}[section]

\newtheorem{lemma}[theorem]{Lemma}

\newtheorem{proposition}[theorem]{Proposition}

\newtheorem{claim}{Claim}

\theoremstyle{definition}

\newtheorem{definition}[theorem]{Definition}

\newtheorem{remark}[theorem]{Remark}

\numberwithin{equation}{section}

\providecommand{\norm}[1]{\lVert#1\rVert} 

\newcommand{\1}{\mathbb{I}}

\newcommand\restr[2]{{
  \left.\kern-\nulldelimiterspace 
  #1 
  \vphantom{\big|} 
  \right|_{#2} 
  }}

\usepackage{eucal}

\title[Point-interaction approximation for bubbles]{The point-interaction approximation for the fields generated by contrasted bubbles at arbitrary fixed frequencies}
\author[Ammari, Challa, Choudhury, Sini]{Habib Ammari $^* $, Durga Prasad Challa $^{**} $, Anupam Pal Choudhury $^{\dag} $, Mourad Sini$^{\ddag} $}

\keywords{bubbly media, Foldy-Lax approximation, effective medium theory.\\
$^* $ Department of Mathematics, ETH Z\"urich, R\"amistrasse 101, CH-8092 Z\"urich, Switzerland. E-mail: habib.ammari@math.ethz.ch\\
$^{**} $ Faculty of Mathematics, Indian Institute of Technology Tirupati, Tirupati, India. Email: chsmdp@iittp.ac.in. This author was partially supported by the Austrian Science Fund (FWF): P28971-N32 and DST SERB MATRICS (Mathematical Research Impact Centric Support) MTR/2017/000539.\\
$^{\dag}$ RICAM, Austrian Academy of Sciences,
Altenbergerstrasse 69, A-4040, Linz, Austria. Email: anupampcmath@gmail.com. This author is supported by the Austrian Science Fund (FWF): P28971-N32.  \\ 
$^{\ddag}$ RICAM, Austrian Academy of Sciences,
Altenbergerstrasse 69, A-4040, Linz, Austria.
Email: mourad.sini@oeaw.ac.at. This author is partially supported by the Austrian Science Fund (FWF): P28971-N32}

\allowdisplaybreaks
\begin{document}
\begin{abstract}  
 We deal with the linearized model of the acoustic wave propagation generated by small bubbles in the harmonic regime. We estimate the waves generated by a cluster of $M$ small bubbles, distributed in a bounded domain $\Omega$, with relative densities having contrasts of the order $a^{\beta}, \beta>0, $
 where $a$ models their relative maximum diameter, $a\ll 1$.
 We provide useful and natural conditions on the number $M$, the minimum distance and the contrasts parameter $\beta$ of the small bubbles under which the point interaction approximation 
 (called also the Foldy-Lax approximation) is valid. 
 
 With the regimes allowed by our conditions, we can deal with a general class of such materials. Applications of these expansions in material sciences and imaging are immediate. 
 For instance, they are enough to derive and justify the effective media of the cluster of the bubbles for a class of gases with densities having contrasts of the order $a^{\beta}$, $\beta \in (\frac{3}{2}, 2)$ and 
 in this case we can handle any fixed frequency. In the particular and important case $\beta=2$, we can handle any fixed frequency far or close (but distinct) from the corresponding Minnaert resonance. The cluster of the bubbles can be distributed to generate volumetric metamaterials but also low dimensional ones as  metascreens and metawires.   
\end{abstract}
\maketitle
\section{Introduction}
Diffusion by highly contrasted small particles is of fundamental importance in several branches of applied sciences as material sciences and imaging. 
We are interested in the models where these small particles have sizes at the micro scales as in the models related to gas bubbles.
To describe properly the mathematical model we are dealing with in this work, let us denote by  $\{D_{s}\}_{s=1}^{M}$ a finite collection of small particles in $\mathbb{R}^{3}$ of the form $D_{s}:= \delta B_{s}+z_{s}$, where
 $B_{s}$ are open, bounded (with Lipschitz boundaries), simply connected sets in $\mathbb{R}^{3}$ containing the origin, and $z_{s}$ specify the locations of the 
 particle. The parameter $\delta > 0 $ characterizes the smallness assumption on the particles. We shall further assume that the Lipschitz constants of the open sets $B_{s}$ are uniformly bounded. 
Let us consider piecewise constant densities of the form 
\begin{equation}
\rho_{\delta}(x)=\begin{cases}
                 \rho_{0}, \ x \in \mathbb{R}^{3}\diagdown \overline{\cup_{l=1}^{M} D_{l}},\\
                 \rho_{s}, \ x \in D_{s}, \ s=1,...,M, \end{cases}
\label{model1}
\end{equation}
%
%
and piecewise constant bulk modulus in the analogous form
\begin{equation}
k_{\delta}(x)=\begin{cases}
                k_{0}, \ x \in \mathbb{R}^{3}\diagdown \overline{\cup_{l=1}^{M} D_{l}},\\
                k_{s}, \ x \in D_{s}, \ s=1,...,M,
               \end{cases}
\label{model2}
\end{equation}
where $\rho_{0},\rho_{s}, k_{0},k_{s}$ are positive constants. Thus $\rho_{0}$ and $k_{0}$ denote the density and bulk modulus of the 
background medium and $\rho_{s}$ and $k_{s}$ denote the density and bulk modulus of the bubbles respectively.\\
We are interested in the following problem describing the acoustic scattering  by the collection of small bubbles $D_s, s=1, ..., M$:
\begin{equation}
\begin{cases}   \nabla. (\frac{1}{\rho_{0}} \nabla u) +\omega^{2} \frac{1}{k_{0}} u =0 \ \text{in} \ \mathbb{R}^{3}\diagdown \overline{\cup_{l=1}^{M} D_{l}},\\
   \nabla. (\frac{1}{\rho_{s}} \nabla u) +\omega^{2} \frac{1}{k_{s}} u =0 \ \text{in}\ D_{s},\ s=1,...,M,\\
   \left.u \right\vert_{-}-\left.u \right\vert_{+}=0, \ \text{on}\ \partial D_{s},\ s=1,...,M,\\
   \frac{1}{\rho_{s}}\left.\frac{\partial u}{\partial \nu^{s}} \right\vert_{-}-\frac{1}{\rho_{0}} \left.\frac{\partial u}{\partial \nu^{s}} \right\vert_{+} =0\ \text{on} \ \partial D_{s},\ s=1,...,M,
   \end{cases}
\label{model3}
\end{equation}
where $\omega > 0$ is a given frequency and $\nu^s $ denotes the external unit normal to $\partial D_{s} $ . Here the total field $u:=u^{I}+u^{s}$, where $u^{I}$ denotes the incident field (we restrict to
plane incident waves) and $u^{s}$ denotes the scattered waves. The above set of equations has to be supplemented by the \textit{Sommerfeld radiation
condition} on $u^{s}$ which we shall henceforth refer to as $(S.R.C)$.\\ 
Keeping in mind the positivity of the bulk modulus, the above problem can be equivalently formulated as 
\begin{equation}
 \begin{cases}
  \Delta u + \kappa_{0}^{2} u =0 \ \text{in} \ \mathbb{R}^{3}\diagdown \overline{\cup_{l=1}^{M} D_{l}}, \\
   \Delta u + \kappa_{s}^{2} u =0 \ \text{in}\ D_{s},\ s=1,...,M,\\
   \left.u \right\vert_{-}-\left.u \right\vert_{+}=0, \ \text{on} \ \partial D_{s},\ s=1,...,M,\\
   \frac{1}{\rho_{s}}\left.\frac{\partial u}{\partial \nu^{s}} \right\vert_{-}-\frac{1}{\rho_{0}} \left.\frac{\partial u}{\partial \nu^{s}} \right\vert_{+} =0\ \text{on} \ \partial D_{s},\ s=1,...,M,\\
   \frac{\partial u^{s}}{\partial \vert x \vert}-i \kappa_{0} u^{s} =o (\frac{1}{\vert x \vert}) , \ \vert x \vert \rightarrow \infty \ (S.R.C),
 \end{cases}
\label{model4}
\end{equation}
where $\kappa_{0}^{2}=\omega^{2} \frac{\rho_0}{k_0}$ and $\kappa_{s}^{2}=\omega^{2} \frac{\rho_s}{k_s}$. As in \eqref{model3}, the total field
$u:=u^{I}+u^{s}$, with $u^{I}$ denoting the acoustic incident field and $u^{s}$ denoting the acoustic scattered field. The scattered field $u^s $ can be expanded as 
\[u^{s}(x,\theta)= \frac{e^{i \kappa_{0} \vert x \vert}}{\vert x \vert} u^{\infty}(\hat{x}, \theta)  + O(\vert x \vert^{-2}), \ \vert x \vert \rightarrow \infty, \]
where $\hat{x}:=\frac{x}{\vert x \vert} $ and $u^{\infty}(\hat{x}, \theta) $ denotes the far-field pattern corresponding to the unit vectors $\hat{x},\theta $, i.e. the incident and propagation directions respectively. 
\begin{definition} 
\label{Def1}
To describe the collection of small bubbles, we use the following parameters:
\begin{enumerate}
\item $
a:=\max\limits_{1\leq m\leq M } diam (D_m) ~~\big[=\delta \max\limits_{1\leq m\leq M } diam (B_m)\big],$

 \item  $
d:=\min\limits_{\substack{m\neq j\\1\leq m,j\leq M }} d_{mj},$
where $\,d_{mj}:=dist(D_m, D_j)$,

\item $\omega_{\max}$ as the upper bound of the used wave numbers, i.e. $\omega\in(0,\,\omega_{\max}]$,

\item there exist constants $\zeta_{m} \in (0,1] $ such that \[B^{3}_{\zeta_{m} \frac{a}{2}}(z_m) \subset D_m \subset B^{3}_{\frac{a}{2}}(z_m), \] where $B^{3}_{r} $ denotes a ball of radius $r$ and $\zeta_m $ are assumed to be uniformly bounded below by a positive constant.\\

The distribution of the small bubbles is modeled as follows:
\item the number $M~:=~M(a)~:=~O(a^{-s})\leq M_{max} a^{-s}$ with a given positive constant $M_{max}$,\\
\item the minimum distance $d~:=~d(a)~\approx ~a^t$, i.e. $d_{min} a^t \leq d(a) \leq d_{max}a^t $, with given positive constants $d_{min}$ and $d_{max}$, \\
\item the coefficients $k_m, \rho_m$ satisfy the conditions: 
\begin{equation}
 \frac{\rho_m}{\rho_0}= C_{\rho} a^{\beta},\ \beta>0, (\mbox{ i.e. } \frac{\rho_m}{\rho_0} \ll 1),
 \label{constant}
 \end{equation}
 keeping the relative speed of propagation uniformly bounded, i.e. 
 \begin{equation}\label{speeds}
 \frac{\kappa^2_{m}}{\kappa^2_0}:= \frac{\rho_{m}k_{0}}{k_{m} \rho_{0}}=\frac{\rho_{m}}{\rho_{0}}\frac{k_{0}}{k_{m}} \sim 1, \mbox{ as } a \ll 1.
 \end{equation}
\end{enumerate}
Here the real numbers $s$, $t$ and $\beta$ are assumed to be non negative.
\bigskip

We call the upper bounds of the Lipschitz character of $B_m$'s, $M_{max}, d_{min}, d_{max}$ 
and $\omega_{max}$ the set of the a priori bounds. 
\end{definition}

\bigskip

The scattering problem described above models the acoustic wave diffracted in the presence of small bubbles. In this case, the parameter $\beta$ fixes the kind of medium we are considering, 
see \cite{Papanicoulaou-1, Papanicoulaou-2, Habib-bubbles, Papanicoulaou-3}. To state our results, let us first denote $\hat{A_{l}}:=\frac{1}{\vert{\partial D_l}\vert}\int_{ \partial D_l}\int_{ \partial D_l}\frac{(s-s')}{\vert{s-s'}\vert} \cdot\nu_{s'} \,d\sigma_{l}(s')\,d\sigma_{l}(s)$ and define 
\[\omega_{M}^2:=\frac{8\pi\; k_l}{(\rho_l-\rho_0) \hat{A_{l}}}.\]
Note that $\hat{A}_{l} $ is negative, as $\hat{A_{l}}=-\frac{2}{\vert \partial D_{l} \vert} \int_{\partial D_{l}} \int_{D_{l}} \frac{1}{\vert s-y \vert} dy \ d\sigma_{l}(s)$ by the divergence theorem, 
and since $\rho_l $ satisfy \eqref{constant} and $a \ll 1 $, it follows that $\omega_{M}^{2} $ is positive. 
In the case $\beta=2$, to simplify the exposition of the results, we assume that the constant $\omega_{M}^{2} $ is the same for all $l=1,\dots,M$. 
For example, this can hold if all the bubbles are identical in shape, and have the same density and bulk modulus. \\
Let \[\Phi_{\kappa}(x,y):= \frac{e^{i \kappa \vert x-y \vert}}{4\pi \vert x-y \vert}, \ \text{for} \ x,y \in \mathbb{R}^3, \] 
denote the fundamental solution of the Helmholtz equation in three dimensions with a fixed wave number $\kappa $. To $ \Phi_k$, we correspond its farfield pattern $ \Phi_{\kappa}^\infty(\hat{x}, y):= e^{-i \kappa \hat{x} \cdot y}$, where $\hat{x}=\frac{x}{\vert x \vert} $.\\
The main results of this work are stated in the following theorem.

\begin{theorem}\label{Main-theorem}
Under the conditions $0\leq t < \frac{1}{2},${\footnote{The condition that $t < \frac{1}{2} $, which is general enough for the applications described later, is imposed to
keep more generalities regarding the other parameters describing the composite (i.e. $s, \beta$ and $h_1$), and bypass some technicalities in the calculations. 
For example, we required this condition to derive the apriori estimates of the densities in Proposition \ref{Prop-phi-estimate}.}} $0 \leq s \leq \frac{3}{2},$ $\beta=1+\gamma,$ with $0 \leq \gamma \leq 1$ and $s+\gamma \leq 2$ we have the following expansions.
\begin{enumerate}
\item Assume that  $\gamma<1 $ or $\gamma=1$ with $\omega$ being away from $\omega_M$, i.e. $\vert 1-\frac{\omega^2_M}{\omega^2}\vert \geq l_0$ with a positive constant $l_0$ independent of $a, a \ll 1$. 
Then
\begin{equation}\label{Away-from-resonance}
 u^\infty(\hat{x}, \theta)= \sum^M_{m=1}\Phi_{\kappa_0}^\infty(\hat{x}, z_m)Q_m +O(a^{2-s}+a^{3-\gamma-s-2t})
\end{equation}
under the additional condition on $t$: $t\geq \frac{s}{3}.$

\item Assume that $\gamma =1$ and the frequency $\omega$ is near $\omega_M$, i.e. $1-\frac{\omega^2_M}{\omega^2}=l_Ma^{h_1},\; h_1 \in (0,1),\ l_M \neq 0$. Then

\begin{equation}\label{Near-resonance}
 u^\infty(\hat{x}, \theta)= \sum^M_{m=1}\Phi_{\kappa_0}^\infty(\hat{x}, z_m) Q_m+O(a^{2-s-2h_{1}}+a^{3-2t-2s-2h_{1}})
\end{equation}
under the additional conditions on $t$ and $h_1$ given by
\bigskip

\begin{itemize}
 \item $t\geq \frac{s}{3}$ and $s+h_1\leq 1$, if $l_M<0$. 
\bigskip

\item $t\geq \frac{s}{3} $, $t+h_1 \leq 1 $, $s+h_1<\min \{\frac{3}{2}-t, 2-h_1 \} $, if $l_M>0$.

\end{itemize}
\end{enumerate}
\bigskip

The vector $(Q_m)^M_{m=1}$ is the solution of the following algebraic system
\begin{equation}\label{LAS-1-theorem}
 {\bf{C_m}}^{-1} Q_m +\sum_{l\neq m}\Phi_{\kappa_0}(z_l, z_m)Q_l=-u^I(z_m),\; m=1, ..., M,
\end{equation}
with
\begin{equation}
 {\bf{C_m}}:=\frac{\kappa_m^2 \vert{ D_m}\vert}{\frac{\rho_{m}}{\rho_m-\rho_{0}}-\frac{1}{8\pi}\kappa_m^2{\hat{A}_m}}\; \mbox{ and } 
 \hat{A}_m:=\frac{1}{\vert{\partial D_m}\vert}\int_{ \partial D_m}\int_{ \partial D_m}\frac{(s-s')}{\vert{s-s'}\vert} \cdot\nu_{s'} \,d\sigma_{m}(s')\,d\sigma_{m}(s).
\end{equation}

 \noindent The algebraic system (\ref{LAS-1-theorem}) is invertible under one of the following conditions:
 \bigskip
 
 \begin{enumerate}
  \item The coefficients ${\bf{C_m}}$ are negative and $\max \vert {\bf{C_m}}\vert =O(a^s)$, as $a \ll 1$. This condition holds if
  \bigskip
  
  \begin{enumerate}
  \item $\gamma<1$ or $\gamma=1$ with $\omega$ being away from $\omega_M$ and we have the relations $ 0 \leq \gamma \leq 1, \; \gamma+s\leq 2$ and $ \frac{s}{3}\leq t \leq 1$.
  \bigskip
  
  \item $\gamma=1$ and the frequency $\omega$ approaches $\omega_M$ from below ($l_M<0$), i.e. $\omega < \omega_M$, and we have the relations $ \frac{s}{3} \leq t \leq 1$ and $1-h_{1}-s \geq 0 $.
\end{enumerate}
\bigskip

 \item The coefficients ${\bf{C_m}}$ are positive and one of the following conditions is fulfilled
 \begin{enumerate}
 \item $\max \vert {\bf{C_m}}\vert =O(a^t)$, as $a \ll 1$, and $\tau:=\min_{1\leq j,m \leq M,\ j\neq m} cos(\kappa_{0}\vert z_{m}-z_{j} \vert)>0$. 
 The first condition holds if $\gamma=1$, the frequency $\omega$ approaches $\omega_M$ from above ($l_M>0$), i.e. $\omega > \omega_M$, and we have the relations $0\leq t \leq 1-h_{1} $ and $s \leq 1 $.
 \bigskip

 \item $\max \vert {\bf{C_m}}\vert =O(a^s)$, as $a\ll 1$. This condition holds if $\gamma=1$ and the frequency $\omega$ approaches $\omega_M$ from above ($l_M>0$), 
 i.e. $\omega > \omega_M$, and we have the relations $ \frac{s}{3} \leq t \leq 1$ and $1-h_{1}-s \geq 0 $.
 \end{enumerate}
 \end{enumerate}
 \bigskip
 
The constants appearing in the error terms of (\ref{Away-from-resonance}) and (\ref{Near-resonance}) depend only on the a priori bounds mentioned above. 
In the case $\gamma=1$ (i.e. $\beta=2$), we assume that the constant $C_\rho$, defined in (\ref{constant}) is larger then a certain constant depending only on those a priori bounds.
\end{theorem}
\begin{remark}\label{improved-estimate}
Let the vector $(Q_m)^M_{m=1}$ be the solution of the  algebraic system
\begin{equation}
 {\bf{C_m}}^{-1} Q_m +\sum_{l\neq m}\Phi_{\kappa_0}(z_l, z_m)Q_l=-u^I(z_m),\; m=1, ..., M,
 \notag
\end{equation}
where ${\bf{C_m}}^{-1} $ is now defined as 
\begin{equation}
  \mathbf{C}_m^{-1}= \frac{1}{[\kappa^{2}_{m}+\frac{\rho_m}{\rho_0}(\kappa_{m}^{2}-\kappa^{2}_{0})]} \Big[I_m^{d}-iI_m^{d} \frac{J_m}{[\kappa_m^2+\frac{\rho_{m}}{\rho_{0}}(\kappa_{m}^2-\kappa_{0}^2)]} \Big]+\frac{1}{\kappa^{2}_{m}}\Big[I_m^{s}-I^{d}_{m}\frac{(J_{m})^2}{\kappa_{m}^4} \Big]
  \notag
  \end{equation} 
   with
  \[
I_m^{d} :=\vert D_{m} \vert^{-1} \Big(\frac{\rho_m}{\rho_m-\rho_0}+\frac{1}{8\pi}\left[-\kappa_m^2-\frac{\rho_{m}}{\rho_{0}}(\kappa_m^2-\kappa_0^2)\right]{\hat{A}_m} \Big),
\]
\[
I_m^s:= i \Big[\frac{\kappa_{m}^{3}}{4\pi}-\frac{1}{32\pi^{2}} (\kappa_{0}-\kappa_{m}) \kappa^{2}_{m} \vert D_{m} \vert^{-1} \hat{A}_{m} Cap_{m} \Big]
\]
and 
\[
J_m:=\frac{\rho_{0}}{\rho_{m}-\rho_{0}} \frac{1}{4 \pi} (\kappa_{0}-\kappa_{m}) \kappa_{m}^{2} \Big[1+\frac{1}{8 \pi} \hat{A}_{m} \vert D_{m} \vert^{-1} Cap_{m} \Big] Cap_{m}\ ,
\]
where $Cap_{m}$ stands for the capacitance of $D_m$, see (\ref{capacitance}), which scales as $Cap_m =O(a)$.
\bigskip
As it is shown in Remark \ref{Improved-expansions}, if $\gamma =1$ and the frequency $\omega$ is near $\omega_M$, i.e. $1-\frac{\omega^2_M}{\omega^2}=l_Ma^{h_1},\; h_1\in (0, 1]${\footnote{The condition that $h_1\in (0, 1] $ appears naturally to ensure the invertibility of the algebraic system, see Remark \ref{extra-conditions}.}}, $l_M \neq 0$, 
the improved estimate 
\begin{equation}
 u^\infty(\hat{x}, \theta)= \sum^M_{m=1}\Phi_{\kappa_0}^\infty(\hat{x}, z_m) Q_m+O(a^{2-s-h_{1}}+a^{3-2t-2s-2h_{1}})
 \notag
\end{equation}
holds instead of \eqref{Near-resonance}. In addition, note that the case $h_1=1$ is not covered by Theorem \ref{Main-theorem}. 
\bigskip

Observe now that, in the case $\gamma=1$, $I_m^{d}=\vert D_{m} \vert^{-1} \Big(\frac{\rho_m}{\rho_m-\rho_0}-\frac{1}{8\pi}\kappa_m^2{\hat{A}_m} \Big)+O(a)$, $I_m^{s}
=i \Big[\frac{\kappa_{m}^{3}}{4\pi}-\frac{1}{32\pi^{2}} (\kappa_{0}-\kappa_{m}) \kappa^{2}_{m} \vert D_{m} \vert^{-1} \hat{A}_{m} Cap_{m} \Big] \sim 1$ and $J_m=O(a)$. 
Hence we can estimate $\mathbf{C}_m^{-1}$ as
$$
\mathbf{C}_m^{-1}=\frac{1}{\kappa^2_m}\Big ( \vert D_{m} \vert^{-1} \Big(\frac{\rho_m}{\rho_m-\rho_0}-\frac{1}{8\pi}\kappa_m^2{\hat{A}_m} \Big)+i \Big[\frac{\kappa_{m}^{3}}{4\pi}-\frac{1}{32\pi^{2}} (\kappa_{0}-\kappa_{m}) \kappa^{2}_{m} \vert D_{m} \vert^{-1} \hat{A}_{m} Cap_{m} \Big] \Big) +O(a)
$$                       
or
$$
\mathbf{C}_m^{-1}=\frac{ \hat{A}_m}{8\pi \vert D_{m} \vert} \Big(\frac{\omega^2_M}{\omega^2}-1 \Big)+i \Big[\frac{\kappa_{m}}{4\pi}-\frac{1}{32\pi^{2}} (\kappa_{0}-\kappa_{m}) \vert D_{m} \vert^{-1} \hat{A}_{m} Cap_{m} \Big]  +O(a).
$$
Now, if in addition the frequency $\omega$ is near $\omega_M$, i.e. $1-\frac{\omega^2_M}{\omega^2}=l_Ma^{h_1},\; h_1> 0,\ l_M \neq 0$, then after scaling, we derive
\begin{equation}\label{dominating-CM}
\mathbf{C}_m^{-1}=-\frac{ l_M \overline{\hat{A}}_m}{8\pi \vert B_{m} \vert} a^{h_1-1}+i \Big[\frac{\kappa_{m}}{4\pi}-\frac{1}{32\pi^{2}} (\kappa_{0}-\kappa_{m}) \vert B_{m} \vert^{-1} \overline{\hat{A}}_{m} Cap(B_m) \Big]  +O(a)
\end{equation}
where $\overline{\hat{A}}_{m}:=\frac{1}{\vert{\partial B_m}\vert}\int_{ \partial B_m}\int_{ \partial B_m}\frac{(s-s')}{\vert{s-s'}\vert} \cdot\nu_{s'} \,d\sigma_{m}(s')\,d\sigma_{m}(s)$ ans $Cap(B_m)$ is the capacitance of $B_m$. We see that this approximation makes sense for any $h_1$ in $(0, 1]$ and while for $h_1\in (0, 1)$, the first term is the most dominating, for $h_1=1$ the first two terms should be taken into account. 

\end{remark}

\bigskip
Related to the model (\ref{model4}), the asymptotic expansion is derived formally 
in \cite{Papanicoulaou-1, Papanicoulaou-2} and justified mathematically in \cite{Habib-bubbles} in the case $\beta=2$ when the frequency $\omega$ is close to the particular frequency $\omega_M$.
The particular frequency $\omega_M$  is shown to be an approximation, as $a \ll 1$, 
of a resonance known as the Minneart resonance, see \cite{Habib-Minnaert}. 
\bigskip

The approximations provided in Theorem \ref{Main-theorem} are valid for a large class of bubbles. In particular, they can be used to derive the equivalent effective media at least in the following class:
\bigskip

\begin{enumerate}
 \item If $\gamma \in (\frac{1}{2}, 1)$, which means that $\beta \in (\frac{3}{2}, 2)$, and for any frequency $\omega$.
 \bigskip
 
 \item If $\gamma =1$, i.e. $\beta =2$, for any frequency $\omega$ away or very close (but distinct) from $\omega_M$.
\end{enumerate}
\bigskip

The error in the approximation in  (\ref{Away-from-resonance}) is going to zero since $t<\frac{1}{2}$ and $s\leq \frac{3}{2}$. 
\bigskip

The error term in  \eqref{Near-resonance} goes to zero provided $s,h_{1}$ and $ t $ satisfy the condition 
\begin{equation}\label{condi-error}
 s<\min\{2-2h_1, \; \frac{3-2t-2h_1}{2}\}.
\end{equation}
These conditions are fulfilled, in particular, if
\begin{equation}
 s+h_1\leq 1,\; ~~ h_1<1\; \mbox{ and } t< \frac{1}{2}.
\end{equation}
These last conditions are the regimes in which one can derive the effective media. Actually, in the case $l_M>0$, one can also allow $s+h_1>1$ (but $s+h_1< \frac{3}{2}$) and hence generate potential walls 
which are completely reflecting any incident wave sent from outside of its support. This phenomenon is already observed and justified in the framework of acoustic waves in the presence of very large number of holes, see \cite{CMS-2017}.  \\
\bigskip
We also observe that, in our approximations, we can handle the case $s=0$ and  $h_1=1$, see Remark \ref{improved-estimate}. 
In this case, the error term goes also to zero as soon as $t<\frac{1}{2}$. 
Hence, as $a\rightarrow 0$, and choosing $t=0$ for instance, 
$u^\infty(\hat{x}, \theta)= \sum^M_{m=1}\Phi^\infty(x, z_m) Q_m$ and $(Q_m)^M_{m=1}$ is the solution of the algebraic system (\ref{LAS-1-theorem}) with
${\bf{C_m}}^{-1}:=-\frac{l_M \overline{\hat{A}}_m}{8\pi \vert B_{m} \vert}+i \Big[\frac{\kappa_{m}}{4\pi}-\frac{1}{32\pi^{2}} (\kappa_{0}-\kappa_{m}) \vert B_{m} \vert^{-1}
\overline{\hat{A}}_{m} Cap(B_m) \Big]$ .
This limiting field describes the Foldy-Lax field generated by the interaction of the point-like scatterers, given by the centers of our small bubbles, with scattering strengths modeled by $\bf{C_m}$'s, compare to \cite{Foldy, Lax} 
and see also \cite{Martin:2006}. 
This limiting field is also modeled by the Dirac-like potentials supported on the centers of the small bubbles, see \cite{A-G-H-H:AMS2005}.
 What we have shown then is that if we inject into the background small bubbles characterized by $\frac{\rho_m}{\rho_0}\approx a^2$ and the frequency $\omega$ near to the Minnaert resonance, as follows
$\frac{\omega^2}{\omega^2_M}=1+l\; a$, the  generated fields behave as the ones created by Dirac-like potentials (or point-like scatterers). This shows a way how to generate fields generated by singular potentials by
injecting regular small bubbles. These singular potentials are supported on points. In \cite{ACCS-effective-media}, we show how to generate fields due to potentials supported by $2$D surfaces (metasurfaces or metascreens) or $3$D domains (metamaterials) by injecting in the background bubbles distributed on $2D$ surfaces or in $3D$ domains. The case of $1D$ curves (i.e. metawires) will be treated separately in a forthcoming work as the scattering by $1D$ curves is more delicate, see \cite{B-R-K-L-R-2017} and the references therein. Thanks to the kind of approximations we provide in Theorem \ref{Main-theorem}, these different settings of the metamaterials can be treated in a unified way, namely as Dirac type potentials supported on curves, surfaces or domains.

The analysis of the wave propagation in the presence of highly heterogeneous media is an object of extensive study, see \cite{Pan:2000}. An important situation is when the heterogeneity of the medium 
is modeled by the high relative contrasts as the relative densities or the relative bulk modulus described above. 
In the recent years there was an increase of interest in describing the effective macroscopic models generated by periodically distributed microscopic structures characterized by high contrasted media, 
see \cite{L-S:2016} and the references therein, 
based on homogenization techniques. Compared to these works, we need to study the expansions of the fields generated by contrasted small bubbles that are not necessarily periodically distributed. 
Besides, we focus on the precise description of the dominant parts of the fields as the cluster of the bubbles becomes dense. 
For this, we propose to derive the point-interaction approximation of those fields using integral equation methods coupled with asymptotic expansions tools. This point-interaction approach has its roots back to the Foldy approximation method known in 
several branches of physics and engineering, see \cite{A-G-H-H:AMS2005, Martin:2006}. The advantage of this approach is that we can characterize clearly the dominant fields generated by the interaction of the small bubbles between each other 
and also with the background medium.  This approach was already tested and justified in our previous works \cite{Challa-Sini-2, A-A-C-K-S} when the contrast is modeled by high surface impedances and in \cite{Habib-bubbles} where the case 
$\gamma=1$ (i.e. $\beta=2$) and the frequency $\omega$ is close to the resonance $\omega_M$. The purpose of our work here is to extend those last results to more general values of $\beta$ and the whole range of frequencies 
$\omega$. 

Related to this bubbles model, other results were derived very recently. In particular, in the case $\beta =2$ and $\omega$ close to the resonance $\omega_M$, we find in \cite{H-F-G-L-Z-1} a 
mathematical framework for modeling metasurfaces with bubbles, in \cite{H-F-G-L-Z} a justification of the superfocusing
of acoustic waves in the presence of gas bubbles and in \cite{H-F-L-Y-Z} a justification of the bandgap opening due to periodically distributed bubbles.

As compared to the results in \cite{Habib-bubbles, H-F-G-L-Z}, using our derived estimates, we can handle not only frequencies near the Minnaert resonance but arbitrary other fixed frequencies\footnote{This, of course,  
means that we are in the Rayleigh regime, i.e. $\omega a <<1$ as $a<<1$.} and any gas modelled by densities having
contrasts of the order $a^{\beta}$ with $\beta \in (\frac{3}{2}, 2]$. 
In the particular case when $\beta=2$ and $\omega$ is close to the resonance $\omega_M$, we retrieve the results in \cite{Habib-bubbles}. Namely taking $s+h_1=1$ in (\ref{Near-resonance}), we see that we can get 
the effective media with an additive coefficient changing sign depending if $\omega$ is lower or higher than $\omega_M$. The equivalent medium is absorbing if $\omega < \omega_M$ and reflecting if $\omega >\omega_M$. 
In addition, and as discussed above, if $\omega$ is close to $\omega_M$ and $\omega > \omega_M$, i.e. with $l_M>0$ in (\ref{Near-resonance}), we can also take $1< s+h_1 <\frac{3}{2}$. In these cases, the equivalent media behave as extremely reflecting media
allowing no incident wave to penetrate inside it, see \cite{CMS-2017} for a different but related setting using holes.
The quantification and justification of these results are reported in \cite{ACCS-effective-media}. 

\bigskip

The rest of the paper is divided into the following sections. In section \ref{Proof-brief-description}, we describe briefly the main steps of the proof. 
In section \ref{Proof-main}, we give the core proof of the result while in section \ref{Proof-details}, we provide the detailed proofs of the main tools used in section 2.

\section{A brief description of the proof}\label{Proof-brief-description}

We recall that the single layer potential, double layer potential and the adjoint of the double layer potential are defined as
\begin{equation}\label{field-representation-description}
\begin{aligned}
&S^{\kappa}_{D_{m}}\phi(x):=\int_{\partial D_{m}} \Phi_{\kappa}(x,y) \phi(y) d\sigma_{m}(y),\\
&K^{\kappa}_{D_{m}}\phi(x):=\int_{\partial D_{m}} \partial_{\nu_{y}} \Phi_{\kappa}(x,y) \phi(y) d\sigma_{m}(y),\\
&(K^{\kappa}_{D_{m}})^{*}\phi(x):=\int_{\partial D_{m}} \partial_{\nu_{x}} \Phi_{\kappa}(x,y) \phi(y) d\sigma_{m}(y).
\end{aligned}
\end{equation}
The problem \eqref{model4} is well-posed (see \cite{Ammari-Kang-2, Ammari-Kang-1, DK} for instance) and the total field can be represented as
\begin{equation}
u(x)=\begin{cases}
      u^{I}+\sum^{M}_{l=1} S^{\kappa_{0}}_{D_{l}} \phi_{l}(x), \ x \in \mathbb{R}^{3}\diagdown \overline{\cup_{l=1}^{M} D_{l}},   \\
       S^{\kappa_{s}}_{D_{s}} \psi_{s}(x) , \ x \in D_{s}, \ s=1,...,M,           \end{cases}
\label{sol10}
\end{equation}
where the densities $\phi_{l}, \psi_{l} $ satisfy the system of integral equations
\begin{align*}
 & \left.\Big(S^{\kappa_{i}}_{D_{i}} \psi_{i}-\sum_{l=1}^{M}S^{\kappa_{0}}_{D_{l}} \phi_{l} \Big) \right\vert_{\partial D_{i}} = u^I \vert_{\partial D_{i}}, \\
 & \frac{\rho_0}{\rho_{i}} \Big[\frac{1}{2} Id +(K^{\kappa_{i}}_{D_{i}})^{*} \Big]\psi_{i}(x)- 
 \Big[-\frac{1}{2} Id+(K^{\kappa_{0}}_{D_{i}})^{*} \Big] \phi_{i}(x) -\left.\sum_{{l=1}\atop {l\neq i}}^{M} \frac{\partial(S^{\kappa_{0}}_{D_{l}} \phi_{l})}{\partial \nu^{i}} \right\vert_{\partial D_{i}} 
 =\frac{\partial u^{I}}{\partial \nu^{i}}  \vert_{\partial D_{i}}.
 \end{align*}
As mentioned in the introduction, a characteristic of the bubbles' model is that the relative speed of propagation is uniformly bounded, precisely, we assumed that 
$ \frac{\kappa^2_{m}}{\kappa^2_0} \sim 1, \mbox{ as } a \ll 1.$ To describe the main steps of the proof, we assume in this section that $
 \frac{\kappa^2_{m}}{\kappa^2_0}=1, \mbox{ as } a \ll 1.$ Under this condition, the system of integral equations simplifies as follows.
For $l=1\dots,M$,
\begin{equation}
S^{\kappa_{0}}_{D_{l}} (\psi_{l}-\phi_{l})-\sum_{{m=1} \atop {m \neq l}}^{M} S^{\kappa_{0}}_{D_{m}} \phi_{m}= u^{I},\ \mbox{ on } {\partial D_{l}}, 
\label{mult-obs-102r-description}
\end{equation}
%
\begin{align}\label{mult-obs-103r-description}
&\frac{\rho_{0}}{\rho_{l}} [\frac{1}{2} Id+(K^{\kappa_{0}}_{D_{l}})^{*}] \psi_{l}-[-\frac{1}{2} Id+(K^{\kappa_{0}}_{D_{l}})^{*}] \phi_{l} - \sum_{{m=1} \atop {m \neq l}}^{M} \frac{\partial (S^{\kappa_{0}}_{D_{m}} \phi_{m})}{\partial \nu^{l}} 
= \frac{\partial u^{I}}{\partial \nu^{l}},\  \mbox{ on } {\partial D_{l}}.
\end{align}
As both the left and right hand sides of (\ref{mult-obs-102r-description}), extended in $D_i, i=1,,,, M$, satisfy the same Helmholtz equation, they are equal there.
Taking the normal derivative from inside $D_{l}$, we have

\begin{align}\notag
 \frac{\rho_{0}}{\rho_{l}} \Big[\frac{1}{2} Id+(K^{\kappa_{0}}_{D_{l}})^{*} \Big] (\psi_{l}-\phi_{l})& - \sum_{{m=1} \atop {m \neq l}}^{M} \frac{\rho_{0}}{\rho_{l}} \frac{\partial (S^{\kappa_{0}}_{D_{m}}\phi_{m})}{\partial \nu^{l}} \Big\vert_{\partial D_{l}}
=\frac{\rho_{0}}{\rho_{l}} \frac{\partial u^{I}}{\partial \nu^{l}}\Big\vert_{\partial D_{l}}.
\end{align}
 Next we use this in the second identity \eqref{mult-obs-103r-description} to derive
\begin{equation} \label{System-phi-only}
-\frac{1}{2} \Big[\frac{\rho_{0}+\rho_{l}}{\rho_{0}-\rho_{l}}  \Big] \phi_{l} - (K^{\kappa_{0}}_{D_{l}})^{*} \phi_{l}- 
\sum_{{m=1} \atop {m \neq l}}^{M} \frac{\partial (S^{\kappa_{0}}_{D_{m}}\phi_{m})}{\partial \nu^{l}}\Big\vert_{\partial D_{l}} =  \frac{\partial u^{I}}{\partial \nu^{l}},\mbox{  on } \partial D_l,
\end{equation}
which we rewrite as
\begin{equation} \label{System-phi-only-2}
-\frac{1}{2} \Big[\frac{\rho_{0}+\rho_{l}}{\rho_{0}-\rho_{l}}  \Big] \phi_{l} - (K^{0}_{D_{l}})^{*} \phi_{l}- 
\sum_{{m=1} \atop {m \neq l}}^{M} \frac{\partial (S^{\kappa_{0}}_{D_{m}}\phi_{m})}{\partial \nu^{l}}\Big\vert_{\partial D_{l}} =  \frac{\partial u^{I}}{\partial \nu^{l}} +
[(K^{\kappa_{0}}_{D_{l}})^{*}-(K^{0}_{D_{l}})^{*}] \phi_{l},\; \mbox{  on } \partial D_l.
\end{equation}
The aim now is to estimate $\int_{\partial D_{l}} \phi_{l} \ d\sigma_{l} $ from \eqref{System-phi-only-2} above, leading us to the final algebraic system. 
By integrating \eqref{System-phi-only-2} on $\partial D_l$, and since $K^{{0}}_{D_{l}}(1)=-\frac{1}{2}$, we obtain
\begin{equation}
-\frac{1}{2} \Big[\frac{\rho_{0}+\rho_{l}}{\rho_{0}-\rho_{l}}-1 \Big] \int_{\partial D_{l}} \phi_{l} \ d\sigma_{l} =\int_{\partial D_{l}} \frac{\partial u^{I}}{\partial \nu^{l}} \ d\sigma_{l}
+ \int_{\partial D_{l}} \Big[(K^{\kappa_{0}}_{D_{l}})^{*}-(K^{0}_{D_{l}})^{*} \Big] \phi_{l} \ d\sigma_{l}
+\sum_{{m=1} \atop {m \neq l}}^{M} \int_{\partial D_{l}} \frac{\partial (S^{\kappa_{0}}_{D_{m}}\phi_{m})}{\partial \nu^{l}} \ d\sigma_{l}\\
\label{mult-obs-109r-description}
\end{equation}

Using the expansions of the fundamental solution, we derive the following estimates:
\begin{align} \label{claim-adjdblle-Lay-dif-prime-description}
\int_{\partial D_l}\left[(K^{\kappa_{0}}_{D_{l}})^{*}-(K^{0}_{D_{l}})^{*}\right]\phi_l(s) d\sigma_{}(s)
=&\frac{1}{8\pi}\kappa_0^2\int_{\partial D_l}\phi_l(s)\left[\int_{ \partial D_l}\frac{(s-t)}{\vert{s-t}\vert} \cdot\nu_t \,dt\right]d\sigma_{l}(s) \\
&\qquad \qquad -\frac{i}{4\pi}\kappa_0^3\vert{D_l}\vert\int_{\partial D_l}\phi_l(s)d\sigma_{l}(s) 
+ O\left(a^5\|\phi_l\|\right),\notag
\end{align}

 \begin{align}  \label{claim-normsingle-Lay--description}
 \int_{\partial D_l}\frac{\partial (S^{\kappa_{0}}_{D_{m}} \phi_{m})}{\partial \nu^{l}}
=&-\kappa_0^2\vert{ D_l}\vert\left[\Phi_{\kappa_0}(z_l,z_m)\int_{\partial D_m}\phi_m(s)\,d\sigma_{m}(s)\right.\\
&\left.+\nabla_t\Phi_{\kappa_0}(z_l,z_m)\cdot\int_{\partial D_m}(s-z_m)\phi_m(s)\,d\sigma_{m}(s)\right]\notag\\
&-\kappa_{0}^{2}\nabla_x\Phi_{\kappa_{0}}(z_{l},z_m)\cdot\Big[ \int_{D_{l}}(x-z_l) dx\Big] \int_{\partial D_{m}}   
\phi_{m}(s) \,d\sigma_{m}(s)+O\left(\frac{a^6}{d^3_{ml}}\|\phi_m\|\right),\, m\neq\;l,\notag
\end{align}

\begin{align}\label{claim-norm-inci--description}
\int_{\partial D_{l}} \frac{\partial u^{I}}{\partial \nu^{l}} 
=&-\kappa_0^2\,\vert{D_l}\vert\,u^{I}(z_l)+O\left(a^4\right).
\end{align}

We set $A_{l}$ to be the function defined on $\partial D_l$ by $A_{l}(s):=\int_{ \partial D_l}\frac{(s-t)}{\vert{s-t}\vert} \cdot\nu_t \,d\sigma_{l}(t)$. It is clear that $A_l$ scales as $a^2$. 
We denote its average as $\hat{A}_{l}:= \frac{1}{\vert \partial D_{l} \vert}\int_{\partial D_{l}} A_{l}(s)\ d\sigma_{l}(s)$ and write
\begin{align}\label{Al-description}
\int_{\partial D_l}\phi_l(s)\left[\int_{ \partial D_l}\frac{(s-t)}{\vert{s-t}\vert} \cdot\nu_t\,d\sigma_{l}(t)\right]d\sigma_{l}(s)=\hat{A}_{l} \int_{\partial D_{l}} \phi_{l} \ d\sigma_{l}(s) +
\int_{\partial D_l}\phi_l(s)(A_l(s)-\hat{A}_{l})\ d\sigma_{l}(s).
\end{align}
In addition, we have the estimate $ \norm{A_{l}-\hat{A}_{l}}_{L^{2}(\partial D_{l})}=O(a^3)
$ and we can use an approximation of $\int_{\partial D_l}\phi_l(s)(A_l(s)-\hat{A}_{l})\ d\sigma_{l}(s)$ of the form

\begin{equation}\label{Al-app}
 \int_{\partial D_l}\phi_l(s)(A_l(s)-\hat{A}_{l})\ d\sigma_{l}(s)=O(a^3 \Vert \phi_l\Vert).
\end{equation}

However, the following more precise  approximation is derived from the original system (\ref{System-phi-only}), or (\ref{System-phi-only-2}):
\begin{align}\label{mult-obs-108r-pr-th-description}
\int_{\partial D_l}\Big(A_l(s)-\hat{A}_l\Big)\phi_{l}\ d\sigma_{l}(s)
 {=}&-R_l\cdot \sum_{{m=1} \atop {m \neq l}}^{M} \nabla_s\Phi_{\kappa_0}(z_l,z_m) Q_m\\
 &+O\left(a^4+a^3\sum_{{m=1}\atop {m\neq l}}^{M}\frac{a^3}{d_{ml}^3}\|\phi_m\|+a^5\|\phi_l\| \right),\nonumber
 \end{align}
 where $R_l:=\int_{\partial D_l} \Big[\Big(\lambda_{l} Id+ K^{0}_{D_{l}} \Big)^{-1}(A_l(\cdot)-\hat{A}_l)\Big] (s)\nu_l(s)\,d\sigma_{l}(s) $ and $\lambda_l:=\frac{1}{2} \frac{\rho_0+\rho_l}{\rho_0-\rho_l} $.
 The value $R_l$ scales as $a^4$. 

\bigskip

Also, the term $\int_{\partial D_m}(s-z_m)\phi_m(s)\,d\sigma_{m}(s)$ can be estimated either directly, but roughly, as 
\begin{equation}\label{Vl-rough-estimates-description}
 \int_{\partial D_m}(s-z_m)\phi_m(s)\,d\sigma_{m}(s) =O(a^2\Vert \phi_m\Vert)
\end{equation}
or more precisely, by using the original system (\ref{System-phi-only}), as
\begin{equation}\label{Vl-estimates-description}
 \int_{\partial D_m}(s-z_m)\phi_m(s)\,d\sigma_{m}(s)=a^{3-\gamma}+O([a^3+lot ]\Vert \phi_m\Vert),
\end{equation}
where the abbreviation $lot$ denotes the lower order terms. \\
To give an idea on what to keep and what to neglect in these approximations, let us consider the regime where $\gamma=1$ and $s<1$.
\bigskip

As  both $A_l$ and $-\frac{1}{2} \Big[\frac{\rho_{0}+\rho_{l}}{\rho_{0}-\rho_{l}} -1 \Big]$ scale as $a^2$, the first term of (\ref{claim-adjdblle-Lay-dif-prime-description}) cannot be negleted.
Putting the second term of (\ref{claim-adjdblle-Lay-dif-prime-description}), the second and the third terms of (\ref{claim-normsingle-Lay--description}) as error terms and plugging 
(\ref{claim-adjdblle-Lay-dif-prime-description})-(\ref{claim-normsingle-Lay--description})-(\ref{claim-norm-inci--description})-(\ref{Al-description})-(\ref{mult-obs-108r-pr-th-description}), we derive the algebraic system, 
for the vector $(Q_l)^M_1$, with $Q_l:=\int_{\partial D_{l}} \phi_{l}\ d\sigma_{l}(s) $,

\begin{equation}\label{Particular-case-Minnaert}
\left[-\frac{1}{2} \Big[\frac{\rho_{0}+\rho_{l}}{\rho_{0}-\rho_{l}} -1 \Big]-\frac{\hat{A}_{l}}{8\pi}\kappa_0^2 \right]Q_l +\kappa_0^2\vert{ D_l}\vert \sum_{{m=1} \atop {m \neq l}}^{M}\Phi_{\kappa_0}(z_l,z_m)Q_m=
-\kappa_0^2\,\vert{D_l}\vert\,u^{I}(z_l)+O\left(a^4\right) +O\left(\sum_{{m=1} \atop {m \neq l}}^{M} \frac{a^5}{d^3_{ml}}\Vert \phi_m\Vert \right).
\end{equation}
By a careful counting of the bubbles, distributed in a bounded domain $\Omega$, we show that $\sum_{{m=1} \atop {m \neq l}}^{M} \frac{1}{d^3_{ml}}=O(d^{-3} \ln(d))$.

We show that the system (\ref{System-phi-only}) is invertible and that $\Vert \phi_l\Vert =O(a^{-1})$ for $\gamma=1$. In addition, as the domain $\Omega$ where we inject the bubbles is bounded hence
its volume will be of the order $d^{-3}$, hence $a^{s}\sim d^{-3}$ and then $d\sim a^{\frac{s}{3}}$. With these estimates, we see that the remaining terms of order $O(a^4) $ and 
$O(\sum_{{m=1} \atop {m \neq l}}^{M} \frac{a^5}{d^3_{ml}}\Vert \phi_m\Vert)\sim \frac{a^4 \ln(d)}{d^3} \sim a^3 a^{1-s}\ln(a)$ are dominated by the first term $\kappa_0^2\,\vert{D_l}\vert\,u^{I}(z_l)$ 
which is of the order $a^3$ (as $s<1$).
\bigskip

Now, from (\ref{field-representation-description}), we see that for $x,\; \vert x\vert >>1$, the scattered field $u^s(x):=u(x)-u^i(x)$ given by  
$u^s(x)=\sum^M_{i=1}\int_{\partial D_i}\Phi_k(x, y)\phi_i(z)d\sigma_i(y)$ can be estimated as:
\begin{equation}\label{appro-scat-description}
 u^s(x)=\sum^M_{i=1}\Big[\Phi_k(x, z_i)Q_i +\nabla \Phi_k(x, z_i)\int_{\partial D_i}(y-z_i)\phi_i(z)d\sigma_i(y) 
 +O(a^3 \Vert \phi_i\Vert)\Big].
\end{equation}
Using (\ref{Vl-estimates-description})\footnote{If we use (\ref{Vl-rough-estimates-description}) instead of (\ref{Vl-estimates-description}), then we will have an error of the order $a^{1-s}$ instead of $a^{2-s}$ in 
(\ref{appro-scat-description-2})} and the fact that $\Vert \phi_i\Vert =O(a^{-1})$, we obtain
\begin{equation}\label{appro-scat-description-2}
 u^s(x)=\sum^M_{i=1}\Phi_k(x, z_i)Q_i +a^{2-s}.
\end{equation}
Finally, combining (\ref{appro-scat-description-2}) and the algebraic system (\ref{Particular-case-Minnaert}), via its invertibility propertiy, we derive the final estimates of the scattered fields.

\bigskip
However, if we take into account the terms appearing in (\ref{claim-adjdblle-Lay-dif-prime-description})-(\ref{claim-normsingle-Lay--description})-(\ref{claim-norm-inci--description})) and the first order terms
in the approximation of (\ref{mult-obs-108r-pr-th-description}) and (\ref{Vl-estimates-description}) respectively, then we can handle, not only the case $\gamma=1$ and $s<1$, 
but a wider range of the parameters as described in our main theorem. 
\bigskip

The basic system (\ref{System-phi-only}) is derived under the conditions $\kappa_m=\kappa_0, m=1, ..., M$. In this case it involves only the densities $\phi_m$'s. 
If these conditions are not satisfied, then the corresponding system involves also the densities $\psi_m$'s. Hence, we need to handle this coupling.
However, apart from additional highly technical issues, the scheme of the proof is as described above.

\section{Proof of Theorem \ref{Main-theorem}}\label{Proof-main}
\subsection{Representation of the solutions}

As discussed in section \ref{Proof-brief-description}, the problem \eqref{model4} is well-posed and the total field can be represented as
\begin{equation}
u(x)=\begin{cases}
      u^{I}+\sum^{M}_{l=1} S^{\kappa_{0}}_{D_{l}} \phi_{l}(x), \ x \in \mathbb{R}^{3}\diagdown \overline{\cup_{l=1}^{M} D_{l}},   \\
       S^{\kappa_{s}}_{D_{s}} \psi_{s}(x) , \ x \in D_{s}, \ s=1,...,M,           \end{cases}
\label{sol1}
\end{equation}
where $\phi_{l}, \psi_{l} $ are appropriate densities. Observe that we represent the field inside each $D_{s},\ s=1,\dots,M, $ using a single density. This simplifies the presentation of the computations.\\ 
Using \eqref{sol1}, the jump conditions across the boundary of the obstacles in \eqref{model4} can be reformulated as follows. 
From the third identity in \eqref{model4}, we derive that on $\partial D_{s}$, 
\begin{align*}
0 &= u \Big\vert_{+}(x)- u \Big\vert_{-}(x) = \left.u^{I} \right\vert_{\partial D_{s}}(x) +\sum_{l=1}^{M} S^{\kappa_{0}}_{D_{l}} \phi_{l}\Big\vert_{+} (x)- S^{\kappa_{s}}_{D_{s}} \psi_{s}\Big\vert_{-} (x).
\end{align*}
Since the single-layer potentials are continuous, this implies that
\begin{equation}
S^{\kappa_{s}}_{D_{s}} \psi_{s}\Big\vert_{\partial D_{s}}(x)-\sum_{l=1}^{M} S^{\kappa_{0}}_{D_{l}} \phi_{l}\Big\vert_{\partial D_{s}}(x)=u^{I}\Big\vert_{\partial D_{s}}(x).
\label{jump1}
\end{equation}
Again
\begin{align*}
\frac{1}{\rho_{0}} \frac{\partial u}{\partial \nu^{s}} \Big\vert_{+} (x) &=\frac{1}{\rho_{0}} \frac{\partial u^{I}}{\partial \nu^{s}} \Big\vert_{\partial D_{s}} (x) + \sum_{l=1}^{M} \frac{1}{\rho_{0}} \frac{\partial(S^{\kappa_{0}}_{D_{l}} \phi_{l})}{\partial \nu^{s}} \Big\vert_{+}(x), \quad 
\frac{1}{\rho_{s}} \frac{\partial u}{\partial \nu^{s}} \Big\vert_{-} (x) =  \frac{1}{\rho_{s}}  \frac{\partial (S^{\kappa_{s}}_{D_{s}} \psi_{s})}{\partial \nu^{s}} \Big\vert_{-} (x),
\end{align*}
and therefore from the fourth identity in \eqref{model4}, we derive that
\begin{align*}
0 &= \frac{1}{\rho_{0}} \left.\frac{\partial u}{\partial \nu^{s}} \right\vert_{+}(x) -\frac{1}{\rho_{s}} \left.\frac{\partial u}{\partial \nu^{s}} \right\vert_{-}(x) 
=\frac{1}{\rho_{0}} \left.\frac{\partial u^{I}}{\partial \nu^{s}} \right\vert_{\partial D_{s}}(x)+\left.\sum_{{l=1} \atop {l\neq s}}^{M} \frac{1}{\rho_{0}} \frac{\partial(S^{\kappa_{0}}_{D_{l}} \phi_{l})}{\partial \nu^{s}}  \right\vert_{\partial D_{s}} (x)\\
&\quad \qquad \qquad \qquad \qquad \qquad \qquad -\frac{1}{2 \rho_{0}}  \phi_{s}(x)+\frac{1}{\rho_{0}} (K^{\kappa_{0}}_{D_{s}})^{*} \phi_{s}(x)-\frac{1}{2 \rho_{s}} \psi_{s}(x)-\frac{1}{\rho_{s}} (K^{\kappa_{s}}_{D_{s}})^{*} \psi_{s}(x)
\end{align*}
and hence
\begin{align}\label{jump2}
\frac{1}{\rho_{0}} \left.\frac{\partial u^{I}}{\partial \nu^{s}} \right\vert_{\partial D_{s}}(x)= & -\frac{1}{\rho_{0}} \Big[-\frac{1}{2} Id+(K^{\kappa_{0}}_{D_{s}})^{*} \Big] \phi_{s}(x)+\frac{1}{\rho_{s}} \Big[\frac{1}{2} Id +(K^{\kappa_{s}}_{D_{s}})^{*} \Big]\psi_{s}(x) -\left.\sum_{{l=1}\atop {l\neq s}}^{M} \frac{1}{\rho_{0}}\frac{\partial(S^{\kappa_{0}}_{D_{l}} \phi_{l})}{\partial \nu^{s}} \right\vert_{\partial D_{s}} (x).
\end{align}
The following proposition guarantees the existence of densities $\phi_{l}, \psi_{l}$ satisfying \eqref{jump1} and \eqref{jump2}.
\begin{proposition}\label{density}
For each \[ \prod_{i=1}^{M} (F_{i},G_{i}) \in Y:=\prod_{i=1}^{M}[H^{1}(\partial D_{i}) \times L^{2}(\partial D_{i})], \] 
there exists a unique solution \[ \prod_{i=1}^{M} (\phi_{i},\psi_{i}) \in X:=\prod_{i=1}^{M} [L^{2}(\partial D_{i}) \times L^{2}(\partial D_{i})] \] to the system of integral equations
 \begin{align*}
 & \left.\Big(S^{\kappa_{i}}_{D_{i}} \psi_{i}-\sum_{l=1}^{M}S^{\kappa_{0}}_{D_{l}} \phi_{l} \Big) \right\vert_{\partial D_{i}} = F_{i}, \\
 & \frac{\rho_0}{\rho_{i}} \Big[\frac{1}{2} Id +(K^{\kappa_{i}}_{D_{i}})^{*} \Big]\psi_{i}(x)- 
 \Big[-\frac{1}{2} Id+(K^{\kappa_{0}}_{D_{i}})^{*} \Big] \phi_{i}(x) -\left.\sum_{{l=1}\atop {l\neq i}}^{M} \frac{\partial(S^{\kappa_{0}}_{D_{l}} \phi_{l})}{\partial \nu^{i}} \right\vert_{\partial D_{i}} =G_{i}.
 \end{align*}
\end{proposition}
\begin{proof}
We outline a proof of this proposition in section \ref{proof-density}.
\end{proof}
Note that in order to prove the representation \eqref{sol1}, we apply the above proposition with 
$ F_{i}:=u^{I}\vert_{\partial D_{i}}$ and $G_{i}:=\left.\frac{\partial u^{I}}{\partial \nu^{i}} \right\vert_{\partial D_{i}} $ respectively, where $i=1,\dots,M $. 
\subsection{Integral identities}
Let us recall that the capacitance $Cap_{l}$ is defined as  
\begin{equation}\label{capacitance}
 Cap_l:=\int_{\partial D_l}\left(S_{D_l}^0\right)^{-1} (1)(t)\ d\sigma_{l}(t).
\end{equation}
It is known that $Cap_m \approx \delta\;$ and hence $Cap_m \approx a$.\\
We note that the system of integral equations  \eqref{jump1}-\eqref{jump2} can be rewritten as follows. 
For $l=1\dots,M$,
\begin{equation}
S^{\kappa_{0}}_{D_{l}} (\psi_{l}-\phi_{l})-\sum_{{m=1} \atop {m \neq l}}^{M} S^{\kappa_{0}}_{D_{m}} \phi_{m}= u^{I} + [S^{\kappa_{0}}_{D_{l}}-S^{\kappa_{l}}_{D_{l}}]\psi_{l},\ \mbox{ on } {\partial D_{l}}, 
\label{mult-obs-102r}
\end{equation}
%
\begin{align}\label{mult-obs-103r}
&\frac{\rho_{0}}{\rho_{l}} [\frac{1}{2} Id+(K^{\kappa_{0}}_{D_{l}})^{*}] \psi_{l}-[-\frac{1}{2} Id+(K^{\kappa_{0}}_{D_{l}})^{*}] \phi_{l} - \sum_{{m=1} \atop {m \neq l}}^{M} \frac{\partial (S^{\kappa_{0}}_{D_{m}} \phi_{m})}{\partial \nu^{l}} 
= \frac{\partial u^{I}}{\partial \nu^{l}} +\frac{\rho_{0}}{\rho_{l}} [(K^{\kappa_{0}}_{D_{l}})^{*}-(K^{\kappa_{l}}_{D_{l}})^{*}] \psi_{l},\  \mbox{ on } {\partial D_{l}}.
\end{align}
%
Our first step is to convert the system of double integral equations \eqref{mult-obs-102r}-\eqref{mult-obs-103r} for $(\phi_l,\psi_l)_{l=1}^{M} $ into a system of single integral equation for $(\phi_l)_{l=1}^{M} $. For this, we need to estimate first the source terms $[S^{\kappa_{0}}_{D_{l}}-S^{\kappa_{l}}_{D_{l}}]\psi_{l} $ and $[(K^{\kappa_{0}}_{D_{l}})^{*}-(K^{\kappa_{l}}_{D_{l}})^{*}] \psi_{l} $.\\
In the following two lemmas, we collect some approximations that we shall use.
\begin{lemma}\label{scaling-1}
The functions $\left(S_{D_l}^0\right)^{-1}\Bigg( \int_{\partial D_l}|\cdot-t|^n \phi_{l}(t)\ d\sigma_{l}(t)\Bigg) $ and $\left(S_{D_l}^0\right)^{-1}\Big((\cdot-z_l)^n\Big)$ exhibit the following approximate behavior.
\begin{align}\label{Sinv_int_difpts_psi}
&\left\|\left(S_{D_l}^0\right)^{-1}\Bigg( \int_{\partial D_l}|\cdot-t|^n \phi_{l}(t) d\sigma_{l}(t)\Bigg)\right\|_{L^2(\partial D_l)}=O\Big(a^{n+1}\|{\phi}_l\|_{L^2(\partial D_l)}\Big),
\end{align}
\begin{align}\label{InvSsmzl--ori}
&\Big\|\left(S_{D_l}^0\right)^{-1}\Big((\cdot-z_l)^n\Big)\Big\|_{L^2(\partial D_l)}=O(a^n).
\end{align}
\end{lemma}
\begin{proof}
We refer to section \ref{proof-scaling-1} for a proof of this lemma.
\end{proof}
\begin{lemma}\label{Claimest} The following approximations hold true.
 %
 \begin{align}  \label{claim-single-Lay-dif-2}
 \left[S^{\kappa_{0}}_{D_{m}}-S^{\kappa_{m}}_{D_{m}}\right]\psi_{m}(x)=&\frac{i}{4\pi}(\kappa_0-\kappa_m)\int_{\partial D_m}\psi_m(s)\,d\sigma_{m}(s)+
 \underbrace{\sum_{n=2}^\infty\frac{i^n(\kappa_0^n-\kappa_m^n)}{4\pi n!}\int_{\partial D_m}\vert{x-s}\vert^{n-1}\psi_m(s)\,d\sigma_{m}(s)}_{=: Err1_{m}=O\left(a^2\|\psi_m\|\right)},
 \end{align}
\begin{align} \label{claim-adjdblle-Lay-dif-prime}
\int_{\partial D_m}\left[(K^{\kappa_{0}}_{D_{m}})^{*}-(K^{\kappa_{m}}_{D_{m}})^{*}\right]\psi_m(s) d\sigma_{m}(s)
=&\frac{1}{8\pi}(\kappa_0^2-\kappa_m^2)\int_{\partial D_m}\psi_m(s)\left[\int_{ \partial D_m}\frac{(s-t)}{\vert{s-t}\vert} \cdot\nu_t \,dt\right]d\sigma_{m}(s) \\
&\qquad \qquad -\frac{i}{4\pi}(\kappa_0^3-\kappa_m^3)\vert{D_m}\vert\int_{\partial D_m}\psi_m(s)d\sigma_{m}(s) 
+ \underbrace{Err2_{m}}_{[=O\left(a^5\|\psi_m\|\right)]},\notag
\end{align}
%

 %
 %
 \begin{align}  \label{claim-normsingle-Lay}
 \int_{\partial D_l}\frac{\partial (S^{\kappa_{0}}_{D_{m}} \phi_{m})}{\partial \nu^{l}}
=&-\kappa_0^2\vert{ D_l}\vert\left[\Phi_{\kappa_0}(z_l,z_m)\int_{\partial D_m}\phi_m(s)\,d\sigma_{m}(s)\right.\\
&\left.+\nabla_t\Phi_{\kappa_0}(z_l,z_m)\cdot\int_{\partial D_m}(s-z_m)\phi_m(s)\,d\sigma_{m}(s)\right]\notag\\
&-\kappa_{0}^{2}\nabla_x\Phi_{\kappa_{0}}(z_{l},z_m)\cdot\Big[ \int_{D_{l}}(x-z_l) dx\Big] \int_{\partial D_{m}}   \phi_{m}(s) \,d\sigma_{m}(s)-\underbrace{Err4_{m}}_{=O\left(\frac{a^6}{d^3_{ml}}\|\phi_m\|\right)},\, m\neq\;l,\notag
\end{align}
%

\begin{align}\label{claim-norm-inci}
\int_{\partial D_{l}} \frac{\partial u^{I}}{\partial \nu^{l}} 
=&-\kappa_0^2\,\vert{D_l}\vert\,u^{I}(z_l)+\underbrace{Err5_{l}}_{=O\left(a^4\right)},
\end{align}
%
\begin{equation}
\begin{split}
\psi_l-\phi_l=&\left(S^{{0}}_{D_{l}}\right)^{-1}\left(\frac{i(\kappa_0-\kappa_l)}{4\pi}Q_l\right)\\
&+\left(S^{{0}}_{D_{l}}\right)^{-1}\left(\sum_{m\neq\;l}\left[\left(1-\frac{i\kappa_l}{4\pi}Cap_l\right) \Phi_{\kappa_0}(z_l,z_m)+(s-z_l)\cdot\nabla_s\Phi_{\kappa_0}(z_l,z_m)\right]Q_m\right)\\
              &  +\left(S^{{0}}_{D_{l}}\right)^{-1}\left( \sum_{m\neq\;l}\nabla_t\Phi_{\kappa_0}(z_l,z_m)\cdot\,V_m\right)+\left(S^{{0}}_{D_{l}}\right)^{-1} u^{I}(z_l)\\
             &+Err6_{l} \, \Big[:=O\left(a+a^2\|\phi_l\|+\sum\limits_{m\neq\,l}\frac{a^3}{d^3}\|\phi_m\|\right)\Big],\mbox{ in }L^2,\label{claim- diff-phpsi}
\end{split}
\end{equation}
where $V_m:=\int_{\partial D_m}(s-z_m)\phi_{m}\;d\sigma_{m}(s) $.
\begin{equation}
\begin{split}             
&\int_{{\partial D_{l}}}\psi_l -\int_{{\partial D_{l}}}\phi_l=Cap_l\left(\frac{i(\kappa_0-\kappa_l)}{4\pi}Q_l\right)\\
&\qquad +\left(\sum_{m\neq\;l}\left[Cap_l \left(1-\frac{i\kappa_l}{4\pi}Cap_l\right)\Phi_{\kappa_0}(z_l,z_m)+\nabla_s\Phi_{\kappa_0}(z_l,z_m)\cdot\int_{\partial D_l}\left(S_{D_l}^0\right)^{-1}(\cdot-z_l)(s) d\sigma_{l}(s)\right]Q_m\right)\\
              &\qquad  +Cap_l\left( \sum_{m\neq\;l}\nabla_t\Phi_{\kappa_0}(z_l,z_m)\cdot\,V_m\right)+Cap_l\ u^{I}(z_l)\\
&\qquad +Err7_{l}\, \Big[:=O\left(a^2+a^3\|\phi_l\|+\sum\limits_{m\neq\,l}\frac{a^4}{d^3}\|\phi_m\|\right)\Big],\label{claim-int- diff-phpsi}
\end{split}
\end{equation}
%
\end{lemma}
\begin{proof}
We defer the proof of the lemma to section \ref{Proof-Claimest}.
\end{proof}
We shall also sometimes use \eqref{claim-adjdblle-Lay-dif-prime} with $\kappa_{m}=0 $. In this case, we shall refer to $Err2_{m} $ as $Err3_{m} $.

\subsection{A way of counting the number of bubbles}

In the following sections, we will repeatedly need to estimate sums of the form $\sum^M_{i=1, i\neq j}f(z_i, z_j)$ with functions $f$ involving inverse power of distances, i.e. $\vert z_i-z_j\vert^{-k},\; k \in \mathbb{R}_+$. 
Here $z_j$ is the center of the small bubble $D_j$.
Following \cite{Challa-Sini-2}, we describe here a way how to handle these sums by a proper counting. To do it, for any $m=1,\dots,M,$ fixed, we distinguish between the points $z_j$, $j\neq\,m,$ by keeping them 
into different layers based on their distance from $D_m$.
\bigskip

Let $\Omega$ be a bounded domain, say of unit volume. The way how the bubbles $D_m, m=1, ..., M$, are distributed in $\Omega$ can be described as follows. Recall that $M=O(a^{-s})$.  
We shall divide $\Omega$ into $[a^{-s}]$ cubes $\Omega_m,\; m=1, ..., [a^{-s}]$,\footnote{For a given real and positive number $x$, we denote by $[x]$, the unique integer $n$ 
 such that $n\leq x \leq n+1$.} such that each 
$\Omega_m$ contains some of the $D_j$'s (and maybe non). It is natural then to assume that the number of bubbles in each $\Omega_m$, 
for $m=1, ..., [a^{-s}]$, is uniformly bounded in terms of $m$. 

 We take the cubes $\Omega_m$'s to be equal, up to translations, and distributed periodically in $\Omega$, i.e. the set of $\Omega_m$'s is a uniform grid of $\Omega$.
 Observe that even if the $\Omega_m$'s are periodically distributed, 
 the bubbles are still not periodically distributed as in each $\Omega_m$ we can have a different number of bubbles. Let us emphasize that any distribution of the bubbles can be
 structured with a periodic set of $\Omega_m$' as above. The case when the number of bubbles in each $\Omega_m$ is exactly one, 
 i.e. the distribution of the bubbles is periodic, is the situation dealt with in the homogenization theory.  
 
 \bigskip

For each $m$, $1\leq\,m\leq\,[a^{-s}],$ $\Omega_m$ is the cube whose sides have sizes which can be estimated as $(\frac{a}{2}+d^\alpha)$ with $0\leq\alpha\leq{1}$. 
Observe that $\alpha=1$ means that the corresponding $\Omega_m$ contains at most one bubble $D_m$. Recall that $d$ is of the order $a^t$, then we have the relation $3\alpha t=s$, see below.  
\bigskip

To estimate the aforementioned sums, we need to count the number of the bubbles. The idea is as follows. As the number of bubbles in each cube $\Omega_m$ is uniformly bounded in terms of $m$, 
we reduce the question to counting the $\Omega_m$'s. To this end, we use the periodic structure of the cubes $\Omega_m$'s. 
Precisely, for a given $m$, we take the cube $\Omega_m$ as the 'starting cube' and then the other cubes are attached to it
as to form different layers of cubes (i.e. as a Rubik cube with $\Omega_m$ in the center). Hence, starting from $\Omega_m$, the total cubes upto the $n^{th}$
layer consists of $(2n+1)^3$ cubes for $n=0,\dots,[d^{-\alpha}]$. It is clear then that the number of bubbles 
located in the $n^{th}$, $n\neq0$ layer  will be at most $[(2n+1)^3-(2n-1)^3]$ and their distance from $D_j$ is more than ${n}d^\alpha$.
\bigskip

With this way of counting, we deduce that for $j$ fixed
\begin{equation}
 \sum^M_{i=1, i\neq j}\vert z_i-z_j\vert^{-k}\leq \sum_{z_j,\; z_m \in \Omega_j, z_m\neq z_j}\vert z_m-z_j\vert^{-k} + \sum^{d^{-\alpha}}_{{l=1}\atop{l\neq j}}\sum_{z_i\in \Omega_l}\vert z_i-z_j\vert^{-k}
\end{equation}
where
$$
\sum_{z_j,\; z_m \in \Omega_j, z_m\neq z_j}\vert z_m-z_j\vert^{-k} =O(d^{- k })
$$
and 
$$
\sum^{d^{-\alpha}}_{{l=1}\atop {l \neq j}}\sum_{z_i\in \Omega_l}\vert z_i-z_j\vert^{-k}\leq \sum^{d^{-\alpha}}_{l=1}((2l+1)^3-(2l-1)^3)(ld^\alpha)^{-k} =O(\sum^{d^{-\alpha}}_{l=1} l^2 (ld^\alpha)^{-k})=O(d^{-\alpha\; k} \sum^{d^{-\alpha}}_{l=1} l^{2-k}). 
$$
Hence
\begin{equation}\label{way-counting}
\sum^M_{i=1, i\neq j}\vert z_i-z_j\vert^{-k}=O(d^{-k })+O(d^{-\alpha\; k} \sum^{d^{-\alpha}}_{l=1} l^{2-k}).
\end{equation}
Observe that for $k=0$, it is obvious that $\sum^M_{i=1, i\neq j}\vert z_i-z_j\vert^{-k}=M-1$. By the previous formulas we have $\sum^M_{i=1, i\neq j}\vert z_i-z_j\vert^{-k}=O(1) +O(d^{-3\alpha})$. Recalling the way we counted, 
$d^{3\alpha}$ is the volume of the $\Omega_j$'. Hence $d^{-3\alpha}$ is of the order of the number of the bubbles, i.e. $d^{-3\alpha}=O(M)$. As we set $M=O(a^{-s})$, we will be using the formula $3\alpha t =s$, 
and then $t\geq \frac{s}{3}$, as $\alpha \leq 1$.
\bigskip

For $k>0$,  we obtain the following formulas:
\begin{enumerate}
 \item If $k < 3$, then 
 \begin{equation}\label{way-counting-1}
\sum^M_{i=1, i\neq j}\vert z_i-z_j\vert^{-k}=O(d^{-k })+O(d^{-\alpha\; k} \sum^{d^{-\alpha}}_{l=1} l^{2-k}) =O(d^{-k}) +O(d^{-3\alpha}).
\end{equation}
\item If $k=3$, then
\begin{equation}\label{way-counting-2}
\sum^M_{i=1, i\neq j}\vert z_i-z_j\vert^{-k}=O(d^{-k })+O(d^{-3\alpha\; } \vert \ln(d) \vert) .
\end{equation}
\item If $k>3$, then 
 \begin{equation}\label{way-counting-3}
\sum^M_{i=1, i\neq j}\vert z_i-z_j\vert^{-k}=O(d^{-k })+O(d^{-\alpha\; k} \sum^{d^{-\alpha}}_{l=1} l^{2-k}) =O(d^{-k}) +O(d^{-\alpha k}).
\end{equation}

\end{enumerate}
If we do not count properly, then we would have $\sum^M_{i=1, i\neq j}\vert z_i-z_j\vert^{-k}=O(M\;d^{-k})=a^{-s -k\; t}$. Let us take $k=1$ as an example. We have shown above that 
$\sum^M_{i=1, i\neq j}\vert z_i-z_j\vert^{-k}=O(a^{-t})+O(d^{-3\alpha})=O(a^{-t})+O(a^{-3\alpha t})=O(a^{-t})+O(a^{-s})$ instead of $O(a^{-s-t})$. Hence, in this case, we gain an order of $a^{-t}$ as soon as $s\geq t$.
This kind of reasoning will be used in the next sections.


%
\subsection{The a priori approximation of the system of integral equations}
Our first step is to approximate the source term $[S^{\kappa_{0}}_{D_{l}}-S^{\kappa_{l}}_{D_{l}}]\psi_{l} $ in \eqref{mult-obs-102r}, that is,
$S^{\kappa_{0}}_{D_{l}} (\psi_{l}-\phi_{l})-\sum_{{m=1} \atop {m \neq l}}^{M} S^{\kappa_{0}}_{D_{m}} \phi_{m}= u^{I} + [S^{\kappa_{0}}_{D_{l}}-S^{\kappa_{l}}_{D_{l}}]\psi_{l}$
by single layers of the form $S^{\kappa_{0}}_{D_{l}} g_{l} $, where $g_l$ is the unique solution of the problem
\begin{equation}\label{def-Sg}
S^{\kappa_{0}}_{D_{l}} g_l(s)= \frac{i}{4\pi}(\kappa_0-\kappa_l)\int_{\partial D_l}\psi_l(t) d\sigma_{l}(t), \mbox { on } \partial D_l.
\end{equation}
Hence \[[S^{\kappa_{0}}_{D_{l}}-S^{\kappa_{l}}_{D_{l}}]\psi_{l}=S^{\kappa_{0}}_{D_{l}} g_{l}+Err1_{l}. \]
Let us define
\[ H1:= S^{\kappa_{0}}_{D_{l}} (\psi_{l}-\phi_{l}) - \sum_{{m=1} \atop {m \neq l}}^{M} S^{\kappa_{0}}_{D_{m}}\phi_{m}- u^{I} - S^{\kappa_{0}}_{D_{l}} g_l, \] and $\tilde{g}_l$ such that
\begin{equation}\label{def-SGer}
S^{\kappa_{0}}_{D_{l}} \tilde{g}_l = Err1_{l}.
\end{equation}
Using \eqref{def-Sg} and \eqref{def-SGer} in \eqref{claim-single-Lay-dif-2}, it follows that $H1$ satisfies 
\begin{equation}
\begin{cases}
(\Delta+\kappa_{0}^{2}) (H1-S^{\kappa_{0}}_{D_{l}}\tilde{g}_l)=0 \ \text{in}\ D_{l},\\
H1=S^{\kappa_{0}}_{D_{l}}\tilde{g}_l \ \text{on}\ \partial D_{l}.
\end{cases}
\notag
\end{equation}
By the maximum principle, we can immediately conclude that $H1= S^{\kappa_{0}}_{D_{l}}\tilde{g}_l$ in $ D_{l} $.
Taking the normal derivative from inside $D_{l}$, we have
\begin{align}\notag
\frac{\partial [S^{\kappa_{0}}_{D_{l}}(\psi_{l}-\phi_{l})] }{\partial \nu^{l}} \Big\vert_{\partial D_{l}}  - \sum_{{m=1} \atop {m \neq l}}^{M} \frac{\partial (S^{\kappa_{0}}_{D_{m}}\phi_{m})}{\partial \nu^{l}} \Big\vert_{\partial D_l} 
  =& \frac{\partial u^{I}}{\partial \nu^{l}}\Big\vert_{\partial D_{l}}+\frac{\partial (S^{\kappa_{0}}_{D_{l}}g_l)}{\partial \nu^{l}} \Big\vert_{\partial D_{l}}+Er_{l},
\end{align}
where $Er_{l}:= \frac{\partial [S^{\kappa_{0}}_{D_{l}}\tilde{g}_{l}]}{\partial \nu^{l}} \Big\vert_{\partial D_{l}}$. This gives us
\begin{align}\notag
 \frac{\rho_{0}}{\rho_{l}} \Big[\frac{1}{2} Id+(K^{\kappa_{0}}_{D_{l}})^{*} \Big] (\psi_{l}-\phi_{l})& - \sum_{{m=1} \atop {m \neq l}}^{M} \frac{\rho_{0}}{\rho_{l}} \frac{\partial (S^{\kappa_{0}}_{D_{m}}\phi_{m})}{\partial \nu^{l}} \Big\vert_{\partial D_{l}}
=\frac{\rho_{0}}{\rho_{l}} \frac{\partial u^{I}}{\partial \nu^{l}}\Big\vert_{\partial D_{l}}+\frac{\rho_{0}}{\rho_{l}}\Big[\frac{1}{2} Id+(K^{\kappa_{0}}_{D_{l}})^{*} \Big] g_l +Er_{2,l},
\end{align}
where $Er_{2,l}:=\frac{\rho_{0}}{\rho_{l}} Er_{l} $. Next we use this in the second identity \eqref{mult-obs-103r} to derive
\begin{equation}
\begin{split}
 -\frac{1}{2} \Big[\frac{\rho_{0}}{\rho_{l}} +1 \Big]& \phi_{l} - \Big[\frac{\rho_{0}}{\rho_{l}}-1 \Big] (K^{\kappa_{0}}_{D_{l}})^{*} \phi_{l}- \sum_{{m=1} \atop {m \neq l}}^{M} \Big(\frac{\rho_{0}}{\rho_{l}}-1 \Big) \frac{\partial (S^{\kappa_{0}}_{D_{m}}\phi_{m})}{\partial \nu^{l}}\Big\vert_{\partial D_{l}}  \\
&= \Big(\frac{\rho_{0}}{\rho_{l}}-1 \Big) \frac{\partial u^{I}}{\partial \nu^{l}} \Big\vert_{\partial D_{l}} +\frac{\rho_{0}}{\rho_{l}}\Big[\frac{1}{2} Id+(K^{\kappa_{0}}_{D_{l}})^{*} \Big] g_l-\frac{\rho_{0}}{\rho_{l}} [(K^{\kappa_{0}}_{D_{l}})^{*}-(K^{\kappa_{l}}_{D_{l}})^{*}] \psi_{l} +Er_{2,l},\\
\end{split}
\notag
\end{equation}
which, in turn, implies that
\begin{equation} 
\begin{split}
& -\frac{1}{2} \Big[\frac{\rho_{0}+\rho_{l}}{\rho_{0}-\rho_{l}}  \Big] \phi_{l} - (K^{\kappa_{0}}_{D_{l}})^{*} \phi_{l}- \sum_{{m=1} \atop {m \neq l}}^{M} \frac{\partial (S^{\kappa_{0}}_{D_{m}}\phi_{m})}{\partial \nu^{l}}\Big\vert_{\partial D_{l}} \\
&\qquad =  \frac{\partial u^{I}}{\partial \nu^{l}} + \Big(1-\frac{\rho_{l}}{\rho_{0}} \Big)^{-1}\Bigg([(K^{\kappa_{l}}_{D_{l}})^{*}-(K^{\kappa_{0}}_{D_{l}})^{*}] \psi_{l} +[\frac{1}{2}Id+(K^{\kappa_{0}}_{D_{l}})^{*}] g_{l} 
 \Bigg)+Er_{3,l},\mbox{  on } \partial D_l,
\end{split}
\notag
\end{equation}
where $Er_{3,l}:=\Big(\frac{\rho_{0}}{\rho_{l}}-1 \Big)^{-1} Er_{2,l} $.
Hence, we have
\begin{equation}
\begin{split}
 -\frac{1}{2} \Big[\frac{\rho_{0}+\rho_{l}}{\rho_{0}-\rho_{l}}  \Big] \phi_{l} - (K^{{0}}_{D_{l}})^{*} \phi_{l}
&=  \frac{\partial u^{I}}{\partial \nu^{l}} + \Big(1-\frac{\rho_{l}}{\rho_{0}} \Big)^{-1}\Bigg([(K^{\kappa_{l}}_{D_{l}})^{*}-(K^{\kappa_{0}}_{D_{l}})^{*}] \psi_{l} +[\frac{1}{2}Id+(K^{\kappa_{0}}_{D_{l}})^{*}] g_{l} 
 \Bigg)\\
&\qquad+\sum_{{m=1} \atop {m \neq l}}^{M} \frac{\partial (S^{\kappa_{0}}_{D_{m}}\phi_{m})}{\partial \nu^{l}}\Big\vert_{\partial D_{l}} +[(K^{{\kappa_0}}_{D_{l}})^{*} - (K^{{0}}_{D_{l}})^{*}] \phi_{l}+Er_{3,l},\ \mbox{  on } \partial D_l.
\end{split}
\label{mult-obs-108r-int}
\end{equation}
Our aim henceforth would be to derive a suitable approximate system for $\int_{\partial D_{l}} \phi_{l} \ d\sigma_{l} $ from \eqref{mult-obs-108r-int} above, leading us to the final algebraic system. To do so, we start by integrating \eqref{mult-obs-108r-int} on $\partial D_l$, and since $K^{{0}}_{D_{l}}(1)=-\frac{1}{2}$, it follows that
\begin{equation}
\begin{split}
-\frac{1}{2} \Big[\frac{\rho_{0}+\rho_{l}}{\rho_{0}-\rho_{l}}-1 \Big] \int_{\partial D_{l}} \phi_{l} \ d\sigma_{l} &=\int_{\partial D_{l}} \frac{\partial u^{I}}{\partial \nu^{l}} \ d\sigma_{l}
+ \int_{\partial D_{l}} \Big[(K^{\kappa_{0}}_{D_{l}})^{*}-(K^{0}_{D_{l}})^{*} \Big] \phi_{l} \ d\sigma_{l}
+\sum_{{m=1} \atop {m \neq l}}^{M} \int_{\partial D_{l}} \frac{\partial (S^{\kappa_{0}}_{D_{m}}\phi_{m})}{\partial \nu^{l}} \ d\sigma_{l}\\
&+ \Big(1-\frac{\rho_{l}}{\rho_{0}} \Big)^{-1}\int_{\partial D_l}\Bigg([(K^{\kappa_{l}}_{D_{l}})^{*}-(K^{\kappa_{0}}_{D_{l}})^{*}] \psi_{l} +[\frac{1}{2}Id+(K^{\kappa_{0}}_{D_{l}})^{*}] g_{l}
 \Bigg)\ d\sigma_{l}+Er_{4,l},
\end{split}
\label{mult-obs-109r}
\end{equation}
where $Er_{4,l}:=\int_{\partial D_{l}} Er_{3,l} $. For $s\in\partial D_l$, let us define \[A_l(s):=\int_{ \partial D_l}\frac{(s-t)}{\vert{s-t}\vert} \cdot\nu_t \,d\sigma_{l}(t).\]
From the definition, it easily follows that $\norm{A_{l}}_{L^{2}(\partial D_{l})}=O(a^3) $.
Next we observe that,
\begin{equation} 
\begin{split}
       Er_{4,l}&=\left(1-\frac{\rho_l}{\rho_0}\right)^{-1}\int_{\partial D_l} Er_{l}  
           =\left(1-\frac{\rho_l}{\rho_0}\right)^{-1}\int_{\partial D_l} [\frac{1}{2}Id+(K^{\kappa_{0}}_{D_{l}})^{*}]\tilde{g}_{l} \\
           &=\left(1-\frac{\rho_l}{\rho_0}\right)^{-1}\int_{\partial D_l} [-(K^{{0}}_{D_{l}})^{*}+(K^{\kappa_{0}}_{D_{l}})^{*}] \tilde{g}_{l}\\
     &=\left(1-\frac{\rho_l}{\rho_0}\right)^{-1} \frac{1}{8\pi}\kappa_0^2 \int_{\partial D_l} \tilde{g}_{l}(s)A_l(s) d\sigma_{l}(s)+O(a^{4} \norm{\tilde{g}_{l}}), \\
\end{split}
\notag
\end{equation}
where the last step follows from \eqref{claim-adjdblle-Lay-dif-prime} applied to $\tilde{g}_{l} $.\\
In order to check the behavior of the $Er_{4,l}$, let us consider its dominating term $\int_{\partial D_l} \tilde{g}_{l}(s)A_l(s) d\sigma_{l}(s)$. From the definition of $\tilde{g}_l$, it can be observed that the dominating term of $\tilde{g}_l$ is $\left(S_{ D_l}^{\kappa_0}\right)^{-1}\left(\int_{\partial D_l}|\cdot-t|\psi_l(t) d\sigma_{l}(t)\right)$ and hence it's sufficient to consider $\int_{\partial D_l}\left(S_{ D_l}^{0}\right)^{-1}\left(\int_{\partial D_l}|\cdot-t|\psi_l(t) d\sigma_{l}(t)\right) (s)\, A_l(s) d\sigma_{l}(s)$. Using \eqref{Sinv_int_difpts_psi}, we can deduce that
\begin{equation}
\begin{split}
&\left\vert\int_{\partial D_l}\left(S_{ D_l}^{0}\right)^{-1}\left(\int_{\partial D_l}|\cdot-t|\psi_l(t) d\sigma_{l}(t) \right) (s)\, A_l(s)\ d\sigma_{l}(s) \right\vert \\
&\leq \left\|\left(S_{ D_l}^{0}\right)^{-1}\left(\int_{\partial D_l}|\cdot-t|\psi_l(t) d\sigma_{l}(t)\right) \right\|_{{L}^2(\partial D_l)}\left\|A_l\right\|_{{L}^2(\partial D_l)}
{=}O\left(a^5\|\psi_l\|\right).
\end{split}
\label{Er4-estimatecomp-1}
\end{equation}
In addition, $\norm{\tilde{g}_{l}} \leq \norm{(S^{\kappa_{0}}_{D_{l}})^{-1}} \norm{Err1_{l}}_{H^1(\partial D_{l})} =O(a \norm{\psi_{l}}) $. Therefore $Er_{4,l}=O\left(a^5\|\psi_l\|\right) $.\\
In the following lemma, we note down some estimates for $g_{l} $ that we shall use.
\begin{lemma}\label{claim-g}
Similar to \eqref{claim-adjdblle-Lay-dif-prime}, $g_{l}$ satisfies the estimate
 \begin{align}   \label{claim-adjdblle-Lay-dif-g}
 \int_{\partial D_l}\left[(K^{\kappa_{0}}_{D_{l}})^{*}-(K^{0}_{D_{l}})^{*}\right]g_l
  =&\frac{1}{8\pi}\kappa_0^2\int_{\partial D_l}g_l(s)\left[\int_{ \partial D_l}\frac{(s-t)}{\vert{s-t}\vert} \cdot\nu_t \,dt\right]d\sigma_{l}(s)\\ \nonumber
  &\qquad -\frac{i}{4\pi}\kappa_0^3\vert{D_l}\vert\int_{\partial D_l}g_l(s)d\sigma_{l}(s)+\underbrace{Err8_{l}}_{=O\left(a^5\|g_l\|\right)}.
  \end{align}
 %
 Also
\begin{align}\label{def-Sg-c}
 g_l (s)  
 =&\frac{i}{4\pi}(\kappa_0-\kappa_l)\left[Cap_l\ u^{I}(z_l)+\int_{\partial D_l}\phi_l+\sum_{m\neq\,l}\Phi_{\kappa_0}(z_l,z_m)\,Q_m Cap_l\right](S^{0}_{D_{l}})^{-1}\left(1\right)(s)\\ \notag
        &+Err9_{l}\Big[:=O\left(a^2+a^2\|\phi_l\|+\sum_{m\neq\,l}\frac{a^3}{d^2_{ml}}\|\phi_m\|\right)\Big], \mbox{ in } L^2. 
\end{align}
\end{lemma}
\begin{proof}
We defer the proof of this result to section \ref{proof-claim-g}.
\end{proof}
Now using lemma \ref{Claimest}, \eqref{mult-obs-109r} can be rewritten as
\begin{align}\label{mult-obs-109r1}
&\frac{\rho_{l}}{\rho_{l}-\rho_{0}} \int_{\partial D_{l}} \phi_{l}\ d\sigma_{l}(s)+\sum_{{m=1} \atop {m \neq l}}^{M}\Big(\kappa_0^2\vert{ D_l}\vert\Big[\Phi_{\kappa_0}(z_l,z_m)\int_{\partial D_m}\phi_m(s)\,d\sigma_{m}(s)\\ \notag
&\qquad +\nabla_t\Phi_{\kappa_0}(z_l,z_m)\cdot\int_{\partial D_m}(s-z_m)\phi_m(s)\,d\sigma_{m}(s)\Big] \\ \notag
&\qquad+\kappa_{0}^{2}\nabla_x\Phi_{\kappa_{0}}(z_{l},z_m)\cdot\Big[ \int_{D_{l}}(x-z_l) dx\Big] \int_{\partial D_{m}}   \phi_{m}(s) \,d\sigma_{m}(s)+Err4_{m}\Big)\\ \notag
=&-\kappa_0^2\,\vert{D_l}\vert\,u^{I}(z_l)+Err5_{l}+\frac{1}{8\pi}\kappa_0^2\int_{\partial D_l}\phi_l(s)\left[\int_{\partial D_l}\frac{(s-t)}{\vert{s-t}\vert} \cdot\nu_t \,d\sigma_{l}(t)\right]d\sigma_{l}(s)\\ \notag
&-\frac{i}{4\pi}\kappa_0^3\vert{D_l}\vert\int_{\partial D_l}\phi_l(s)d\sigma_{l}(s)+Err3_{l} \\ \notag
&+\Big(1-\frac{\rho_{l}}{\rho_{0}} \Big)^{-1}\frac{1}{8\pi}(\kappa_l^2-\kappa_0^2)\int_{\partial D_l}\psi_l(s)\left[\int_{ \partial D_l}\frac{(s-t)}{\vert{s-t}\vert} \cdot\nu_t \,d\sigma_{l}(t)\right]d\sigma_{l}(s)\\ \notag
&-\Big(1-\frac{\rho_{l}}{\rho_{0}} \Big)^{-1}\frac{i}{4\pi}(\kappa_l^3-\kappa_0^3)\vert{D_l}\vert\int_{\partial D_l}\psi_l(s)d\sigma_{l}(s)-\Big(1-\frac{\rho_{l}}{\rho_{0}} \Big)^{-1}Err2_{l} \\ \notag
&+\Big(1-\frac{\rho_{l}}{\rho_{0}} \Big)^{-1}\frac{1}{8\pi}\kappa_0^2\int_{\partial D_l}g_l(s)\left[\int_{ \partial D_l}\frac{(s-t)}{\vert{s-t}\vert} \cdot\nu_t \,d\sigma_{l}(t)\right]d\sigma_{l}(s)\\ \notag
&-\Big(1-\frac{\rho_{l}}{\rho_{0}} \Big)^{-1}\frac{i}{4\pi}\kappa_0^3\vert{D_l}\vert\int_{\partial D_l}g_l(s)d\sigma_{l}(s)+\Big(1-\frac{\rho_{l}}{\rho_{0}} \Big)^{-1}Err8_{l} +Er_{4,l}.
\end{align}
%
Dividing by $\vert D_{l} \vert$ and making use of \eqref{claim- diff-phpsi},\eqref{claim-int- diff-phpsi} and \eqref{def-Sg-c} in \eqref{mult-obs-109r1}, we obtain 

\begin{align}\label{mult-obs-109r4}
&\frac{\rho_{l}}{\rho_{l}-\rho_{0}} \vert{ D_l}\vert^{-1} \int_{\partial D_{l}} \phi_{l}(s)\ d\sigma_{l}(s)\\ \notag
&+\kappa_0^2\sum_{{m=1} \atop {m \neq l}}^{M}\Big[\Phi_{\kappa_0}(z_l,z_m)\int_{\partial D_m}\phi_m(s)\,d\sigma_{m}(s) +\nabla_t\Phi_{\kappa_0}(z_l,z_m)\cdot\int_{\partial D_m}(s-z_m)\phi_m(s)\,d\sigma_{m}(s)\Big]\\ \notag
&+\kappa_{0}^{2}\vert{ D_l}\vert^{-1}\sum_{{m=1} \atop {m \neq l}}^{M}\nabla_x\Phi_{\kappa_{0}}(z_{l},z_m)\cdot\Big[ \int_{D_{l}}(x-z_l) dx\Big] \int_{\partial D_{m}}   \phi_{m}(s) \,d\sigma_{m}(s)\\ \notag
=&-\kappa_0^2\,u^{I}(z_l)+ \frac{1}{8\pi}\kappa_0^2\vert{ D_l}\vert^{-1}\int_{\partial D_l}\phi_l(s){A_l(s)}d\sigma_{l}(s)-\frac{i}{4\pi}\kappa_0^3\int_{\partial D_l}\phi_l(s)d\sigma_{l}(s)\\ \notag
&+\Big(1-\frac{\rho_{l}}{\rho_{0}} \Big)^{-1}\frac{1}{8\pi}(\kappa_l^2-\kappa_0^2)\vert{ D_l}\vert^{-1}\int_{\partial D_l}\phi_l(s){A_l(s)}d\sigma_{l}(s)
\\ \notag
&+\Big(1-\frac{\rho_{l}}{\rho_{0}} \Big)^{-1}\frac{1}{8\pi}(\kappa_l^2-\kappa_0^2)\vert{ D_l}\vert^{-1}\int_{\partial D_l}\left(S^{{0}}_{D_{l}}\right)^{-1}\Bigg(\frac{i(\kappa_0-\kappa_l)}{4\pi}Q_l\\ \notag
&\qquad \qquad +\sum_{m\neq\;l}\left[\left(1-\frac{i\kappa_l}{4\pi}Cap_l\right)\Phi_{\kappa_0}(z_l,z_m)+(\cdot-z_l)\cdot\nabla_s\Phi_{\kappa_0}(z_l,z_m)\right]Q_m \\ \notag
&\qquad \qquad+\sum_{m\neq\;l}\nabla_t\Phi_{\kappa_0}(z_l,z_m)\cdot\,V_m\Bigg)(s) {A_l(s)}d\sigma_{l}(s)
\\ \notag
&+\Big(1-\frac{\rho_{l}}{\rho_{0}} \Big)^{-1}\frac{1}{8\pi}(\kappa_l^2-\kappa_0^2)\vert{ D_l}\vert^{-1}\int_{\partial D_l}\left[\left(S^{{0}}_{D_{l}}\right)^{-1} u^{I}(z_l)+Err6_{l}\right]{A_l(s)}d\sigma_{l}(s)\\ \notag
&-\Big(1-\frac{\rho_{l}}{\rho_{0}} \Big)^{-1}\frac{i}{4\pi}(\kappa_l^3-\kappa_0^3)\int_{\partial D_l}\phi_l(s)d\sigma_{l}(s)\\ \notag
&-\Big(1-\frac{\rho_{l}}{\rho_{0}} \Big)^{-1}\frac{i}{4\pi}(\kappa_l^3-\kappa_0^3)
\Bigg[\frac{i(\kappa_0-\kappa_l)}{4\pi}Cap_l Q_l
+\sum_{m\neq\;l}\Big[Cap_l \left(1-\frac{i\kappa_l}{4\pi}Cap_l\right)\Phi_{\kappa_0}(z_l,z_m)\\ \notag
&\qquad \qquad \qquad +\nabla_s\Phi_{\kappa_0}(z_l,z_m)\cdot\int_{\partial D_l}\left(S_{D_l}^0\right)^{-1}(\cdot-z_l)(s) d\sigma_{l}(s)\Big]Q_m +Cap_l \sum_{m\neq\;l}\nabla_t\Phi_{\kappa_0}(z_l,z_m)\cdot\,V_m\Bigg]
\\ \notag
&-\Big(1-\frac{\rho_{l}}{\rho_{0}}\Big)^{-1}\frac{i}{4\pi}(\kappa_l^3-\kappa_0^3)\left[Cap_l u^{I}(z_l)+Err7_{l}\right]
\\ \notag
&+\Big(1-\frac{\rho_{l}}{\rho_{0}} \Big)^{-1}\frac{1}{8\pi}\kappa_0^2\vert{ D_l}\vert^{-1}\int_{\partial D_l} \frac{i}{4\pi}(\kappa_0-\kappa_l)\Big[Cap_l u^{I}(z_l)\\ \notag
&\qquad \qquad \qquad \qquad +Q_l+\sum_{m\neq\,l}\Phi_{\kappa_0}(z_l,z_m)\,Q_m Cap_l\Big](S^{0}_{D_{l}})^{-1}\left(1\right)(s){A_l(s)}d\sigma_{l}(s)\\ \notag
&+\Big(1-\frac{\rho_{l}}{\rho_{0}} \Big)^{-1}\frac{1}{8\pi}\kappa_0^2\vert{ D_l}\vert^{-1}\int_{\partial D_l}Err9_{l} \cdot {A_l(s)}\,d\sigma_{l}(s)+Er_{5,l},
\end{align}
where
\begin{align}\label{Er_5}
 Er_{5,l}&=\vert{D_l}\vert^{-1}Er_{4,l}+\vert{D_l}\vert^{-1}\Bigg[Err5_{l}-\sum_{{m=1} \atop {m \neq l}}^{M} Err4_{m} + \Bigg(Err3_{l}+\Big(1-\frac{\rho_{l}}{\rho_{0}} \Big)^{-1}(-Err2_{l}+Err8_{l})\Bigg)\Bigg]\\ \notag
 &=O(a^2\|\psi_l\|)+O\left(a+\sum_{m\neq\,l}\frac{a^3}{d^3_{ml}}\|\phi_m\|+a^2\|\psi_l\|+a^2\|\phi_l\|+a^2\|g_l\|\right)\\ \notag
 &=O(a^2\|\psi_l\|)+O\left(a+\sum_{m\neq\,l}\frac{a^3}{d^3_{ml}}\|\phi_m\|+a^2\|\psi_l\|+a^2\|\phi_l\|+a^3\|\psi_l\|\right)\\ \notag
 &=O\left(a+\sum_{m\neq\,l}\frac{a^3}{d^3_{ml}}\|\phi_m\|+a^2\|\phi_l\|+a^2\|\psi_l\|\right). \notag
\end{align}

Collecting the error terms together, we can rewrite \eqref{mult-obs-109r4} further as
\begin{align}\label{mult-obs-109r6}
& \vert{ D_l}\vert^{-1} \int_{\partial D_{l}} \left(\frac{\rho_{l}}{\rho_{l}-\rho_{0}}+\frac{1}{8\pi}\left[\frac{\rho_{0}}{\rho_{l}-\rho_{0}}(\kappa_l^2-\kappa_0^2)-\kappa_0^2\right]{A_l(s)}\right) \phi_{l}(s) d\sigma_{l}(s)\\ \notag
&+\frac{i}{4\pi}\kappa_0^3\,Q_l+\frac{\rho_{0}}{\rho_{l}-\rho_{0}}\frac{i}{4\pi}(\kappa_0^3-\kappa_l^3)Q_l\\ \notag
&+\frac{\rho_{0}}{\rho_{l}-\rho_{0}}\frac{1}{16\pi^2}(\kappa_l^3-\kappa_0^3)(\kappa_0-\kappa_l)Cap_l Q_l\\ \notag
&+\frac{\rho_{0}}{\rho_{l}-\rho_{0}}\frac{i}{32\pi^2}(\kappa_l^2-\kappa_0^2)(\kappa_0-\kappa_l)\vert{ D_l}\vert^{-1}Q_l\int_{\partial D_l}\left(S^{{0}}_{D_{l}}\right)^{-1}\Big(1\Big)(s) {A_l(s)}d\sigma_{l}(s)
\\ \notag
&+\frac{\rho_{0}}{\rho_{l}-\rho_{0}} \frac{i}{32\pi^2}(\kappa_0-\kappa_l)\kappa_0^2\vert{ D_l}\vert^{-1}\left(\int_{\partial D_l} \left(S^{{0}}_{D_l}\right)^{-1}\left(1\right)(s){A_l(s)}d\sigma_{l}(s)\right)\;Q_l \;\\ \notag
&+\frac{\rho_{0}}{\rho_{l}-\rho_{0}}\frac{i}{4\pi}(\kappa_0^3-\kappa_l^3)
\sum_{m\neq\;l}\Big[Cap_l \left(1-\frac{i\kappa_l}{4\pi}Cap_l\right)\Phi_{\kappa_0}(z_l,z_m)\\ \notag
&\qquad \qquad +\nabla_s\Phi_{\kappa_0}(z_l,z_m)\cdot\int_{\partial D_l}\left(S_{D_l}^0\right)^{-1}(\cdot-z_l)(s) d\sigma_{l}(s)\Big]Q_m
\\ \notag
&+\frac{\rho_{0}}{\rho_{l}-\rho_{0}}\frac{1}{8\pi}(\kappa_l^2-\kappa_0^2)\vert{ D_l}\vert^{-1}\\ \notag
&\qquad\int_{\partial D_l}\left(S^{{0}}_{D_{l}}\right)^{-1}\Bigg(\sum_{m\neq\;l}Q_m\left[\left(1-\frac{i\kappa_l}{4\pi}Cap_l\right)\Phi_{\kappa_0}(z_l,z_m)+(\cdot-z_l)\cdot\nabla_s\Phi_{\kappa_0}(z_l,z_m)\right]\Bigg)(s) {A_l(s)}d\sigma_{l}(s)
\\ \notag
&+\frac{\rho_{0}}{\rho_{l}-\rho_{0}}\frac{i}{32\pi^2}\kappa_0^2(\kappa_0-\kappa_l)\vert{ D_l}\vert^{-1}\left[\sum_{m\neq\,l}\Phi_{\kappa_0}(z_l,z_m)\,Q_m Cap_l\right]\int_{\partial D_l} (S^{0}_{D_{l}})^{-1}\left(1\right)(s){A_l(s)}d\sigma_{l}(s)\\ \notag
&+\sum_{{m=1} \atop {m \neq l}}^{M}\left[\kappa_0^2\Phi_{\kappa_0}(z_l,z_m)\right]Q_m\\ \notag
&+\kappa_{0}^{2}\vert{ D_l}\vert^{-1}\sum_{{m=1} \atop {m \neq l}}^{M}\nabla_x\Phi_{\kappa_{0}}(z_{l},z_m)\cdot\Big[ \int_{D_{l}}(x-z_l) dx\Big] Q_m\\ \notag
&+\frac{\rho_{0}}{\rho_{l}-\rho_{0}}\frac{1}{8\pi}(\kappa_l^2-\kappa_0^2)\vert{ D_l}\vert^{-1}\Bigg(\sum_{m\neq\;l}\nabla_t\Phi_{\kappa_0}(z_l,z_m)\cdot\,V_m\Bigg)\int_{\partial D_l}\left(S^{{0}}_{D_{l}}\right)^{-1}\Big(1\Big)(s) {A_l(s)}d\sigma_{l}(s)
\\ \notag
&+\frac{\rho_{0}}{\rho_{l}-\rho_{0}}\frac{i}{4\pi}(\kappa_0^3-\kappa_l^3)
Cap_l \sum_{m\neq\;l}\nabla_t\Phi_{\kappa_0}(z_l,z_m)\cdot\,V_m
\\ \notag
&+\sum_{{m=1} \atop {m \neq l}}^{M}\left[\kappa_0^2\nabla_t\Phi_{\kappa_0}(z_l,z_m)\right]\cdot\, V_m\\ \notag
=&-\kappa_0^2\,u^{I}(z_l)
+\frac{\rho_{0}}{\rho_{l}-\rho_{0}}\frac{1}{8\pi}(\kappa_0^2-\kappa_l^2)\vert{ D_l}\vert^{-1}u^{I}(z_l)\int_{\partial D_l}\left(S^{{0}}_{D_{l}}\right)^{-1}\big(1\big)(s) {A_l(s)}d\sigma_{l}(s)
\\ \notag
&+\frac{\rho_{0}}{\rho_{l}-\rho_{0}}\frac{i}{4\pi}(\kappa_l^3-\kappa_0^3)Cap_l u^{I}(z_l)
\\ \notag
&+\frac{\rho_{0}}{\rho_{l}-\rho_{0}}\frac{i}{32\pi^2}\kappa_0^2(\kappa_l-\kappa_0)\vert{ D_l}\vert^{-1}Cap_l u^{I}(z_l)\int_{\partial D_l} (S^{0}_{D_{l}})^{-1}\left(1\right)(s){A_l(s)}d\sigma_{l}(s)+Er_{6,l}, \notag
\end{align} 
where
\begin{align} \label{Er_6}
 Er_{6,l}&=Er_{5,l}+\Big(1-\frac{\rho_{l}}{\rho_{0}} \Big)^{-1}\frac{1}{8\pi}(\kappa_l^2-\kappa_0^2)\vert{ D_l}\vert^{-1}\int_{\partial D_l}Err6_{l} \cdot {A_l(s)}d\sigma_{l}(s)\\ \notag
 &+\Big(1-\frac{\rho_{l}}{\rho_{0}} \Big)^{-1}\frac{1}{8\pi}\kappa_0^2\vert{ D_l}\vert^{-1}\int_{\partial D_l}Err9_{l} \cdot {A_l(s)}d\sigma_{l}(s)-\Big(1-\frac{\rho_{l}}{\rho_{0}}\Big)^{-1}\frac{i}{4\pi}(\kappa_l^3-\kappa_0^3)Err7_{l}\\ \notag
&= O\left(a+\sum_{m\neq\,l}\frac{a^3}{d^3_{ml}}\|\phi_m\|+a^2\|\phi_l\|+a^2\|\psi_l\|\right)+O\left(a+a^2\|\phi_l\|+\sum_{m\neq\,l}\frac{a^3}{d^3_{ml}}\|\phi_m\|\right)\\ \notag
&\quad+O\left(a^2+a^2\|\phi_l\|+\sum_{m\neq\,l}\frac{a^3}{d^2_{ml}}\|\phi_m\|\right)+O\left(a^2+a^3\|\phi_l\|+\sum_{m\neq\,l}\frac{a^4}{d^3_{ml}}\|\phi_m\|\right)\\ \notag
&=O\left(a+\sum_{m\neq\,l}\frac{a^3}{d^3_{ml}}\|\phi_m\|+a^2\|\phi_l\|+a^2\|\psi_l\|\right). \notag
\end{align}
%
In the following lemma, we state some identities concerning the integrals with $A_{l}$ in the integrand.
\begin{lemma}
The following identities hold true.
\begin{equation}
\int_{\partial D_l}[\left(S_{D_l}^0\right)^{-1} (1)](s) A_l(s) \,d\sigma_{l}(s) = -8\pi\vert{D_l}\vert,
\label{int-invs-Al}
\end{equation}
\begin{equation}
\int_{\partial D_l}[\left(S_{D_l}^0\right)^{-1} (\cdot-z_l)] (s) A_l(s) \,d\sigma_{l}(s) 
=-8\pi\int_{D_l}(x-z_l)\,dx.
\label{int-invs-Al-2}
\end{equation}
\end{lemma}
\begin{proof}
Observe that, 
\begin{align*}
\int_{\partial D_l}[\left(S_{D_l}^0\right)^{-1} (1)](s) A_l(s) \,d\sigma_{l}(s)&=-2\int_{D_l}\int_{\partial D_l}[\left(S_{D_l}^0\right)^{-1} (1)] (s) \frac{1}{\vert{s-y}\vert} \,d\sigma_{l}(s)\, dy\\
&=-8\pi\int_{D_l}\int_{\partial D_l}[\left(S_{D_l}^0\right)^{-1} (1)] (s) \Phi_0(s,y)\,d\sigma_{l}(s)\, dy\\
&=-8\pi\int_{D_l}S_{D_l}^0\left(\left(S_{D_l}^0\right)^{-1} (1) \right)(y)\,dy.
\end{align*}
%
Now, by denoting $f:=S_{D_l}^0\left(\left(S_{D_l}^0\right)^{-1} (1)\right)$, we can observe that $\Delta\,f=0$ in $D_l$ and $f=1$ on $\partial D_l$. 
Since the boundary integral equation has a unique solution $f=1$ in $\bar{D}$, we can conclude that
\begin{equation}
\int_{\partial D_l}[\left(S_{D_l}^0\right)^{-1} (1)](s) A_l(s) \,d\sigma_{l}(s)=-8\pi\vert{D_l}\vert.
\notag
\end{equation}
Similarly, we can prove \eqref{int-invs-Al-2}.
\end{proof}
Let us denote the average of $A_{l}$ as \[\hat{A}_{l}:= \frac{1}{\vert \partial D_{l} \vert}\int_{\partial D_{l}} A_{l}(s)\ d\sigma_{l}(s).\]
It is easy to see from the definitions that both $A_{l}$ and $\hat{A}_{l}$ scale as $O(a^3)$ in $L^{2}$ norm and therefore 
\begin{equation}
\norm{A_{l}-\hat{A}_{l}}_{L^{2}(\partial D_{l})}=O(a^3).
\label{estimate-Al}
\end{equation}
\begin{remark}
While dealing with the difference of $A_{l}$ and $\hat{A}_{l}$, one is tempted to believe that the difference behaves better than $A_{l}$ and such a result might be possible to derive using Poincar\'e type inequalities. But such a result is not true in this case. Indeed from the definition of $A_{l}$, we note that 
\begin{equation}
 A_l(s)=-2\int_{  D_l}\frac{1}{\vert{s-y}\vert}dy.
\label{def-A_l}
\end{equation}
%
%
Now, using the Poincar\'e type inequality mentioned in \cite[Proposition 3.2]{Aless-Morassi2002-Inv}, we can conclude that there exists a positive constant $C_A$ independent of $a$ such that
\begin{equation}
\begin{split}
\|A- \hat{A}_{l}\|^2_{L^2(\partial D_l)}
&\leq C_A\, a \int_{  D_l}\vert{\nabla A_l(x)}\vert^2 dx \leq 4C_A\, a \int_{  D_l} \left\vert\int_{  D_l}\frac{1}{\vert{x-y}\vert^2}dy\right\vert^2 dx \\
&\qquad \leq 4C_A\, a^6 \int_{  B_l} \left\vert\int_{  B_l}\frac{1}{\vert{\xi-\eta}\vert^2}d\xi\right\vert^2 d\eta =\tilde{C}_A a^6, 
\end{split}
\label{Poincare-A}
\end{equation}
where
$$\tilde{C}_A := 4 C_A \int_{  B_l} \left\vert\int_{  B_l}\frac{1}{\vert{\xi-\eta}\vert^2}d\xi\right\vert^2 d\eta.$$
\QEDB
\end{remark}
Now by making use of \eqref{int-invs-Al} and splitting $A_l$ as $\hat{A}_l+(A_l-\hat{A}_l)$, we can rewrite \eqref{mult-obs-109r6} as 
\begin{align}\label{mult-obs-109r6-pr}
& \Bigg[\vert{ D_l}\vert^{-1}\left(\frac{\rho_{l}}{\rho_{l}-\rho_{0}}+\frac{1}{8\pi}\left[\frac{\rho_{0}}{\rho_{l}-\rho_{0}}(\kappa_l^2-\kappa_0^2)-\kappa_0^2\right]{\hat{A}_l}\right) +\frac{i}{4\pi}\kappa_0^3+\frac{\rho_{0}}{\rho_{l}-\rho_{0}}\frac{i}{4\pi}(\kappa_0^3-\kappa_l^3) \\
&+\frac{\rho_{0}}{\rho_{l}-\rho_{0}}\frac{1}{16\pi^2}(\kappa_l^3-\kappa_0^3)(\kappa_0-\kappa_l)Cap_l
-\frac{\rho_{0}}{\rho_{l}-\rho_{0}}\frac{i}{4\pi}(\kappa_l^2-\kappa_0^2)(\kappa_0-\kappa_l)
-\frac{\rho_{0}}{\rho_{l}-\rho_{0}} \frac{i}{4\pi}(\kappa_0-\kappa_l)\kappa_0^2\Bigg]\;Q_l \;  \nonumber\\
&+\frac{1}{8\pi}\left[\frac{\rho_{0}}{\rho_{l}-\rho_{0}}(\kappa_l^2-\kappa_0^2)-\kappa_0^2\right] \vert{ D_l}\vert^{-1} \int_{\partial D_{l}} \left(A_l(s)-\hat{A}_l(s)\right) \phi_{l}(s) d\sigma_{l}(s) \nonumber \\
&+
\Bigg[\frac{\rho_{0}}{\rho_{l}-\rho_{0}}\frac{i}{4\pi}(\kappa_0^3-\kappa_l^3)Cap_l \left(1-\frac{i\kappa_l}{4\pi}Cap_l\right)-\frac{\rho_{0}}{\rho_{l}-\rho_{0}}(\kappa_l^2-\kappa_0^2) \left(1-\frac{i\kappa_l}{4\pi}Cap_l\right)
 \nonumber \\
&-\frac{\rho_{0}}{\rho_{l}-\rho_{0}}\frac{i}{4\pi}\kappa_0^2(\kappa_0-\kappa_l)Cap_l+\kappa_0^2\Bigg]\sum_{{m=1} \atop {m \neq l}}^{M}\Phi_{\kappa_0}(z_l,z_m)Q_m \nonumber\\
&+\Bigg[\frac{\rho_{0}}{\rho_{l}-\rho_{0}}\frac{i}{4\pi}(\kappa_0^3-\kappa_l^3)\int_{\partial D_l}\left(S_{D_l}^0\right)^{-1}(\cdot-z_l)(s) d\sigma_{l}(s)-\frac{\rho_{0}}{\rho_{l}-\rho_{0}}(\kappa_l^2-\kappa_0^2)\vert{ D_l}\vert^{-1}\Big[ \int_{D_{l}}(x-z_l) dx\Big]
\nonumber \\
&+\kappa_{0}^{2}\vert{ D_l}\vert^{-1}\Big[ \int_{D_{l}}(x-z_l) dx\Big]\Bigg]\cdot\sum_{{m=1} \atop {m \neq l}}^{M}\nabla_x\Phi_{\kappa_{0}}(z_{l},z_m) Q_m \nonumber \\
&+\Bigg[\kappa_0^2 -\frac{\rho_{0}}{\rho_{l}-\rho_{0}}(\kappa_l^2-\kappa_0^2)+\frac{\rho_{0}}{\rho_{l}-\rho_{0}}\frac{i}{4\pi}(\kappa_0^3-\kappa_l^3)
Cap_l\Bigg]\sum_{{m=1} \atop {m \neq l}}^{M}\nabla_t\Phi_{\kappa_0}(z_l,z_m)\cdot\, V_m \nonumber\\
=&\Bigg[-\kappa_0^2
-\frac{\rho_{0}}{\rho_{l}-\rho_{0}}(\kappa_0^2-\kappa_l^2)
+\frac{\rho_{0}}{\rho_{l}-\rho_{0}}\frac{i}{4\pi}(\kappa_l^3-\kappa_0^3)Cap_l 
-\frac{\rho_{0}}{\rho_{l}-\rho_{0}}\frac{i}{4\pi}\kappa_0^2(\kappa_l-\kappa_0)Cap_l \Bigg]u^{I}(z_l)+Er_{6,l}.\nonumber
\end{align}
%
To proceed further, we shall use the following result describing the structure of $V_m$. Let $\overline{z}_{l}:=\frac{1}{\vert \partial D_{l} \vert} \int_{\partial D_{l}} s \ d\sigma_{l}(s) .$
\begin{proposition}\label{estimate-Vm}
The term $V_m$ can be written as \[V_m:=V^{dom}_{m}+V^{rem}_{m}, \]
where
\begin{align}\label{V-dom}
&V^{dom}_{m}:=\sum_{{n=1} \atop {n \neq m}}^{M} (\overline{z}_{m}-z_{m}) \Big(\lambda_{m} -\frac{1}{2} \Big)^{-1} \kappa_0^2\vert{ D_m}\vert  \Phi_{\kappa_0}(z_m,z_n) Q_{n}-\frac{1}{8 \pi} \kappa_{0}^{2} (\overline{z}_{m}-z_{m}) (\lambda_{m}-\frac{1}{2})^{-1} \hat{A}_{m} Q_{m} \\
&-\frac{1}{8 \pi} \kappa_{0}^{2} \Big(1-\frac{\rho_{m}}{\rho_{0}}\Big)^{-1} (\overline{z}_{m}-z_{m}) (\lambda_{m}-\frac{1}{2})^{-1} \hat{A}_{m} Q_{m}\nonumber\\
&-\frac{1}{8\pi} \kappa_{0}^{2} \Big(1-\frac{\rho_{m}}{\rho_{0}}\Big)^{-1} (\overline{z}_{m}-z_{m}) (\lambda_{m}-\frac{1}{2})^{-1} \hat{A}_{m} Cap_{m} \sum_{{n=1} \atop {n \neq m}}^{M} \Phi_{\kappa_{0}}(z_{m},z_{n}) Q_{n} \nonumber\\
&-\frac{1}{8\pi} \kappa_{0}^{2} \Big(1-\frac{\rho_{m}}{\rho_{0}}\Big)^{-1} (\overline{z}_{m}-z_{m}) (\lambda_{m}-\frac{1}{2})^{-1}  \int_{\partial D_{m}} (A_{m}(s)-\hat{A}_{m}) \Big(\sum_{{n=1} \atop {n \neq m}}^{M} \Phi_{\kappa_{0}}(z_{m},z_{n}) Q_{n} \Big) (S^{0}_{D_{m}})^{-1}(s) d\sigma_{m}(s),\nonumber
\end{align}
and $V_{m}^{rem} $ satisfies the estimate
\begin{equation}
\vert V^{rem}_{m} \vert
=O(a^{3-\gamma})+O\Big( \Big(a^{3}+a^{4-\gamma}+\frac{a^{5-\gamma}}{d^2}+\frac{a^{5-\gamma}}{d^{3\alpha}}+\frac{a^4}{d^2}+\frac{a^4}{d^{3 \alpha}}+\frac{a^{6-\gamma}}{d^{3}}+\frac{a^{6-\gamma}}{d^{3 \alpha+1}} \Big) \norm{\phi} \Big).
\label{V-rem}
\end{equation}
\end{proposition}
\begin{proof}
We refer to section \ref{Vm} for a proof of this result.
\end{proof}
\begin{remark}\label{split-Vm}
We shall also sometimes, for the sake of the analysis, split the term $V^{dom}_m $ into $V^{dom}_{m,1} $ and $ V^{dom}_{m,2}$, where
\begin{align}
&V^{dom}_{m,1}:=-\frac{1}{8 \pi} \kappa_{0}^{2} (\overline{z}_{m}-z_{m}) (\lambda_{m}-\frac{1}{2})^{-1} \hat{A}_{m} Q_{m} -\frac{1}{8 \pi} \kappa_{0}^{2} \Big(1-\frac{\rho_{m}}{\rho_{0}}\Big)^{-1} (\overline{z}_{m}-z_{m}) (\lambda_{m}-\frac{1}{2})^{-1} \hat{A}_{m} Q_{m},
\label{Vm-dom-1}
\intertext{and }
\label{Vm-dom-2}
&V^{dom}_{m,2}:=\sum_{{n=1} \atop {n \neq m}}^{M} (\overline{z}_{m}-z_{m}) \Big(\lambda_{m} -\frac{1}{2} \Big)^{-1} \kappa_0^2\vert{ D_m}\vert  \Phi_{\kappa_0}(z_m,z_n) Q_{n}\\
&-\frac{1}{8\pi} \kappa_{0}^{2} \Big(1-\frac{\rho_{m}}{\rho_{0}}\Big)^{-1} (\overline{z}_{m}-z_{m}) (\lambda_{m}-\frac{1}{2})^{-1} \hat{A}_{m} Cap_{m} \sum_{{n=1} \atop {n \neq m}}^{M} \Phi_{\kappa_{0}}(z_{m},z_{n}) Q_{n}\nonumber\\
&-\frac{1}{8\pi} \kappa_{0}^{2} \Big(1-\frac{\rho_{m}}{\rho_{0}}\Big)^{-1} (\overline{z}_{m}-z_{m}) (\lambda_{m}-\frac{1}{2})^{-1}  \int_{\partial D_{m}} (A_{m}(s)-\hat{A}_{m}) \Big(\sum_{{n=1} \atop {n \neq m}}^{M} \Phi_{\kappa_{0}}(z_{m},z_{n}) Q_{n} \Big) (S^{0}_{D_{m}})^{-1}(s) d\sigma_{m}(s).\nonumber
\end{align}
Note that unlike the terms in $V^{dom}_{m,1} $, the terms in $V^{dom}_{m,2} $ contain a summation and this makes a difference in the manner we deal with these terms and therefore we chose to distinguish between them.\QEDB
\end{remark}
Next we need to deal with the quantities $V_{m} $ and link them to $Q_{m} $ to derive a closed system involving only $Q_{m}$. For this, we use the splitting of $V_{m}$ into its dominant and remainder parts, $V_{m}^{dom} $ and $V_{m}^{rem} $ respectively, to obtain
\begin{align}\label{mult-obs-109r6-pr2}
& \Bigg[\vert{ D_l}\vert^{-1}\left(\frac{\rho_{l}}{\rho_{l}-\rho_{0}}+\frac{1}{8\pi}\left[\frac{\rho_{0}}{\rho_{l}-\rho_{0}}(\kappa_l^2-\kappa_0^2)-\kappa_0^2\right]{\hat{A}_l}\right) +\frac{i}{4\pi}\kappa_0^3+\frac{\rho_{0}}{\rho_{l}-\rho_{0}}\frac{i}{4\pi}(\kappa_0^3-\kappa_l^3)\\
&+\frac{\rho_{0}}{\rho_{l}-\rho_{0}}\frac{1}{16\pi^2}(\kappa_l^3-\kappa_0^3)(\kappa_0-\kappa_l)Cap_l
-\frac{\rho_{0}}{\rho_{l}-\rho_{0}}\frac{i}{4\pi}(\kappa_l^2-\kappa_0^2)(\kappa_0-\kappa_l)
-\frac{\rho_{0}}{\rho_{l}-\rho_{0}} \frac{i}{4\pi}(\kappa_0-\kappa_l)\kappa_0^2\Bigg]\;Q_l \;\nonumber\\
&+
\Bigg[\frac{\rho_{0}}{\rho_{l}-\rho_{0}}\frac{i}{4\pi}(\kappa_0^3-\kappa_l^3)Cap_l \left(1-\frac{i\kappa_l}{4\pi}Cap_l\right)-\frac{\rho_{0}}{\rho_{l}-\rho_{0}}(\kappa_l^2-\kappa_0^2)\left(1-\frac{i\kappa_l}{4\pi}Cap_l\right)
\nonumber\\
&-\frac{\rho_{0}}{\rho_{l}-\rho_{0}}\frac{i}{4\pi}\kappa_0^2(\kappa_0-\kappa_l)Cap_l+\kappa_0^2\Bigg]\sum_{{m=1} \atop {m \neq l}}^{M}\Phi_{\kappa_0}(z_l,z_m)Q_m\nonumber\\
&+\Bigg[\frac{\rho_{0}}{\rho_{l}-\rho_{0}}\frac{i}{4\pi}(\kappa_0^3-\kappa_l^3)\int_{\partial D_l}\left(S_{D_l}^0\right)^{-1}(\cdot-z_l)(s) d\sigma_{l}(s)-\frac{\rho_{0}}{\rho_{l}-\rho_{0}}(\kappa_l^2-\kappa_0^2)\vert{ D_l}\vert^{-1}\Big[ \int_{D_{l}}(x-z_l) dx\Big]
\nonumber\\
&+\kappa_{0}^{2}\vert{ D_l}\vert^{-1}\Big[ \int_{D_{l}}(x-z_l) dx\Big]\Bigg]\cdot\sum_{{m=1} \atop {m \neq l}}^{M}\nabla_x\Phi_{\kappa_{0}}(z_{l},z_m) Q_m\nonumber\\
&+\Bigg[\kappa_0^2 -\frac{\rho_{0}}{\rho_{l}-\rho_{0}}(\kappa_l^2-\kappa_0^2)+\frac{\rho_{0}}{\rho_{l}-\rho_{0}}\frac{i}{4\pi}(\kappa_0^3-\kappa_l^3)
Cap_l\Bigg]\sum_{{m=1} \atop {m \neq l}}^{M}\nabla_t\Phi_{\kappa_0}(z_l,z_m)\cdot\, V^{dom}_m \nonumber\\
=&\Bigg[-\kappa_0^2
-\frac{\rho_{0}}{\rho_{l}-\rho_{0}}(\kappa_0^2-\kappa_l^2)
\Bigg]u^{I}(z_l)-\frac{1}{8\pi}\left[\frac{\rho_{0}}{\rho_{l}-\rho_{0}}(\kappa_l^2-\kappa_0^2)-\kappa_0^2\right] \vert{ D_l}\vert^{-1} \int_{\partial D_{l}} \left(A_l(s)-\hat{A}_l(s)\right) \phi_{l}(s) d\sigma_{l}(s)\nonumber\\
&+O\left((\kappa_0-\kappa_l)a\right)+O\left(\sum_{{m=1} \atop {m \neq l}}^{M}\frac{1}{d_{ml}^2} \vert{V^{rem}_m}\vert\right)+Er_{6,l}.\nonumber
\end{align}
We next have to deal with the term $\int_{\partial D_{l}} \left(A_l(s)-\hat{A}_l(s)\right) \phi_{l}(s) d\sigma_{l}(s) $ in the above identity.
The following lemma provides an approximation of this term in terms of $Q_{l}$.
\begin{proposition}\label{Est-intAlminhatAlphil}
We have the following estimate
\begin{align}\label{mult-obs-108r-pr-th}
\int_{\partial D_l}\Big(A_l(s)-\hat{A}_l\Big)\phi_{l}\ d\sigma_{l}(s)
 {=}&-R_l\cdot \sum_{{m=1} \atop {m \neq l}}^{M} \nabla_s\Phi_{\kappa_0}(z_l,z_m) Q_m\\
 &-\frac{\rho_{0}}{\rho_{l}-\rho_{0}}  \frac{i}{4\pi}(\kappa_0-\kappa_l)\big[Q_l+\sum_{m\neq\,l}Cap_l\Phi_{\kappa_0}(z_l,z_m)Q_m\big]\Big(8\pi\vert{D_l}\vert+\hat{A}_lCap_l\Big)\nonumber\\
 &+O\left(a^4+a^3\sum_{{m=1}\atop {m\neq l}}^{M}\frac{a^3}{d_{ml}^3}\|\phi_m\|+a^5\|\phi_l\|+a^5\|\psi_l\|\right),\nonumber
 \end{align}
 where $R_l:=\int_{\partial D_l} \Big[\Big(\lambda_{l} Id+ K^{0}_{D_{l}} \Big)^{-1}(A_l(\cdot)-\hat{A}_l)\Big] (s)\nu_l(s)\,d\sigma_{l}(s) $ and $\lambda_l:=\frac{1}{2} \frac{\rho_0+\rho_l}{\rho_0-\rho_l} $.
\end{proposition}
\begin{proof}
We defer the proof of this result to section \ref{proof-Al-hat}.
\end{proof}
Making use of Proposition \ref{Est-intAlminhatAlphil} we can simplify \eqref{mult-obs-109r6-pr2} as below.
\begin{align}\label{mult-obs-109r6-pr3-int}
& \Bigg[\vert{ D_l}\vert^{-1}\left(\frac{\rho_{l}}{\rho_{l}-\rho_{0}}+\frac{1}{8\pi}\left[\frac{\rho_{0}}{\rho_{l}-\rho_{0}}(\kappa_l^2-\kappa_0^2)-\kappa_0^2\right]{\hat{A}_l}\right) +\frac{i}{4\pi}\kappa_0^3+\frac{\rho_{0}}{\rho_{l}-\rho_{0}}\frac{i}{4\pi}(\kappa_0^3-\kappa_l^3)\\
&+\frac{\rho_{0}}{\rho_{l}-\rho_{0}}\frac{1}{16\pi^2}(\kappa_l^3-\kappa_0^3)(\kappa_0-\kappa_l)Cap_l
-\frac{\rho_{0}}{\rho_{l}-\rho_{0}}\frac{i}{4\pi}(\kappa_l^2-\kappa_0^2)(\kappa_0-\kappa_l)
-\frac{\rho_{0}}{\rho_{l}-\rho_{0}} \frac{i}{4\pi}(\kappa_0-\kappa_l)\kappa_0^2\Bigg]\;Q_l \; \nonumber\\
&+
\Bigg[\frac{\rho_{0}}{\rho_{l}-\rho_{0}}\frac{i}{4\pi}(\kappa_0^3-\kappa_l^3)Cap_l \left(1-\frac{i\kappa_l}{4\pi}Cap_l\right)-\frac{\rho_{0}}{\rho_{l}-\rho_{0}}(\kappa_l^2-\kappa_0^2)\left(1-\frac{i\kappa_l}{4\pi}Cap_l\right)\nonumber
\\
&-\frac{\rho_{0}}{\rho_{l}-\rho_{0}}\frac{i}{4\pi}\kappa_0^2(\kappa_0-\kappa_l)Cap_l+\kappa_0^2\Bigg]\sum_{{m=1} \atop {m \neq l}}^{M}\Phi_{\kappa_0}(z_l,z_m)Q_m \nonumber\\
&+\Bigg[\frac{\rho_{0}}{\rho_{l}-\rho_{0}}\frac{i}{4\pi}(\kappa_0^3-\kappa_l^3)\int_{\partial D_l}\left(S_{D_l}^0\right)^{-1}(\cdot-z_l)(s) d\sigma_{l}(s)-\frac{\rho_{0}}{\rho_{l}-\rho_{0}}(\kappa_l^2-\kappa_0^2)\vert{ D_l}\vert^{-1}\Big[ \int_{D_{l}}(x-z_l) dx\Big]\nonumber
 \\
&+\kappa_{0}^{2}\vert{D_l}\vert^{-1}\Big[ \int_{D_{l}}(x-z_l) dx\Big]\Bigg]\cdot\sum_{{m=1} \atop {m \neq l}}^{M}\nabla_x\Phi_{\kappa_{0}}(z_{l},z_m) Q_m \nonumber\\
&+\Bigg[\kappa_0^2 -\frac{\rho_{0}}{\rho_{l}-\rho_{0}}(\kappa_l^2-\kappa_0^2)+\frac{\rho_{0}}{\rho_{l}-\rho_{0}}\frac{i}{4\pi}(\kappa_0^3-\kappa_l^3)
Cap_l\Bigg]\sum_{{m=1} \atop {m \neq l}}^{M}\nabla_t\Phi_{\kappa_0}(z_l,z_m)\cdot\, V^{dom}_m \nonumber\\
=&\Bigg[-\kappa_0^2
-\frac{\rho_{0}}{\rho_{l}-\rho_{0}}(\kappa_0^2-\kappa_l^2)
\Bigg]u^{I}(z_l)+O\left((\kappa_0-\kappa_l)a\right)+O\left(\sum_{{m=1} \atop {m \neq l}}^{M}\frac{1}{d_{ml}^2} \vert{V^{rem}_m}\vert\right)+Er_{6,l} \nonumber\\
&-\frac{1}{8\pi}\left[\frac{\rho_{0}}{\rho_{l}-\rho_{0}}(\kappa_l^2-\kappa_0^2)-\kappa_0^2\right] \vert{ D_l}\vert^{-1} \Bigg[-R_l\cdot \sum_{{m=1} \atop {m \neq l}}^{M} \nabla_s\Phi_{\kappa_0}(z_l,z_m) Q_m \nonumber\\
 &-\frac{\rho_{0}}{\rho_{l}-\rho_{0}}  \frac{i}{4\pi}(\kappa_0-\kappa_l)\big[Q_l+\sum_{m\neq\,l}Cap_l\Phi_{\kappa_0}(z_l,z_m)Q_m\big]\Big(8\pi\vert{D_l}\vert+\hat{A}_lCap_l\Big) \nonumber\\
 &+O\left(a^4+a^3\sum_{{m=1}\atop {m\neq l}}^{M}\frac{a^3}{d_{ml}^3}\|\phi_m\|+a^5\|\phi_l\|+a^5\|\psi_l\|\right)\Bigg].\nonumber
\end{align}
%
%
We can rewrite it further as, by making use of \eqref{InvSsmzl--ori} and the expression for $V^{dom}_{m} $,
\begin{align}\label{mult-obs-109r6-pr3}
& \Bigg[\vert{ D_l}\vert^{-1}\left(\frac{\rho_{l}}{\rho_{l}-\rho_{0}}+\frac{1}{8\pi}\left[\frac{\rho_{0}}{\rho_{l}-\rho_{0}}(\kappa_l^2-\kappa_0^2)-\kappa_0^2\right]{\hat{A}_l}\right) +\frac{i}{4\pi}\kappa_0^3+\frac{\rho_{0}}{\rho_{l}-\rho_{0}}\frac{i}{4\pi}(\kappa_0^3-\kappa_l^3)\\
&
-\frac{\rho_{0}}{\rho_{l}-\rho_{0}}\frac{i}{4\pi}\Big[(\kappa_l^2-\kappa_0^2)(\kappa_0-\kappa_l)
+(\kappa_0-\kappa_l)\kappa_0^2\Big]\;\nonumber\\
&-\frac{i}{32\pi^2}(\kappa_0-\kappa_l)\left[\frac{\rho_{0}}{\rho_{l}-\rho_{0}}(\kappa_l^2-\kappa_0^2)-\kappa_0^2\right] \vert{ D_l}\vert^{-1}\frac{\rho_{0}}{\rho_{l}-\rho_{0}}  \left(8\pi\vert{D_l}\vert+\hat{A}_lCap_l\right)\Bigg]\;Q_l \nonumber\\
&+\Bigg[\kappa_0^2-\frac{\rho_{0}}{\rho_{l}-\rho_{0}}(\kappa_l^2-\kappa_0^2)\left(1-\frac{i\kappa_l}{4\pi}Cap_l\right)+\frac{\rho_{0}}{\rho_{l}-\rho_{0}}\frac{i}{4\pi}(\kappa_0^3-\kappa_l^3)Cap_l-\frac{\rho_{0}}{\rho_{l}-\rho_{0}}\frac{i}{4\pi}\kappa_0^2(\kappa_0-\kappa_l)Cap_l\nonumber
\\
&-\frac{i}{32\pi^2}(\kappa_0-\kappa_l)\left[\frac{\rho_{0}}{\rho_{l}-\rho_{0}}(\kappa_l^2-\kappa_0^2)-\kappa_0^2\right] \vert{ D_l}\vert^{-1}\frac{\rho_{0}}{\rho_{l}-\rho_{0}}  \left(8\pi\vert{D_l}\vert+\hat{A}_lCap_l\right)Cap_l\Bigg]\sum_{{m=1} \atop {m \neq l}}^{M}\Phi_{\kappa_0}(z_l,z_m)Q_m \nonumber\\
&+\Bigg[-\frac{\rho_{0}}{\rho_{l}-\rho_{0}}(\kappa_l^2-\kappa_0^2)\vert{ D_l}\vert^{-1}\Big[ \int_{D_{l}}(x-z_l) dx\Big]
+\kappa_{0}^{2}\vert{ D_l}\vert^{-1}\Big[ \int_{D_{l}}(x-z_l) dx\Big] 
\nonumber\\
&\qquad -\frac{1}{8\pi}\left[\frac{\rho_{0}}{\rho_{l}-\rho_{0}}(\kappa_l^2-\kappa_0^2)-\kappa_0^2\right] \vert{ D_l}\vert^{-1} R_l\Bigg]\cdot\sum_{{m=1} \atop {m \neq l}}^{M}\nabla_x\Phi_{\kappa_{0}}(z_{l},z_m) Q_m\nonumber\\
&+\Bigg[\kappa_0^2 -\frac{\rho_{0}}{\rho_{l}-\rho_{0}}(\kappa_l^2-\kappa_0^2)+\frac{\rho_{0}}{\rho_{l}-\rho_{0}}\frac{i}{4\pi}(\kappa_0^3-\kappa_l^3)
Cap_l\Bigg]\kappa_{0}^{2} \sum_{{m=1} \atop {m \neq l}}^{M} (\overline{z}_{m}-z_{m}) (\lambda_{m}-\frac{1}{2})^{-1} \nabla_t\Phi_{\kappa_0}(z_l,z_m)\cdot\nonumber\\
&\qquad \Big[-\frac{1}{8 \pi} \hat{A}_{m} Q_{m} -\frac{1}{8 \pi} \Big(1-\frac{\rho_{m}}{\rho_{0}} \Big)^{-1} \hat{A}_{m} Q_{m}+ \vert D_{m} \vert \sum_{{n=1} \atop {n \neq m}}^{M} \Phi_{\kappa_{0}}(z_{m},z_{n}) Q_{n}\nonumber\\
&\qquad \qquad -\frac{1}{8 \pi} \Big(1-\frac{\rho_{m}}{\rho_{0}} \Big)^{-1} \hat{A}_{m} Cap_{m} \sum_{{n=1} \atop {n \neq m}}^{M} \Phi_{\kappa_{0}}(z_{m},z_{n}) Q_{n}\nonumber\\
&\qquad \qquad -\frac{1}{8\pi} \Big(1-\frac{\rho_{m}}{\rho_{0}} \Big)^{-1} \Big(\sum_{{n=1} \atop {n \neq m}}^{M} \Phi_{k_{0}}(z_{m},z_{n}) Q_{n} \Big) \int_{\partial D_{m}} (A_{m}(s)-\hat{A}_{m}) (S^{0}_{D_{m}})^{-1}(1)(s)\ d\sigma_{m}(s) \Big] \nonumber\\
=&\Bigg[-\kappa_0^2
-\frac{\rho_{0}}{\rho_{l}-\rho_{0}}(\kappa_0^2-\kappa_l^2)
\Bigg]u^{I}(z_l)\nonumber\\
&+O\left((\kappa_0-\kappa_l)a\right)+O\left(\sum_{{m=1} \atop {m \neq l}}^{M}\frac{1}{d_{ml}^2} \vert{V^{rem}_m}\vert\right)+Er_{6,l}+O\left(a+\sum_{{m=1}\atop {m\neq l}}^{M}\frac{a^3}{d_{ml}^3}\|\phi_m\|+a^2\|\phi_l\|+a^2\|\psi_l\|\right)\nonumber\\
=&-\kappa_l^2u^{I}(z_l)+O\left((\kappa_0-\kappa_l)a\right)+O\left(\sum_{{m=1} \atop {m \neq l}}^{M}\frac{1}{d_{ml}^2} \vert{V^{rem}_m}\vert\right) \nonumber \\
&+Er_{6,l}+O\left(a+\sum_{{m=1}\atop {m\neq l}}^{M}\frac{a^3}{d_{ml}^3}\|\phi_m\|+a^2\|\phi_l\|+a^2\|\psi_l\|\right)+O(\rho_{l}(\kappa_{0}-\kappa_{l})),\nonumber
\end{align}
where to get the last identity, we use the fact that
 \begin{align}
&\frac{\rho_{0}}{\rho_{l}-\rho_{0}}(\kappa_0^2-\kappa_l^2)\,=\,\left(1-\frac{\rho_{l}}{\rho_{0}}\right)^{-1}(\kappa_l^2-\kappa_0^2)\,=\,\kappa_l^2-\kappa_0^2+O\left(\rho_l(\kappa_0-\kappa_l)\right).
\label{obs1coeui}
\end{align}
Let us write 
\begin{align}\label{def-of-Jl}
J_l&:=\Bigg[\kappa_0^2-\frac{\rho_{0}}{\rho_{l}-\rho_{0}}(\kappa_l^2-\kappa_0^2)\left(1-\frac{i\kappa_l}{4\pi}Cap_l\right)+\frac{\rho_{0}}{\rho_{l}-\rho_{0}}\frac{i}{4\pi}(\kappa_0^3-\kappa_l^3)Cap_l-\frac{\rho_{0}}{\rho_{l}-\rho_{0}}\frac{i}{4\pi}\kappa_0^2(\kappa_0-\kappa_l)Cap_l
\\
&\textcolor{blue}{-}\frac{i}{32\pi^2}(\kappa_0-\kappa_l)\left[\frac{\rho_{0}}{\rho_{l}-\rho_{0}}(\kappa_l^2-\kappa_0^2)-\kappa_0^2\right] \vert{ D_l}\vert^{-1}\frac{\rho_{0}}{\rho_{l}-\rho_{0}}  \left(8\pi\vert{D_l}\vert+\hat{A}_lCap_l\right)Cap_l\Bigg]\nonumber\\
&=\kappa_l^2+O\left(\rho_l(\kappa_0-\kappa_l)\right)+iO\left((\kappa_0-\kappa_l)a\right), \ (\text{using \eqref{obs1coeui}}),\nonumber
\end{align}
where using \eqref{obs1coeui}, it follows that the real part $J_{l}^{r} $ of $J_{l}$ satisfies
\begin{equation}
J_{l}^{r}=\kappa_0^2-\frac{\rho_{0}}{\rho_{l}-\rho_{0}}(\kappa_l^2-\kappa_0^2)=\kappa_l^2+\frac{\rho_{l}}{\rho_{0}}(\kappa^{2}_{l}-\kappa^{2}_{0})+O\left(\rho_l^2(\kappa_0-\kappa_l)\right).
\notag
\end{equation}
Let
\begin{align}\label{def-of-Il}
I_l&:= \Bigg[\vert{ D_l}\vert^{-1}\left(\frac{\rho_{l}}{\rho_{l}-\rho_{0}}+\frac{1}{8\pi}\left[\frac{\rho_{0}}{\rho_{l}-\rho_{0}}(\kappa_l^2-\kappa_0^2)-\kappa_0^2\right]{\hat{A}_l}\right) +\frac{i}{4\pi}\kappa_0^3+\frac{\rho_{0}}{\rho_{l}-\rho_{0}}\frac{i}{4\pi}(\kappa_0^3-\kappa_l^3)\\
&\qquad -\frac{\rho_{0}}{\rho_{l}-\rho_{0}}\frac{i}{4\pi}\Big[(\kappa_l^2-\kappa_0^2)(\kappa_0-\kappa_l)
+(\kappa_0-\kappa_l)\kappa_0^2\Big]\;\nonumber\\
&\qquad-\frac{i}{32\pi^2}(\kappa_0-\kappa_l)\left[\frac{\rho_{0}}{\rho_{l}-\rho_{0}}(\kappa_l^2-\kappa_0^2)-\kappa_0^2\right] \vert{ D_l}\vert^{-1}\frac{\rho_{0}}{\rho_{l}-\rho_{0}}  \left(8\pi\vert{D_l}\vert+\hat{A}_lCap_l\right)\Bigg]\nonumber\\
&= \vert{ D_l}\vert^{-1}\left(\frac{\rho_l}{\rho_l-\rho_0}+\frac{1}{8\pi}\left[-\sum_{n=0}^\infty\left(\frac{\rho_{l}}{\rho_{0}}\right)^n (\kappa_l^2-\kappa_0^2)-\kappa_0^2\right]{\hat{A}_l}\right) +\frac{i}{4\pi}\kappa_0^3\nonumber\\
&\quad -\left[1+\sum_{n=1}^\infty\left(\frac{\rho_{l}}{\rho_{0}}\right)^n\right]\frac{i}{4\pi}(\kappa_0^3-\kappa_l^3)
+\left[1+\sum_{n=1}^\infty\left(\frac{\rho_{l}}{\rho_{0}}\right)^n\right]
\frac{i}{4\pi}\Big[(\kappa_l^2-\kappa_0^2)(\kappa_0-\kappa_l)
+(\kappa_0-\kappa_l)\kappa_0^2\Big]\;\nonumber\\
&\quad -\frac{i}{32\pi^2}(\kappa_0-\kappa_l)\left[\left[1+\sum_{n=1}^\infty\left(\frac{\rho_{l}}{\rho_{0}}\right)^n\right](\kappa_l^2-\kappa_0^2)+\kappa_0^2\right] \vert{ D_l}\vert^{-1}\left[1+\sum_{n=1}^\infty\left(\frac{\rho_{l}}{\rho_{0}}\right)^n\right] \left(8\pi\vert{D_l}\vert+\hat{A}_l Cap_l\right)\nonumber\\
&=I_l^\prime+O\left(a^{-1}(\kappa_0-\kappa_l)\rho_l^2\right)+iO\left((\kappa_0-\kappa_l)\rho_l\right),\nonumber
\end{align}
where 
\begin{equation}
\begin{split}
I_l^\prime&:= \vert{ D_l}\vert^{-1}\left(\frac{\rho_l}{\rho_l-\rho_0}+\frac{1}{8\pi}\left[-\kappa_l^2-\frac{\rho_{l}}{\rho_{0}}(\kappa_l^2-\kappa_0^2)\right]{\hat{A}_l}\right)+\frac{i}{4\pi}\kappa_0^3-\frac{i}{4\pi}(\kappa_0^3-\kappa_l^3)\\
&+\frac{i}{4\pi}\Big[(\kappa_l^2-\kappa_0^2)(\kappa_0-\kappa_l)
+(\kappa_0-\kappa_l)\kappa_0^2\Big]\;-\frac{i}{32\pi^2}(\kappa_0-\kappa_l)\kappa_l^2 \vert{ D_l}\vert^{-1}(8\pi\vert{D_l}\vert+\hat{A}_lCap_l)\\
&= \vert D_{l} \vert^{-1} \Big(\frac{\rho_l}{\rho_l-\rho_0}+\frac{1}{8\pi}\left[-\kappa_l^2-\frac{\rho_{l}}{\rho_{0}}(\kappa_l^2-\kappa_0^2)\right]{\hat{A}_l} \Big)+ i \Big[\frac{\kappa_{l}^{3}}{4\pi}-\frac{1}{32\pi^{2}} (\kappa_{0}-\kappa_{l}) \kappa^{2}_{l} \vert D_{l} \vert^{-1} \hat{A}_{l} Cap_{l} \Big],
\end{split}
\label{def-of-Ilp}
\end{equation}
\begin{equation}
\begin{split}
F_l&:=-\frac{\rho_{0}}{\rho_{l}-\rho_{0}}(\kappa_l^2-\kappa_0^2)\vert{ D_l}\vert^{-1}\Big[ \int_{D_{l}}(x-z_l) dx\Big]
+\kappa_{0}^{2}\vert{ D_l}\vert^{-1}\Big[ \int_{D_{l}}(x-z_l) dx\Big]\\
&\qquad -\frac{1}{8\pi}\left[\frac{\rho_{0}}{\rho_{l}-\rho_{0}}(\kappa_l^2-\kappa_0^2)-\kappa_0^2\right] \vert{ D_l}\vert^{-1} R_l\\
&=F_l^\prime+O((\kappa_0-\kappa_l)\rho_la) \qquad (\text{using \eqref{obs1coeui}}),
\end{split}
\label{def-of-Fl}
\end{equation}
where
\begin{equation}
\begin{aligned}
F_l^\prime&:=
\kappa_l^2\vert{ D_l}\vert^{-1} \Big[\int_{D_{l}}(x-z_l) dx \Big]+\frac{\kappa_l^2}{8\pi} \vert{ D_l}\vert^{-1} R_l.
\end{aligned}
\label{def-of-Flp}
\end{equation}
We can observe that $J_l$ is not scaling. Also, let us denote the dominating term of $I_l^\prime$ by \[ I_l^{\prime^d}:= \vert D_{l} \vert^{-1} \Big(\frac{\rho_l}{\rho_l-\rho_0}-\frac{1}{8\pi}\kappa_l^2 \hat{A}_l \Big),\] and the remaining terms (of $O(1)$) by \[ I_l^{\prime^s}:=- \frac{1}{8\pi}\frac{\rho_{l}}{\rho_{0}}(\kappa_l^2-\kappa_0^2){\hat{A}_l} \vert D_{l} \vert^{-1} + i \Big[\frac{\kappa_{l}^{3}}{4\pi}-\frac{1}{32\pi^{2}} (\kappa_{0}-\kappa_{l}) \kappa^{2}_{l} \vert D_{l} \vert^{-1} \hat{A}_{l} Cap_{l} \Big].\] Then 
\begin{equation}
\begin{split}
\frac{1}{J_l}&=\left[J_l^r+iJ_{l}^i\right]^{-1}=\frac{1}{(J_l^r)^2}\left[J_l^r-iJ_{l}^i\right]\sum_{n=0}^\infty(-1)^n  \left(\frac{J_l^i}{J_l^r}\right)^{2n} \\
&=\frac{1}{J_l^r}\left[1-i\frac{J_{l}^i}{J_l^r}\right]\sum_{n=0}^\infty(-1)^n \left(\frac{J_l^i}{J_l^r}\right)^{2n}\\
&=\frac{1}{J_l^r}+O((\kappa_0-\kappa_l)a) \qquad[\mbox{since}\ J_l^i=O((\kappa_0-\kappa_l)a)]\\
&=\frac{1}{\kappa_l^2+O((\kappa_0-\kappa_l)\rho_l)}+O((\kappa_0-\kappa_l)a)
=\frac{1}{\kappa_l^2}+O((\kappa_0-\kappa_l)\rho_l)+O((\kappa_0-\kappa_l)a)\\ 
&=\frac{1}{\kappa_l^2}+O((\kappa_0-\kappa_l)a),
\end{split}
\label{inv-of-Jl-int}
\end{equation}
and hence
\begin{equation}
\begin{aligned}
\frac{F_l^\prime}{J_l}&=&\frac{F_l^\prime}{\kappa_l^2}+O(a^2).
\end{aligned}
\label{Fl-by-Jl}
\end{equation}
Also note that using \eqref{obs1coeui}, we have
\begin{equation}
\begin{aligned}
J_l^i&=\frac{\rho_{0}}{\rho_{l}-\rho_{0}}\Bigg[\frac{1}{4\pi}\kappa_l Cap_l(\kappa_l^2-\kappa_0^2)+\frac{1}{4\pi}(\kappa_0^3-\kappa_l^3)Cap_l-\frac{1}{4\pi}\kappa_0^2(\kappa_0-\kappa_l)Cap_l
\\
&\qquad -\frac{1}{32\pi^2}(\kappa_0-\kappa_l)\left[\frac{\rho_{0}}{\rho_{l}-\rho_{0}}(\kappa_l^2-\kappa_0^2)-\kappa_0^2\right] \vert{ D_l}\vert^{-1}  \left(8\pi\vert{D_l}\vert+\hat{A}_lCap_l\right)Cap_l\Bigg]\\
&=\frac{\rho_{0}}{\rho_{l}-\rho_{0}} \Big[\frac{1}{4 \pi} (\kappa_{0}-\kappa_{l}) \kappa_{l}^{2} Cap_{l}+\frac{1}{32 \pi^{2}} (\kappa_{0}-\kappa_{l}) \kappa_{l}^{2} \hat{A}_{l} \vert D_{l} \vert^{-1} Cap_{l}^{2} \Big] + O(a \rho_{l} (\kappa_{0}-\kappa_{l}))\\
&=J_l^{i^\prime}+O((\kappa_0-\kappa_l)a \rho_l),
\end{aligned}
\label{def-of-Jli}
\end{equation}
where
\begin{equation}
J_l^{i^\prime}:=\frac{\rho_{0}}{\rho_{l}-\rho_{0}} \frac{1}{4 \pi} (\kappa_{0}-\kappa_{l}) \kappa_{l}^{2} \Big[1+\frac{1}{8 \pi} \hat{A}_{l} \vert D_{l} \vert^{-1} Cap_{l} \Big] Cap_{l}.
\label{def-of-Jlip}
\end{equation}
Therefore, since $J_l^i=O((\kappa_0-\kappa_l)a))$, we can write
\begin{equation}
\begin{aligned}
\frac{I_l^\prime}{J_l}&=I_l^\prime\left[J_l^r+iJ_{l}^i\right]^{-1}=\frac{I_l^\prime}{(J_l^r)^2}\left[J_l^r-iJ_{l}^i\right]\sum_{n=0}^\infty(-1)^n  \left(\frac{J_l^i}{J_l^r}\right)^{2n} \\
&=\frac{I_l^{\prime^d}+I_l^{\prime^s}}{J_l^r}\left[1-i\frac{J_{l}^i}{J_l^r}\right]\sum_{n=0}^\infty(-1)^n \left(\frac{J_l^i}{J_l^r}\right)^{2n}\\
&=\frac{1}{[\kappa^{2}_{l}+\frac{\rho_l}{\rho_0}(\kappa_{l}^{2}-\kappa^{2}_{0})]} \Big[I_l^{\prime^d}-iI_l^{\prime^d} \frac{J_l^{i^\prime}}{[\kappa_l^2+\frac{\rho_{l}}{\rho_{0}}(\kappa_{l}^2-\kappa_{0}^2)]} \Big]+\frac{1}{\kappa^{2}_{l}}\Big[I_l^{\prime^s}-I^{\prime^d}_{l}\frac{(J^{i^\prime}_{l})^2}{\kappa_{l}^4} \Big]
+O((\kappa_0-\kappa_l)a),
\end{aligned} 
\label{Il-by-Jl}
\end{equation}
Using (\ref{obs1coeui}-\ref{def-of-Fl}), we can rewrite \eqref{mult-obs-109r6-pr3} as
\begin{equation}
\begin{aligned}
\frac{I_l^\prime}{J_l} Q_l +&\sum_{{m=1} \atop {m\neq\;l}}^{M}\Phi_{\kappa_0}(z_l,z_m)Q_m
+\frac{F_l^\prime}{J_l} \cdot\sum_{{m=1} \atop {m\neq\;l}}^{M}\nabla_x\Phi_{\kappa_0}(z_l,z_m)Q_m\\
&+\frac{1}{J_{l}} \Bigg[\kappa_0^2 -\frac{\rho_{0}}{\rho_{l}-\rho_{0}}(\kappa_l^2-\kappa_0^2)+\frac{\rho_{0}}{\rho_{l}-\rho_{0}}\frac{i}{4\pi}(\kappa_0^3-\kappa_l^3)
Cap_l\Bigg]\kappa_{0}^{2} \sum_{{m=1} \atop {m \neq l}}^{M} (\overline{z}_{m}-z_{m}) (\lambda_{m}-\frac{1}{2})^{-1} \nabla_t\Phi_{\kappa_0}(z_l,z_m)\cdot\\
&\qquad \Big[-\frac{1}{8 \pi} \hat{A}_{m} Q_{m} -\frac{1}{8 \pi} \Big(1-\frac{\rho_{m}}{\rho_{0}} \Big)^{-1} \hat{A}_{m} Q_{m}+ \vert D_{m} \vert \sum_{{n=1} \atop {n \neq m}}^{M} \Phi_{\kappa_{0}}(z_{m},z_{n}) Q_{n}\\
&\qquad \qquad -\frac{1}{8 \pi} \Big(1-\frac{\rho_{m}}{\rho_{0}} \Big)^{-1} \hat{A}_{m} Cap_{m} \sum_{{n=1} \atop {n \neq m}}^{M} \Phi_{\kappa_{0}}(z_{m},z_{n}) Q_{n}\\
&\qquad \qquad -\frac{1}{8\pi} \Big(1-\frac{\rho_{m}}{\rho_{0}} \Big)^{-1} \Big(\sum_{{n=1} \atop {n \neq m}}^{M} \Phi_{\kappa_{0}}(z_{m},z_{n}) Q_{n} \Big) \int_{\partial D_{m}} (A_{m}(s)-\hat{A}_{m}) (S^{0}_{D_{m}})^{-1}(1)(s)\ d\sigma_{m}(s) \Big] \\
=&-\frac{\kappa_l^2}{J_l} u^{I}(z_l)+\frac{1}{J_l} \Big[Er_6+O(a)+O(a^2\|\phi_l\|)+\sum^{M}_{{m=1} \atop {m\neq\;l}}O\left(\frac{a^3}{d^3_{ml}}\|\phi_m\|\right)+O(a^2\|\psi_l\|)+O\left(\sum_{{m=1} \atop {m \neq l}}^{M}\frac{1}{d^2_{ml}}\vert{V^{rem}_m}\vert\right)\\
&+ O((\kappa_{0}-\kappa_{l})\rho_{l})+ O((\kappa_{0}-\kappa_{l}) \rho_{l}^{2} \norm{\phi_{l}})+ O((\kappa_{0}-\kappa_{l})a \rho_{l} \norm{\phi_{l}})+O((\kappa_{0}-\kappa_{l}) a^{2} \rho_{l} (\sum_{{m=1} \atop {m \neq l}}^{M}\frac{1}{d^2_{ml}}\norm{\phi_{m}}))\Big].
\end{aligned}
\label{mult-obs-109r6-pr4int}
\end{equation}
Now let us assume that $\rho_{l}\simeq a^{1+\gamma}$, with $\gamma \geq 0$. Then from \eqref{V-rem}, we can deduce that 
\begin{equation}
\begin{aligned}
\sum_{{m=1} \atop {m \neq l}}^{M}\frac{1}{d^2_{ml}}\vert{V^{rem}_m}\vert &=O\Big(\frac{a^{3-\gamma}}{d^2}+\frac{a^{3-\gamma}}{d^{3\alpha}}\Big)\\
&+O\Big(\Big(\frac{a^3}{d^2}+\frac{a^3}{d^{3\alpha}}+\frac{a^{4-\gamma}}{d^2}+\frac{a^{4-\gamma}}{d^{3\alpha}}+\frac{a^4}{d^4}+\frac{a^4}{d^{2+3\alpha}}+\frac{a^4}{d^{6\alpha}}+\frac{a^{5-\gamma}}{d^4}+\frac{a^{5-\gamma}}{d^{2+3\alpha}}+\frac{a^{5-\gamma}}{d^{6\alpha}}\\
&\qquad \qquad +\frac{a^{6-\gamma}}{d^5}+\frac{a^{6-\gamma}}{d^{3+3\alpha}}+\frac{a^{6-\gamma}}{d^{6\alpha+1}} \Big)\norm{\phi}\Big).
\end{aligned}
\label{V-rem-er}
\end{equation}
%
Let us define $\mathbf{C}_l^{-1}:= \frac{1}{[\kappa^{2}_{l}+\frac{\rho_l}{\rho_0}(\kappa_{l}^{2}-\kappa^{2}_{0})]} \Big[I_l^{\prime^d}-iI_l^{\prime^d} \frac{J_l^{i^\prime}}{[\kappa_l^2+\frac{\rho_{l}}{\rho_{0}}(\kappa_{l}^2-\kappa_{0}^2)]} \Big]+\frac{1}{\kappa^{2}_{l}}\Big[I_l^{\prime^s}-I^{\prime^d}_{l}\frac{(J^{i^\prime}_{l})^2}{\kappa_{l}^4} \Big]$ . 
Then, using \eqref{obs1coeui},\eqref{inv-of-Jl-int} and \eqref{V-rem-er}, we can rewrite \eqref{mult-obs-109r6-pr4int} as
\begin{equation}
\begin{aligned}
\mathbf{C}_{l}^{-1} Q_l +&\sum_{{m=1} \atop {m\neq\;l}}^{M}\Phi_{\kappa_0}(z_l,z_m)Q_m
+\frac{F_l^\prime}{\kappa_{l}^{2}} \cdot\sum_{{m=1} \atop {m\neq\;l}}^{M}\nabla_x\Phi_{\kappa_0}(z_l,z_m)Q_m\\
&+\kappa_{0}^{2} \sum_{{m=1} \atop {m \neq l}}^{M} (\overline{z}_{m}-z_{m}) (\lambda_{m}-\frac{1}{2})^{-1} \nabla_t\Phi_{\kappa_0}(z_l,z_m)\cdot 
\Big[-\frac{1}{4 \pi} \hat{A}_{m} Q_{m} + \vert D_{m} \vert \sum_{{n=1} \atop {n \neq m}}^{M} \Phi_{\kappa_{0}}(z_{m},z_{n}) Q_{n}\\
&\qquad \qquad -\frac{1}{8\pi} \Big(1-\frac{\rho_{m}}{\rho_{0}} \Big)^{-1} \Big(\sum_{{n=1} \atop {n \neq m}}^{M} \Phi_{\kappa_{0}}(z_{m},z_{n}) Q_{n} \Big) \int_{\partial D_{m}} A_{m}(s) (S^{0}_{D_{m}})^{-1}(1)(s)\ d\sigma_{m}(s) \Big]
\end{aligned}
\label{mult-obs-109r6-pr4}
\end{equation}
\begin{equation}
\begin{aligned}
&=-u^{I}(z_{l})+O\Big(a+\frac{a^{3-\gamma}}{d^{2}}+\frac{a^{3-\gamma}}{d^{3\alpha}}+\Big(a^{2}+\frac{a^3}{d^3}+\frac{a^3}{d^{3\alpha+1}}+\frac{a^{4-\gamma}}{d^2}+\frac{a^{4-\gamma}}{d^{3\alpha}}+\frac{a^4}{d^4}+\frac{a^4}{d^{2+3\alpha}}+\frac{a^4}{d^{6\alpha}}\\
&\qquad \qquad \qquad \qquad \qquad \qquad \qquad \qquad+\frac{a^{5-\gamma}}{d^4}+\frac{a^{5-\gamma}}{d^{2+3\alpha}}+\frac{a^{5-\gamma}}{d^{6\alpha}}+\frac{a^{6-\gamma}}{d^5}+\frac{a^{6-\gamma}}{d^{3+3\alpha}}+\frac{a^{6-\gamma}}{d^{6\alpha+1}} \Big)\norm{\phi} \Big). 
\end{aligned}
\notag
\end{equation}
\begin{remark}\label{sign}
We note that the dominant part of $\mathbf{C}_l^{-1} $ is given by $\frac{I_{l}^{\prime^d}}{\kappa_{l}^{2}}= \frac{\vert D_{l} \vert^{-1}}{\kappa_{l}^{2}} \left[\frac{\rho_{l}}{\rho_{l}-\rho_{0}}-\frac{1}{8 \pi} \kappa_{l}^{2} \hat{A}_{l} \right]$. When $\gamma<1 $ or the frequency is away from the resonance, the sign of the real part $(\mathbf{C}_l^{-1})^{r} $ of $\mathbf{C}_l^{-1} $ is given by the sign of the term $\frac{\vert D_{l} \vert^{-1}}{\kappa_{l}^{2}} \cdot \frac{\rho_{l}}{\rho_{l}-\rho_{0}} $ which is negative. Therefore in this case, $(\mathbf{C}_l^{-1})^{r}<0,\ \forall \ l=1,\dots,M. $\\
When the frequency $\omega $ is near the resonance, we can write $1-\left(\frac{\omega_{M}}{\omega} \right)^{2}=l_{M} a^{h_{1}} $. In this case, we can write
\begin{equation}
\begin{aligned}
\frac{I_{l}^{\prime^d}}{\kappa_{l}^{2}}&=\frac{\hat{A}_{l} \vert D_l \vert^{-1} \frac{\rho_{l}}{k_l}}{8 \pi} \left[\frac{\omega_{M}^{2}}{\omega^{2}}-1 \right]=\frac{\hat{A}_{l} \vert D_l \vert^{-1} \frac{\rho_{l}}{k_l}}{8 \pi} \Big[-l_{M} a^{h_{1}}\Big].
\end{aligned}
\notag
\end{equation}
Therefore if $l_{M}>0 $, $(\mathbf{C}_l^{-1})^{r}>0,\ \forall\ l=1,\dots,M $  and if $l_{M}<0 $, $(\mathbf{C}_l^{-1})^{r}<0,\ \forall\ l=1,\dots,M $.
\QEDB
\end{remark}

 \begin{lemma}\label{Mazyawrkthmela}
 For $m=1,\cdots,M $, let us define 
 \begin{equation}
 \begin{aligned}
 &Y_{m}:=-u^{I}(z_{m})+O\Big(a+\frac{a^{3-\gamma}}{d^{2}}+\frac{a^{3-\gamma}}{d^{3\alpha}}+\Big(a^{2}+\frac{a^3}{d^3}+\frac{a^3}{d^{3\alpha+1}}+\frac{a^{4-\gamma}}{d^2}+\frac{a^{4-\gamma}}{d^{3\alpha}}+\frac{a^4}{d^4}+\frac{a^4}{d^{2+3\alpha}}+\frac{a^4}{d^{6\alpha}}\\
&\qquad \qquad \qquad \qquad \qquad \qquad \qquad \qquad+\frac{a^{5-\gamma}}{d^4}+\frac{a^{5-\gamma}}{d^{2+3\alpha}}+\frac{a^{5-\gamma}}{d^{6\alpha}}+\frac{a^{6-\gamma}}{d^5}+\frac{a^{6-\gamma}}{d^{3+3\alpha}}+\frac{a^{6-\gamma}}{d^{6\alpha+1}} \Big)\norm{\phi} \Big).
 \end{aligned}
 \notag
 \end{equation}
Then the algebraic system \eqref{mult-obs-109r6-pr4} is invertible provided 
\begin{itemize}
\item in the case when $(\mathbf{C}_l^{-1})^{r}>0,\ \forall\ l=1,\dots,M $,
\begin{equation}
\begin{aligned}
& \frac{\min\limits_{1\leq m\leq M}{\mathbf{C}_m^r}}{(\max\limits_{1\leq m\leq M} \vert{\mathbf{C}_m}\vert)^2} \geq \frac{3\tau}{5\pi\,d}+\max\limits_{1\leq{m}\leq{M}} \left\vert\frac{F_m^\prime}{\kappa_m^2}\right\vert C \sqrt{M M_{max}}  \left[\frac{1}{ d^4 }+\frac{1}{ d^{5\alpha}}\right]^\frac{1}{2}\\
&\qquad \qquad \qquad +C a^{2-\gamma} \sqrt{M M_{max}}  \left[\frac{1}{ d^4 }+\frac{1}{ d^{5\alpha}}\right]^\frac{1}{2}+C a^{3-\gamma} M M_{max} \left[\frac{1}{ d^2 }+\frac{1}{ d^{3\alpha}}\right]^\frac{1}{2}\left[\frac{1}{ d^4 }+\frac{1}{ d^{5\alpha}}\right]^\frac{1}{2},
\end{aligned}
\label{condition-inv}
\end{equation}
\item in the case when $(\mathbf{C}_l^{-1})^{r}<0,\ \forall\ l=1,\dots,M $,
\begin{equation}
\begin{aligned}
&\frac{\min\limits_{1\leq m\leq M}{\vert \mathbf{C}_m^r \vert}}{(\max\limits_{1\leq m\leq M} \vert{\mathbf{C}_m}\vert)^2} \geq C \sqrt{M M_{max}} \left[\frac{1}{ d^2 }+\frac{1}{ d^{3\alpha}}\right]^\frac{1}{2}+\max\limits_{1\leq{m}\leq{M}} \left\vert\frac{F_m^\prime}{\kappa_m^2}\right\vert C \sqrt{M M_{max}}  \left[\frac{1}{ d^4 }+\frac{1}{ d^{5\alpha}}\right]^\frac{1}{2}\\
&\qquad \qquad \qquad +C a^{2-\gamma} \sqrt{M M_{max}}  \left[\frac{1}{ d^4 }+\frac{1}{ d^{5\alpha}}\right]^\frac{1}{2}+C a^{3-\gamma} M M_{max} \left[\frac{1}{ d^2 }+\frac{1}{ d^{3\alpha}}\right]^\frac{1}{2}\left[\frac{1}{ d^4 }+\frac{1}{ d^{5\alpha}}\right]^\frac{1}{2},
\end{aligned}
\label{condition-inv-2}
\end{equation}
\end{itemize}
and the solution vector $Q_m,\ m=1, ..., M,$  satisfies either the estimate
\begin{equation}
\begin{aligned}
 \sum_{m=1}^{M}\vert{Q_m}\vert^{2}
&\leq 4 \ (\max\limits_{1\leq m\leq M} \vert{\mathbf{C}_m}\vert)^2 \sum_{m=1}^{M}\vert{Y_m}\vert^{2} \\
&\quad \Big( \frac{\min\limits_{1\leq m\leq M}{\vert \mathbf{C}_m^r \vert}}{\max\limits_{1\leq m\leq M} \vert{\mathbf{C}_m}\vert} -\Big(\frac{3\tau}{5\pi\,d}+\max\limits_{1\leq{m}\leq{M}} \left\vert\frac{F_m^\prime}{\kappa_m^2}\right\vert C \sqrt{M M_{max}}  \left[\frac{1}{ d^4 }+\frac{1}{ d^{5\alpha}}\right]^\frac{1}{2}\\
&\quad +C a^{2-\gamma} \sqrt{M M_{max}}  \left[\frac{1}{ d^4 }+\frac{1}{ d^{5\alpha}}\right]^\frac{1}{2}+C a^{3-\gamma} M M_{max} \left[\frac{1}{ d^2 }+\frac{1}{ d^{3\alpha}}\right]^\frac{1}{2}\left[\frac{1}{ d^4 }+\frac{1}{ d^{5\alpha}}\right]^\frac{1}{2}  \Big) \max\limits_{1\leq m\leq M} \vert{\mathbf{C}_m}\vert \Big)^{-2},
\end{aligned}
\label{mazya-fnlinvert-small-ela-2}
\end{equation}
where $\tau:=\min_{l\neq\,m}\cos(\kappa\vert{z_m-z_l}\vert)\geq\;0$, or the estimate
\begin{equation}
\begin{aligned}
 \sum_{m=1}^{M}\vert{Q_m}\vert^{2}
&\leq 4 \ (\max\limits_{1\leq m\leq M} \vert{\mathbf{C}_m}\vert)^2 \sum_{m=1}^{M}\vert{Y_m}\vert^{2} \\
&\quad \Big( \frac{\min\limits_{1\leq m\leq M}{\vert \mathbf{C}_m^r \vert}}{\max\limits_{1\leq m\leq M} \vert{\mathbf{C}_m}\vert} -\Big(C \sqrt{M M_{max}} \left[\frac{1}{ d^2 }+\frac{1}{ d^{3\alpha}}\right]^\frac{1}{2}+\max\limits_{1\leq{m}\leq{M}} \left\vert\frac{F_m^\prime}{\kappa_m^2}\right\vert C \sqrt{M M_{max}}  \left[\frac{1}{ d^4 }+\frac{1}{ d^{5\alpha}}\right]^\frac{1}{2}\\
&\quad +C a^{2-\gamma} \sqrt{M M_{max}}  \left[\frac{1}{ d^4 }+\frac{1}{ d^{5\alpha}}\right]^\frac{1}{2}+C a^{3-\gamma} M M_{max} \left[\frac{1}{ d^2 }+\frac{1}{ d^{3\alpha}}\right]^\frac{1}{2}\left[\frac{1}{ d^4 }+\frac{1}{ d^{5\alpha}}\right]^\frac{1}{2}  \Big) \max\limits_{1\leq m\leq M} \vert{\mathbf{C}_m}\vert \Big)^{-2},
\end{aligned}
\label{mazya-fnlinvert-small-ela-3}
\end{equation}
where $C$ is a constant (depending on the Lipschitz characters of the obstacles) uniformly bounded with respect to $a$. 
%
\end{lemma}
\begin{proof}
We refer to section \ref{alg-sys} for a proof of this lemma.
\end{proof}
\begin{remark}
From \eqref{mazya-fnlinvert-small-ela-2}, we conclude that
\begin{equation}
\begin{aligned}
 \sum_{m=1}^{M}\vert{Q_m}\vert
&\leq 2 M \ (\max\limits_{1\leq m\leq M} \vert{\mathbf{C}_m}\vert) \max_{1\leq\,m\leq\,M}\vert{Y_m}\vert \\
&\quad \Big( \frac{\min\limits_{1\leq m\leq M}{\vert \mathbf{C}_m^r \vert}}{\max\limits_{1\leq m\leq M} \vert{\mathbf{C}_m}\vert} -\Big(\frac{3\tau}{5\pi\,d}+\max\limits_{1\leq{m}\leq{M}} \left\vert\frac{F_m^\prime}{\kappa_m^2}\right\vert C \sqrt{M M_{max}}  \left[\frac{1}{ d^4 }+\frac{1}{ d^{5\alpha}}\right]^\frac{1}{2}\\
&\quad +C a^{2-\gamma} \sqrt{M M_{max}}  \left[\frac{1}{ d^4 }+\frac{1}{ d^{5\alpha}}\right]^\frac{1}{2}+C a^{3-\gamma} M M_{max} \left[\frac{1}{ d^2 }+\frac{1}{ d^{3\alpha}}\right]^\frac{1}{2}\left[\frac{1}{ d^4 }+\frac{1}{ d^{5\alpha}}\right]^\frac{1}{2}  \Big) \max\limits_{1\leq m\leq M} \vert{\mathbf{C}_m}\vert \Big)^{-1},
\end{aligned}
\label{L1-norm-estimate-algebraic-system}
\end{equation}
while \eqref{mazya-fnlinvert-small-ela-3} yields
\begin{equation}
\begin{aligned}
 \sum_{m=1}^{M}\vert{Q_m}\vert
&\leq 2 M \ (\max\limits_{1\leq m\leq M} \vert{\mathbf{C}_m}\vert) \max_{1\leq\,m\leq\,M}\vert{Y_m}\vert \\
&\quad \Big( \frac{\min\limits_{1\leq m\leq M}{\vert \mathbf{C}_m^r \vert}}{\max\limits_{1\leq m\leq M} \vert{\mathbf{C}_m}\vert} -\Big(C \sqrt{M M_{max}} \left[\frac{1}{ d^2 }+\frac{1}{ d^{3\alpha}}\right]^\frac{1}{2}+\max\limits_{1\leq{m}\leq{M}} \left\vert\frac{F_m^\prime}{\kappa_m^2}\right\vert C \sqrt{M M_{max}}  \left[\frac{1}{ d^4 }+\frac{1}{ d^{5\alpha}}\right]^\frac{1}{2}\\
&\quad +C a^{2-\gamma} \sqrt{M M_{max}}  \left[\frac{1}{ d^4 }+\frac{1}{ d^{5\alpha}}\right]^\frac{1}{2}+C a^{3-\gamma} M M_{max} \left[\frac{1}{ d^2 }+\frac{1}{ d^{3\alpha}}\right]^\frac{1}{2}\left[\frac{1}{ d^4 }+\frac{1}{ d^{5\alpha}}\right]^\frac{1}{2}  \Big) \max\limits_{1\leq m\leq M} \vert{\mathbf{C}_m}\vert \Big)^{-1}.
\end{aligned}
\label{L1-norm-estimate-algebraic-system-2}
\end{equation}
\QEDB
\end{remark}
\begin{remark}\label{extra-conditions}
We note that if 
\begin{equation}
\begin{aligned}
&(\max\limits_{1\leq m\leq M} \vert{\mathbf{C}_m}\vert)\Big[ \frac{3\tau}{5 \pi d}+\max\limits_{1\leq{m}\leq{M}} \left\vert\frac{F_m^\prime}{\kappa_m^2}\right\vert C \sqrt{M M_{max}}  \left[\frac{1}{ d^4 }+\frac{1}{ d^{5\alpha}}\right]^\frac{1}{2}\\
&\qquad \qquad \qquad +C a^{2-\gamma} \sqrt{M M_{max}}  \left[\frac{1}{ d^4 }+\frac{1}{ d^{5\alpha}}\right]^\frac{1}{2}+C a^{3-\gamma} M M_{max} \left[\frac{1}{ d^2 }+\frac{1}{ d^{3\alpha}}\right]^\frac{1}{2}\left[\frac{1}{ d^4 }+\frac{1}{ d^{5\alpha}}\right]^\frac{1}{2}\Big]=O(1),
\end{aligned}
\notag
\notag
\end{equation}
then the condition \eqref{condition-inv} holds. Since in this case the frequency is near the resonance $\omega_M$ (and $l_M>0 $), this is equivalent to the condition
\begin{equation}
a^{1-h_{1}}\cdot O(a^{-t}+a^{1-s-t}+a^{2-\gamma-s-t}+a^{3-\gamma-2s-t} )=O(1).
\label{lmg0}
\end{equation}
This leads to the additional conditions $2-h_{1}-s-t \geq 0 $ and $1-t-h_{1}\geq 0 $ provided $\gamma+s \leq 2$. As in this case $\gamma =1$ and hence $s\leq 1$, the conditions reduce to $1-t-h_{1}\geq 0 $ and $s\leq 1$. Note that \eqref{lmg0} also gives rise to the condition $2(a) $ in Theorem \ref{Main-theorem}.\\
Similarly if 
\begin{equation}
\begin{aligned}
&(\max\limits_{1\leq m\leq M} \vert{\mathbf{C}_m}\vert)\Big[ C \sqrt{M M_{max}} \left[\frac{1}{ d^2 }+\frac{1}{ d^{3\alpha}}\right]^\frac{1}{2}+\max\limits_{1\leq{m}\leq{M}} \left\vert\frac{F_m^\prime}{\kappa_m^2}\right\vert C \sqrt{M M_{max}}  \left[\frac{1}{ d^4 }+\frac{1}{ d^{5\alpha}}\right]^\frac{1}{2}\\
&\qquad \qquad \qquad +C a^{2-\gamma} \sqrt{M M_{max}}  \left[\frac{1}{ d^4 }+\frac{1}{ d^{5\alpha}}\right]^\frac{1}{2}+C a^{3-\gamma} M M_{max} \left[\frac{1}{ d^2 }+\frac{1}{ d^{3\alpha}}\right]^\frac{1}{2}\left[\frac{1}{ d^4 }+\frac{1}{ d^{5\alpha}}\right]^\frac{1}{2}\Big]=O(1),
\end{aligned}
\notag
\end{equation}
then the condition \eqref{condition-inv-2} holds true. When the frequency is away from the resonance $\omega_M$, and as we take $3\alpha t=s$ with $\alpha \in (0, 1]$, this is equivalent to the condition
\begin{equation}
a^{2-\gamma}\cdot O(a^{-s}+a^{1-s-t}+a^{2-\gamma-s-t}+a^{3-\gamma-2s-t} )=O(1),
\notag
\end{equation}
which holds if $\gamma+s\leq 2, \frac{s}{3}\leq t \leq 1$ and $ 0 \leq \gamma \leq 1$. This also gives rise to the condition $1(a) $ in Theorem \ref{Main-theorem}. \\
If the frequency is near the resonance $\omega_M$ (and $l_M<0 $), and as we take $3\alpha t=s$ with $\alpha \in (0, 1]$, then the condition is equivalent to
\begin{equation}
a^{1-h_{1}}\cdot O(a^{-s}+a^{1-s-t}+a^{2-\gamma-s-t}+a^{3-\gamma-2s-t} )=O(1),
\label{lml0}
\end{equation}
which holds provided $1-h_{1}-s \geq 0,\; \gamma+s\leq 2, 0\leq t \leq  1, 0 \leq \gamma \leq 1$ and $ \frac{s}{3}\leq t $. Again as here also $\gamma=1$ and hence $s\leq 1$, these conditions
reduce to $1-h_{1}-s \geq 0$ and $\frac{s}{3}\leq t\leq 1$. Also note that \eqref{lml0} gives rise to the condition $1(b) $ in Theorem \ref{Main-theorem}. \\
The condition $2(b) $ follows from the fact that a condition similar to \eqref{lml0} can also be derived when the coefficients ${\bf{C_m}}$ are positive.
\QEDB
\end{remark}
%
%
From \eqref{L1-norm-estimate-algebraic-system} or \eqref{L1-norm-estimate-algebraic-system-2} and the definition of $Y_{m}$, we can derive the a priori estimate
  \begin{align} \label{alg-system-koneqkl-Q-generalshape}
 \sum_{m=1}^M\vert{Q_m}\vert
 &= O\Big(M\max\vert{\mathbf{C}_m}\vert \Big[1+a+\frac{a^{3-\gamma}}{d^{2}}+\frac{a^{3-\gamma}}{d^{3\alpha}}+\Big(a^{2}+\frac{a^3}{d^3}+\frac{a^3}{d^{3\alpha+1}}+\frac{a^{4-\gamma}}{d^2}+\frac{a^{4-\gamma}}{d^{3\alpha}}+\frac{a^4}{d^4}+\frac{a^4}{d^{2+3\alpha}}+\frac{a^4}{d^{6\alpha}}\nonumber\\
&\qquad \qquad \qquad \qquad+\frac{a^{5-\gamma}}{d^4}+\frac{a^{5-\gamma}}{d^{2+3\alpha}}+\frac{a^{5-\gamma}}{d^{6\alpha}}+\frac{a^{6-\gamma}}{d^5}+\frac{a^{6-\gamma}}{d^{3+3\alpha}}+\frac{a^{6-\gamma}}{d^{6\alpha+1}} \Big)\norm{\phi} \Big]\Big)\nonumber\\
&= O\Big(M\max\vert{\mathbf{C}_m}\vert \Big[1+a+a^{3-\gamma-2t}+a^{3-\gamma-s}+\Big(a^{2}+a^{3-3t}+a^{3-s-t}\nonumber\\
&\qquad \qquad \qquad \qquad+a^{4-\gamma-2t}+a^{4-\gamma-s}+a^{4-4t}+a^{4-s-2t}+a^{4-2s}\nonumber\\
&\qquad \qquad \qquad \qquad+a^{5-\gamma-4t}+a^{5-\gamma-s-2t}+a^{5-\gamma-2s}+a^{6-\gamma-5t}+a^{6-\gamma-s-3t}+a^{6-\gamma-2s-t} \Big)\norm{\phi} \Big]\Big)\nonumber\\
&=O\Big(M\max\vert{\mathbf{C}_m}\vert \Big[1+a+\left(a^2+a^{3-3t}+a^{3-s-t}+a^{4-2s} \right)\norm{\phi} \Big] \Big),
  \end{align}
assuming that $0\leq t < \frac{1}{2}, \; 0 \leq s \leq \frac{3}{2},\; 0 \leq \gamma \leq 1$ and $s+\gamma \leq 2$.\\
Using \eqref{alg-system-koneqkl-Q-generalshape}, we also note that
\begin{align}\label{term-1}
\frac{F_l^\prime}{\kappa_l^2}\cdot\sum_{m\neq\;l}\nabla_x\Phi_{\kappa_0}(z_l,z_m)Q_m &=O\Big(\frac{a}{d^2} \sum_{m=1}^{M} \vert Q_m \vert \Big)\\
&=O\left(a^{1-2t} M\max\vert{\mathbf{C}_m}\vert \Big[1+a+\left(a^2+a^{3-3t}+a^{3-s-t}+a^{4-2s} \right)\norm{\phi} \Big]\right),\nonumber
\end{align}
\begin{align}\label{term-2}
&-\frac{1}{4 \pi} \kappa_{0}^{2} \sum_{{m=1} \atop {m \neq l}}^{M} (\overline{z}_{m}-z_{m}) (\lambda_{m}-\frac{1}{2})^{-1} \nabla_t\Phi_{\kappa_0}(z_l,z_m) \hat{A}_{m} Q_{m}\\
&=\Big(\frac{a^{2-\gamma}}{d^2} \sum_{m=1}^{M} \vert Q_m \vert \Big)
=O\left(a^{2-\gamma-2t} M\max\vert{\mathbf{C}_m}\vert \Big[1+a+\left(a^2+a^{3-3t}+a^{3-s-t}+a^{4-2s} \right)\norm{\phi} \Big] \right),\nonumber
\end{align}
{ and }
\begin{align}\label{term-3}
&\kappa_{0}^{2} \sum_{{m=1} \atop {m \neq l}}^{M} (\overline{z}_{m}-z_{m}) (\lambda_{m}-\frac{1}{2})^{-1} \nabla_t\Phi_{\kappa_0}(z_l,z_m)\cdot 
\Big[\vert D_{m} \vert \sum_{{n=1} \atop {n \neq m}}^{M} \Phi_{\kappa_{0}}(z_{m},z_{n}) Q_{n}\\
&\qquad \qquad -\frac{1}{8\pi} \Big(1-\frac{\rho_{m}}{\rho_{0}} \Big)^{-1} \Big(\sum_{{n=1} \atop {n \neq m}}^{M} \Phi_{\kappa_{0}}(z_{m},z_{n}) Q_{n} \Big) \int_{\partial D_{m}} A_{m}(s) (S^{0}_{D_{m}})^{-1}(1)(s)\ d\sigma_{m}(s) \Big]\nonumber\\
&\qquad \qquad \qquad=O\left(a^{3-\gamma} \left(\sum_{m \neq l} \frac{1}{d^{2}_{ml}}\right) \frac{1}{d} \sum_{n=1}^{M} \vert Q_n \vert \right)=O\left(a^{3-\gamma} \left(\frac{1}{d^3}+\frac{1}{d^{3\alpha+1}} \right)  \sum_{n=1}^{M} \vert Q_n \vert\right)\nonumber\\
&\qquad \qquad \qquad=O\Big(a^{3-\gamma-s-t} M\max\vert{\mathbf{C}_m}\vert \Big[1+a+\left(a^2+a^{3-3t}+a^{3-s-t}+a^{4-2s} \right)\norm{\phi} \Big]\Big).\nonumber
\end{align}
Therefore we can rewrite \eqref{mult-obs-109r6-pr4} as
  \begin{align}\label{mult-obs-109r6-pr5-smpl}
&\mathbf{C}_l^{-1}Q_l +\sum_{{m=1} \atop {m\neq\;l}}^{M}\Phi_{\kappa_0}(z_l,z_m)Q_m =-u^{I}(z_l)+O \left(a+a^{3-\gamma-s}+ \left(a^2+a^{3-3t}+a^{3-s-t}+a^{4-2s} \right)\norm{\phi} \right)\\
&+O\Big((a^{1-2t}+a^{2-\gamma-2t}+a^{3-\gamma-s-t}) M\max\vert{\mathbf{C}_m}\vert \Big[1+a+\left(a^2+a^{3-3t}+a^{3-s-t}+a^{4-2s} \right)\norm{\phi} \Big]\Big).\nonumber
\end{align}
%
 %
%
%
\subsection{The final approximations}
\begin{proposition}\label{Prop-phi-estimate} Let the parameters $t, s, \gamma$ satisfy the conditions $0\leq t < \frac{1}{2}, \; 0 \leq s \leq \frac{3}{2},\; 0 \leq \gamma \leq 1, \ s+\gamma \leq 2$, $\frac{s}{3}\leq t $ 
and let $ M\max\vert{\mathbf{C}_m}\vert=O(a^{-h}),\;  h < \frac{1}{2}$. Then for every $l$, we have
  $$\Vert \phi_l\Vert =O(a^{-\gamma})+O(a^{-\gamma-h}).$$
 \end{proposition}
 \begin{proof}
 We refer to section \ref{phi-estimate} for a proof of this result.
 \end{proof}
 \begin{remark}\label{mcm}
 In case $\gamma<1 $ or when the frequency is away from the resonance, we have $\max\vert{\mathbf{C}_m}\vert=O(a^{2-\gamma}) $ whence it follows using 
 $s+\gamma\leq 2 $ that $M\max\vert{\mathbf{C}_m}\vert=O(a^{2-\gamma-s})=O(1) $. Therefore in this case, we have $h\leq 0 $ and then $\Vert \phi_l\Vert =O(a^{-\gamma}) $.\\
 In contrast with this case, near the resonance $\max_{l} \vert \mathbf{C}_{l} \vert= O(a^{1-h_{1}}) $ where $h_{1} \geq 0 $. Therefore $M \max_{l} \vert \mathbf{C}_{l} \vert=O(a^{1-s-h_{1}})=O(a^{-h}) $.
Now if $l_M<0 $, we need the condition $1-h_{1}-s\geq 0 $, see Remark \ref{extra-conditions}, which implies that $h\leq 0 $ and then $\Vert \phi_l\Vert =O(a^{-\gamma}) $. But if $l_M>0 $, we have only the condition 
$1-h_1\geq t$ and $s\leq 1$ and it is possible to have $-1+s+h_{1}>0$ and hence $h>0 $.
  \QEDB
 \end{remark}
Combining \eqref{mult-obs-109r6-pr5-smpl} and Proposition \ref{Prop-phi-estimate}, we deduce that under the conditions 
$0\leq t < \frac{1}{2}, \; 0 \leq s \leq \frac{3}{2},\; 0 \leq \gamma \leq 1, s+\gamma \leq 2$ and $1-2t-h>0 $ with $ h < \frac{1}{2} $, \footnote{Here, we do the computations when $h\geq0$ in which case 
$\Vert \phi_l\Vert =O(a^{-\gamma-h})$. For the case $h<0$, we have $\Vert \phi_l\Vert =O(a^{-\gamma})$.} the vectors $(Q_l)^M_{l=1}$ satisfy

\begin{equation}
\begin{aligned}
&\mathbf{C}_l^{-1}Q_l +\sum_{{m=1}\atop {m\neq\;l}}^{M}\Phi_{\kappa_0}(z_l,z_m)Q_m 
   =-u^{I}(z_l)+O \left(a+a^{2-\gamma-h}+a^{3-\gamma-3t-h}+a^{3-\gamma-s-t-h}+a^{4-\gamma-2s-h}  \right)\\
&+O\Big((a^{1-2t}+a^{2-\gamma-2t}+a^{3-\gamma-s-t}) a^{-h} \Big[1+a+a^{2-\gamma-h}+a^{3-\gamma-h-3t}+a^{3-\gamma-h-s-t}+a^{4-\gamma-h-2s}  \Big]\Big)\\
&=-u^{I}(z_l)+O \Big(a+a^{2-\gamma-h}+a^{3-\gamma-h-3t}+a^{3-\gamma-h-s-t}+a^{4-\gamma-h-2s}+
a^{1-2t-h} \Big)\\
&=-u^{I}(z_l)+O \Big(a^{4-\gamma-h-2s}+a^{1-2t-h} \Big),
\end{aligned}
\label{mult-obs-109r6-pr5-smpl-final}
\end{equation}
 where 
 \begin{equation}
  \mathbf{C}_l^{-1}= \frac{1}{[\kappa^{2}_{l}+\frac{\rho_l}{\rho_0}(\kappa_{l}^{2}-\kappa^{2}_{0})]} \Big[I_l^{\prime^d}-iI_l^{\prime^d} \frac{J_l^{i^\prime}}{[\kappa_l^2+\frac{\rho_{l}}{\rho_{0}}(\kappa_{l}^2-\kappa_{0}^2)]} \Big]+\frac{1}{\kappa^{2}_{l}}\Big[I_l^{\prime^s}-I^{\prime^d}_{l}\frac{(J^{i^\prime}_{l})^2}{\kappa_{l}^4} \Big]
  \end{equation} 
   with
  \begin{equation}
  \begin{aligned}
I_l^{\prime} :=I_l^{\prime^d}+I_l^{\prime^s} :=\vert D_{l} \vert^{-1} \Big(\frac{\rho_l}{\rho_l-\rho_0}+\frac{1}{8\pi}\left[-\kappa_l^2-\frac{\rho_{l}}{\rho_{0}}(\kappa_l^2-\kappa_0^2)\right]{\hat{A}_l} \Big)+ i \Big[\frac{\kappa_{l}^{3}}{4\pi}-\frac{1}{32\pi^{2}} (\kappa_{0}-\kappa_{l}) \kappa^{2}_{l} \vert D_{l} \vert^{-1} \hat{A}_{l} Cap_{l} \Big]
\end{aligned}
\notag
\end{equation}
and 
\[
J_l^{i^\prime}=\frac{\rho_{0}}{\rho_{l}-\rho_{0}} \frac{1}{4 \pi} (\kappa_{0}-\kappa_{l}) \kappa_{l}^{2} \Big[1+\frac{1}{8 \pi} \hat{A}_{l} \vert D_{l} \vert^{-1} Cap_{l} \Big] Cap_{l}.
\]
We recall that we have the representation of the scattered field of the form $u^s(x, \theta)=\sum^M_{m=1}S^{\kappa_{0}}_{D_m}\phi_m $.

From this we can deduce that the far-field can be approximated as
\[
u^\infty(\hat{x}, \theta)=\sum^M_{m=1}\Phi_{\kappa_{0}}^\infty(\hat{x}, z_m)Q_m +\sum_{m=1}^{M} \nabla\Phi_{\kappa_{0}}^\infty(\hat{x}, z_m)\cdot V_{m} + O(M  a^{3-\gamma}).
\]
Now under the assumptions on $t,s,\gamma,h $, using \eqref{V-rem}, we see that
\begin{align*}
\sum_{m=1}^{M} \nabla\Phi_{\kappa_{0}}^\infty(\hat{x}, z_m)\cdot V^{rem}_{m}&= O(M \vert V^{rem}_{m}\vert)\\
&=O\Big(a^{3-\gamma-s}+a^{4-2\gamma-s-h}+a^{5-2\gamma-s-h-2t}+a^{5-2\gamma-h-2s}\\
&\qquad \qquad +a^{4-\gamma-s-h-2t}+a^{4-\gamma-h-2s}+a^{6-2\gamma-s-h-3t}+a^{6-2\gamma-h-2s-t}\Big)\\
&=O\left(a^{3-\gamma-s}+a^{4-2\gamma-s-h}+a^{1-h}+a^{2-s-h} \right).
\end{align*}
Similarly
\begin{align*}
\sum_{m=1}^{M} \nabla\Phi_{\kappa_{0}}^\infty(\hat{x}, z_m)\cdot V^{dom}_{m,1}&= -\frac{\kappa_{0}^{2}}{8 \pi} \sum_{m=1}^{M} \nabla\Phi_{\kappa_{0}}^\infty(\hat{x}, z_m) \left(1+ \Big(1-\frac{\rho_m}{\rho_0} \Big)^{-1}\right) (\overline{z}_m-z_m) \left(\lambda_m-\frac{1}{2}\right)^{-1} \hat{A}_m Q_m  \\
&=O\left(a^{2-\gamma} \sum_{m=1}^{M} \nabla\Phi_{\kappa_{0}}^\infty(\hat{x}, z_m) Q_{m} \right),
\end{align*}

\begin{align*}
&\sum_{m=1}^{M} \nabla\Phi_{\kappa_{0}}^\infty(\hat{x}, z_m)\cdot V^{dom}_{m,2}=\kappa_{0}^{2} \sum_{{m=1} \atop {m \neq l}}^{M} (\overline{z}_{m}-z_{m}) (\lambda_{m}-\frac{1}{2})^{-1} \nabla\Phi^{\infty}_{\kappa_0}(\hat{x},z_m)\cdot \\
&\qquad \qquad \qquad \qquad \Big[ \vert D_{m}\vert -\frac{1}{8\pi} \Big(1-\frac{\rho_{m}}{\rho_{0}} \Big)^{-1}  \int_{\partial D_{m}} A_{m}(s) (S^{0}_{D_{m}})^{-1}(1)(s)\ d\sigma_{m}(s)  \Big] \left(\sum_{{n=1} \atop {n \neq m}}^{M} \Phi_{\kappa_{0}}(z_{m},z_{n}) Q_{n}\right)\\
&=\kappa_{0}^{2} \sum_{{m=1} \atop {m \neq l}}^{M} (\overline{z}_{m}-z_{m}) (\lambda_{m}-\frac{1}{2})^{-1} \cdot \nabla\Phi^{\infty}_{\kappa_0}(\hat{x},z_m)
\Big[ \vert D_{m}\vert -\frac{1}{8\pi} \Big(1-\frac{\rho_{m}}{\rho_{0}} \Big)^{-1}  \int_{\partial D_{m}} A_{m}(s) (S^{0}_{D_{m}})^{-1}(1)(s)\ d\sigma_{m}(s)  \Big]\\
&\qquad \qquad \Big(-\mathbf{C}_{m}^{-1} Q_{m}-u^{I}(z_m)+O \left(a^{4-\gamma-h-2s}+a^{1-2t-h} \right) \Big)\\
&=\underbrace{- \kappa_{0}^{2} \sum_{{m=1} \atop {m \neq l}}^{M} (\overline{z}_{m}-z_{m}) (\lambda_{m}-\frac{1}{2})^{-1} \cdot \nabla\Phi^{\infty}_{\kappa_0}(\hat{x},z_m)\Big[ \vert D_{m}\vert -\frac{1}{8\pi} \Big(1-\frac{\rho_{m}}{\rho_{0}} \Big)^{-1}  \int_{\partial D_{m}} A_{m}(s) (S^{0}_{D_{m}})^{-1}(1)(s)\ d\sigma_{m}(s)  \Big] \mathbf{C}^{-1}_{m} Q_m}_{=O(a \sum_{m=1}^{M} \nabla \Phi_{\kappa_0}^{\infty}(\hat{x},z_m) Q_m)}\\
&\quad +O(a^{3-\gamma-s})+O\Big(a^{3-\gamma-s} \left( a^{4-\gamma-h-2s}+a^{1-2t-h} \right) \Big).
\end{align*}
%
Thus under the conditions $0\leq t < \frac{1}{2}, \; 0 \leq s \leq \frac{3}{2},\; 0 \leq \gamma \leq 1$,$s+\gamma \leq 2$ and $1-2t-h>0 $ with $ h<\frac{1}{2} $, using \eqref{alg-system-koneqkl-Q-generalshape} and proposition \ref{Prop-phi-estimate}, we have the following estimate for the far-field:
\begin{align}\label{far-field}
 u^\infty(\hat{x}, \theta)&=\sum^M_{m=1}\Phi_{\kappa_{0}}^\infty(\hat{x}, z_m)Q_m -\frac{\kappa_{0}^{2}}{8 \pi} \sum_{m=1}^{M} \nabla\Phi_{\kappa_{0}}^\infty(\hat{x}, z_m) \left(1+ \Big(1-\frac{\rho_m}{\rho_0} \Big)^{-1}\right) (\overline{z}_m-z_m) \left(\lambda_m-\frac{1}{2}\right)^{-1} \hat{A}_m Q_m\\
 &- \kappa_{0}^{2} \sum_{{m=1} \atop {m \neq l}}^{M} (\overline{z}_{m}-z_{m}) (\lambda_{m}-\frac{1}{2})^{-1} \cdot \nabla\Phi^{\infty}_{\kappa_0}(\hat{x},z_m)  \mathbf{C}^{-1}_{m} Q_m \nonumber\\
 &\Big[ \vert D_{m}\vert -\frac{1}{8\pi} \Big(1-\frac{\rho_{m}}{\rho_{0}} \Big)^{-1}  \int_{\partial D_{m}} A_{m}(s) (S^{0}_{D_{m}})^{-1}(1)(s)\ d\sigma_{m}(s)  \Big] \nonumber\\
 &+O(a^{3-\gamma-s}) +O\Big(a^{3-\gamma-s} \left( a^{4-\gamma-h-2s}+a^{1-2t-h} \right) \Big)
 +O\left(a^{3-\gamma-s}+a^{4-2\gamma-s-h}+a^{1-h}+a^{2-s-h} \right)\nonumber\\
 &=\sum^M_{m=1}\Phi_{\kappa_{0}}^\infty(\hat{x}, z_m)Q_m +O\left(a \sum_{m=1}^{M} \nabla\Phi_{\kappa_{0}}^\infty(\hat{x}, z_m) Q_{m} \right)+O(a^{3-\gamma-s})\nonumber \\
 &\qquad +O\Big(a^{3-\gamma-s} \left( a^{4-\gamma-h-2s}+a^{1-2t-h} \right) \Big) 
  +O\left(a^{3-\gamma-s}+a^{4-2\gamma-s-h}+a^{1-h}+a^{2-s-h} \right)\nonumber\\
 &=\sum^M_{m=1}\Phi_{\kappa_{0}}^\infty(\hat{x}, z_m)Q_m +O(a^{1-h}+a^{5-2s-\gamma-2h})+O(a^{3-\gamma-s})
 +O\Big(a^{3-\gamma-s} \left( a^{4-\gamma-h-2s}+a^{1-2t-h} \right) \Big)\nonumber\\
 &\qquad +O\left(a^{4-2\gamma-s-h}+a^{2-s-h} \right).\nonumber
 \end{align}
Next let us observe that for $l=1,\dots,M, $ we can write
$\mathbf{C}^{-1}_{l}=\frac{I^{\prime^d}_{l}}{\kappa_{l}^{2}}+Rem_{l}, $
where $\vert Rem_{l} \vert=O(a^{-1+\gamma}) $.\\
Let the vectors $(\tilde{Q}_{l})^{M}_{l=1} $ satisfy the algebraic system
\begin{equation}
\frac{I^{\prime^d}_{l}}{\kappa_{l}^{2}} \tilde{Q}_l +\sum_{{m=1}\atop {m\neq\;l}}^{M}\Phi_{\kappa_0}(z_l,z_m)\tilde{Q}_m =-u^{I}(z_l).
\label{Q-tilde}
\end{equation}
Note that the dominant terms in the systems satisfied by $(Q_{l})^{M}_{l=1} $ and $(\tilde{Q}_{l})^{M}_{l=1} $ are same and therefore the algebraic system \eqref{Q-tilde} is invertible under the same conditions. \\
From \eqref{mult-obs-109r6-pr5-smpl-final} and \eqref{Q-tilde}, we observe that the vector $(\tilde{Q}_{l}-Q_{l})^{M}_{l=1} $ satisfies the algebraic system
\begin{equation}
\frac{I^{\prime^d}_{l}}{\kappa_{l}^{2}} (\tilde{Q}_l-Q_{l}) +\sum_{{m=1}\atop {m\neq\;l}}^{M}\Phi_{\kappa_0}(z_l,z_m)(\tilde{Q}_m-Q_{m}) =Rem_{l} Q_{l} +O(a^{4-\gamma-2s-h}+a^{1-2t-h}).
\notag
\end{equation}
Therefore
\begin{align*}
\sum_{l=1}^{M} \vert \tilde{Q}_l-Q_{l} \vert^2&\leq C \left(\max_{l} \left\vert \frac{\kappa_{l}^{2}}{I^{\prime^d}_{l}}\right\vert \right)^2 \left[\sum_{l=1}^{M} \vert Rem_{l}\vert^2 \vert Q_{l}\vert^2 + \sum_{l=1}^{M} \vert O(a^{4-\gamma-2s-h}+a^{1-2t-h}) \vert^2 \right]\\
& \leq C (\max_{l} \vert \mathbf{C}_{l}\vert)^2 \left[\max_{l} \vert Rem_{l}\vert^2 \sum_{l=1}^{M} \vert Q_{l}\vert^2 + \sum_{l=1}^{M} \vert O(a^{4-\gamma-2s-h}+a^{1-2t-h}) \vert^2  \right]\\
&\leq C M (\max_{l} \vert \mathbf{C}_{l}\vert)^4 (\max_{l} \vert Rem_{l}\vert)^2 \left[O(1+a^{3-s-t-\gamma-h}+a^{4-\gamma-2s-h}) \right]^2 \\
&\qquad + C M (\max_{l} \vert \mathbf{C}_{l}\vert)^2 \left[O(a^{4-\gamma-2s-h}+a^{1-2t-h}) \right]^2,
\end{align*}
where $C$ is some generic constant. Hence
\begin{align*}
\left(\sum_{l=1}^{M} \vert \tilde{Q}_l-Q_{l} \vert \right)^2 \leq M \sum_{l=1}^{M} \vert \tilde{Q}_l-Q_{l} \vert^2  &\leq C M^2 (\max_{l} \vert \mathbf{C}_{l}\vert)^4 (\max_{l} \vert Rem_{l}\vert)^2 \left[O(1+a^{3-s-t-\gamma-h}+a^{4-\gamma-2s-h}) \right]^2 \\
&\qquad + C M^2 (\max_{l} \vert \mathbf{C}_{l}\vert)^2 \left[O(a^{4-\gamma-2s-h}+a^{1-2t-h}) \right]^2,
\end{align*}
which gives
\begin{align}\label{Q-diff}
\sum_{l=1}^{M} \vert \tilde{Q}_l-Q_{l} \vert &\leq C M (\max_{l} \vert \mathbf{C}_{l}\vert)^2 (\max_{l} \vert Rem_{l}\vert) \left[O(1+a^{3-s-t-\gamma-h}+a^{4-\gamma-2s-h}) \right] \\
&\qquad + C M (\max_{l} \vert \mathbf{C}_{l}\vert) \left[O(a^{4-\gamma-2s-h}+a^{1-2t-h}) \right]\nonumber\\
&\leq C M (\max_{l} \vert \mathbf{C}_{l}\vert)^2 a^{-1+\gamma} \left[O(1+a^{3-s-t-\gamma-h}+a^{4-\gamma-2s-h}) \right]+ O(a^{4-\gamma-2s-2h}+a^{1-2t-2h}).\nonumber
\end{align} 
Using \eqref{Q-diff} in \eqref{far-field}, we obtain
\begin{align}\label{far-field-final}
u^\infty(\hat{x}, \theta)&=\sum^M_{m=1}\Phi_{\kappa_{0}}^\infty(\hat{x}, z_m)\tilde{Q}_m+ \sum^M_{m=1}\Phi_{\kappa_{0}}^\infty(\hat{x}, z_m)(Q_m-\tilde{Q}_m) +O(a^{1-h}+a^{5-2s-\gamma-2h})\\
&\qquad+O(a^{3-\gamma-s})
 +O\Big(a^{3-\gamma-s} \left( a^{4-\gamma-h-2s}+a^{1-2t-h} \right) \Big)
 +O\left(a^{4-2\gamma-s-h}+a^{2-s-h} \right)\nonumber\\
&=\sum^M_{m=1}\Phi_{\kappa_{0}}^\infty(\hat{x}, z_m)\tilde{Q}_m +O(a^{1-h}+a^{5-2s-\gamma-2h})\nonumber\\
&\qquad+O(a^{3-\gamma-s})
 +O\Big(a^{3-\gamma-s} \left( a^{4-\gamma-h-2s}+a^{1-2t-h} \right) \Big)
 +O\left(a^{4-2\gamma-s-h}+a^{2-s-h} \right)\nonumber\\
&\qquad +M(\max_{l} \vert \mathbf{C}_{l}\vert)^2 a^{-1+\gamma} \left[O(1+a^{3-s-t-\gamma-h}+a^{4-\gamma-2s-h}) \right]+ O(a^{4-\gamma-2s-2h}+a^{1-2t-2h}).\nonumber
\end{align}
Finally from \eqref{far-field-final}, we can conclude the following:
\begin{itemize}
\item We recall that when $\gamma<1 $ or we are away from the resonance, we have $M \max_{l} \vert \mathbf{C}_{l}\vert=O(a^{2-\gamma-s}) $ and hence $h=\gamma+s-2 \leq 0$ (see remark \ref{mcm}). Also $\max_{l} \vert \mathbf{C}_{l}\vert=O(a^{2-\gamma}) $. Therefore \eqref{far-field-final} reduces to
\begin{align}\label{far-field-final-1}
u^{\infty}(\hat{x},\theta)&=\sum^M_{m=1}\Phi_{\kappa_{0}}^\infty(\hat{x}, z_m)\tilde{Q}_m +O(a+a^{5-2s-\gamma})\\
&\qquad+O(a^{3-\gamma-s})
 +O\Big(a^{3-\gamma-s} \left( a^{4-\gamma-2s}+a^{1-2t} \right) \Big)
 +O\left(a^{4-2\gamma-s}+a^{2-s} \right)\nonumber\\
&\qquad +a^{3-\gamma-s}\ O(1+a^{3-s-t-\gamma}+a^{4-\gamma-2s})+ O(a^{4-\gamma-2s}+a^{5-2\gamma-2s-2t})\nonumber\\
&=\sum^M_{m=1}\Phi_{\kappa_{0}}^\infty(\hat{x}, z_m)\tilde{Q}_m+O(a^{2-s}+a^{4-\gamma-2s}+a^{5-2\gamma-2s-2t})\nonumber\\
&=\sum^M_{m=1}\Phi_{\kappa_{0}}^\infty(\hat{x}, z_m)\tilde{Q}_m+O(a^{2-s}+a^{3-\gamma-s-2t}),\nonumber
\end{align}
where in the last line we use the fact $s+\gamma \leq 2$.
\item We recall that near the resonance $\max_{l} \vert \mathbf{C}_{l} \vert=O(a^{1-h_{1}}) $ and therefore $O(a^{-h})=M \max_{l} \vert \mathbf{C}_{l} \vert=O(a^{-s+1-h_{1}}) $ leading to the condition $h=-1+s+h_{1} $.\\
Also since $\gamma=1 $, we obtain that $M(\max_{l} \vert \mathbf{C}_{l}\vert)^2 a^{-1+\gamma}=O(a^{s-2h}) $. Then from \eqref{far-field-final}, we conclude that  
\begin{align}\label{far-field-final-2}
u^{\infty}(\hat{x},\theta)&=\sum^M_{m=1}\Phi_{\kappa_{0}}^\infty(\hat{x}, z_m)\tilde{Q}_m +O(a^{1-h}+a^{5-2s-\gamma-2h})\\
&\qquad+O(a^{3-\gamma-s})
 +O\Big(a^{3-\gamma-s} \left( a^{4-\gamma-h-2s}+a^{1-2t-h} \right) \Big)
 +O\left(a^{4-2\gamma-s-h}+a^{2-s-h} \right)\nonumber\\
&\qquad +a^{s-2h}\ O(1+a^{3-s-t-\gamma-h}+a^{4-\gamma-2s-h}) + O(a^{4-\gamma-2s-2h}+a^{1-2t-2h})\nonumber\\
&=\sum^M_{m=1}\Phi_{\kappa_{0}}^\infty(\hat{x}, z_m)\tilde{Q}_m+O(a^{1-h}+a^{2-s-2h}+a^{1-2t-2h}+a^{s-2h})\nonumber\\
&=\sum^M_{m=1}\Phi_{\kappa_{0}}^\infty(\hat{x}, z_m)\tilde{Q}_m+O(a^{2-s-h_{1}}+a^{4-3s-2h_{1}}+a^{3-2t-2s-2h_{1}}+a^{2-s-2h_{1}})\nonumber\\
&=\sum^M_{m=1}\Phi_{\kappa_{0}}^\infty(\hat{x}, z_m)\tilde{Q}_m+O(a^{3-2t-2s-2h_{1}}+a^{2-s-2h_{1}})\nonumber
\end{align}
where in the last line we use the fact that $s\leq 1$ (as we have $s+\gamma\leq 2$ and $\gamma=1$).
\begin{enumerate}
 \item When $l_M<0$, then we have seen that we need that $h\leq 0$, see Remark \ref{mcm} and Remark \ref{extra-conditions}). Hence in this case we need  $s+h_1\leq 1$. 

\item  When $l_M>0$, we need only $h<\frac{1}{2}$ which means that $s+h_1< \frac{3}{2}$.

\end{enumerate}

Note that the error term in \eqref{far-field-final-2} goes to zero provided $s,h_{1}$ and $ t $ satisfy the condition 
\begin{equation}\label{condi-errorb}
 s<\min\{2-2h_1, \; \frac{3-2t-2h_1}{2}\}.
\end{equation}
These conditions are fulfilled if
\begin{equation}
 s+h_1\leq 1,\; ~~ h_1<1\; \mbox{ and } t < \frac{1}{2}.
\end{equation}
These last conditions are the regimes in which one can derive the effective media. Actually, in the case $l_M>0$, one can allow $s+h_1>1$ (but $s+h_1< \min \{ \frac{3}{2}-t, 2-h_1 \}$) and hence generate very large effective potentials.  
\end{itemize}
\begin{remark}\label{Improved-expansions}
Note that in the algebraic system \eqref{Q-tilde}, unlike that in the case of \eqref{mult-obs-109r6-pr5-smpl-final} , we have only used the (leading) term $\frac{I^{\prime^d}_{l}}{\kappa_{l}^{2}} $ instead of $\mathbf{C}^{-1}_{l} $. Now suppose that the vectors $(\tilde{Q}_{l})^{M}_{l=1} $ satisfy the algebraic system 
\begin{equation}
 \mathbf{C}^{-1}_{l} \tilde{Q}_l +\sum_{{m=1}\atop {m\neq\;l}}^{M}\Phi_{\kappa_0}(z_l,z_m)\tilde{Q}_m =-u^{I}(z_l).
\notag
\end{equation}
Then proceeding as in the derivation of \eqref{far-field-final-2}, when the frequency is near the Minnaert resonance, we can obtain the improved estimate
\begin{equation}
u^{\infty}(\hat{x},\theta) =\sum^M_{m=1}\Phi_{\kappa_{0}}^\infty(\hat{x}, z_m)\tilde{Q}_m+O(a^{3-2t-2s-2h_{1}}+a^{2-s-h_{1}}).
\notag
\end{equation} 
The error term in the above estimate goes to zero under the same conditions as in the case of \eqref{far-field-final-2}. This can be observed if we take $Rem_l=0 $  and continue as in the derivation of \eqref{far-field-final-2} from the algebraic system \eqref{Q-tilde}, now obviously replacing $\frac{I^{\prime^d}_{l}}{\kappa_{l}^{2}} $ by $\mathbf{C}^{-1}_{l} $.
\end{remark}

\section{Proofs of auxiliary results}\label{Proof-details}
\subsection{Proof of Proposition \ref{density}:}\label{proof-density}
Let us define operators $T,T_{0}:X \rightarrow Y$ by  
\begin{equation}
\begin{aligned}
 \Big(\phi_{1},\psi_{1},...,\phi_{M},\psi_{M} \Big) \mapsto \Bigg(&\left.\Big(S^{\kappa_{1}}_{D_{1}} \psi_{1}-\sum_{l=1}^{M}S^{\kappa_{0}}_{D_{l}} \phi_{l} \Big) \right\vert_{\partial D_{1}},\\
 &\frac{\rho_0}{\rho_{1}} \Big[\frac{1}{2} Id +(K^{\kappa_{1}}_{D_{1}})^{*} \Big]\psi_{1}- \Big[-\frac{1}{2} Id+(K^{\kappa_{0}}_{D_{1}})^{*} \Big] \phi_{1}-\left.\sum_{{l=1}\atop {l\neq 1}}^{M} \frac{\partial(S^{\kappa_{0}}_{D_{l}} \phi_{l})}{\partial \nu^{1}} \right\vert_{\partial D_{1}},...,\\ 
 &\left.\Big(S^{\kappa_{M}}_{D_{M}} \psi_{M}-\sum_{l=1}^{M}S^{\kappa_{0}}_{D_{l}} \phi_{l} \Big) \right\vert_{\partial D_{M}},\\
 &\frac{\rho_0}{\rho_{M}} \Big[\frac{1}{2} Id +(K^{\kappa_{M}}_{D_{M}})^{*} \Big]\psi_{M}- \Big[-\frac{1}{2} Id+(K^{\kappa_{0}}_{D_{M}})^{*} \Big] \phi_{M}-\left.\sum_{{l=1}\atop {l\neq M}}^{M} \frac{\partial(S^{\kappa_{0}}_{D_{l}} \phi_{l})}{\partial \nu^{M}} \right\vert_{\partial D_{M}}  \Bigg)
 \end{aligned}
 \notag
\end{equation}
and
\begin{equation}
\begin{aligned}
 \Big(\phi_{1},\psi_{1},...,\phi_{M},\psi_{M} \Big) \mapsto &\Big(\Big(S^{0}_{D_{1}} \psi_{1}-S^{0}_{D_{1}} \phi_{1} \Big) \Big\vert_{\partial D_{1}},\ \frac{\rho_0}{\rho_{1}} \Big[\frac{1}{2} Id +(K^{0}_{D_{1}})^{*} \Big]\psi_{1}- \Big[-\frac{1}{2} Id+(K^{0}_{D_{1}})^{*} \Big] \phi_{1},...,\\ 
 &\Big(S^{0}_{D_{M}} \psi_{M}-S^{0}_{D_{M}} \phi_{M} \Big) \Big\vert_{\partial D_{M}},\ \frac{\rho_0}{\rho_{M}} \Big[\frac{1}{2} Id +(K^{0}_{D_{M}})^{*} \Big]\psi_{M}- \Big[-\frac{1}{2} Id+(K^{0}_{D_{M}})^{*} \Big] \phi_{M} \Big).
 \end{aligned}
 \notag
\end{equation}
 We note that the operator $T-T_{0}:X \rightarrow Y$ is a compact operator.
To see this, first of all we observe that the typical terms in $(T-T_{0})(\phi_{1},\psi_{1},...,\phi_{M},\psi_{M})$ are of the form
\begin{equation}
-\sum_{{l=1} \atop {l \neq i}}^{M} S_{D_{l}}^{\kappa_{0}} \phi_{l} +(S^{\kappa_{i}}_{D_{i}} -S^{0}_{D_{i}})\psi_{i} -(S^{\kappa_{0}}_{D_{i}}-S^{0}_{D_{i}})\phi_{i}
 \notag
\end{equation}
and
\begin{equation}
-\sum_{{l=1}\atop {l \neq i}}^{M} \frac{\partial(S^{\kappa_{0}}_{D_{l}} \phi_{l})}{\partial \nu^{i}} 
+\frac{\rho_{0}}{\rho_{i}} \Big[(K^{\kappa_i}_{D_i})^{*}-(K^{0}_{D_i})^{*} \Big]\psi_i - \Big[(K^{\kappa_0}_{D_i})^{*}-(K^{0}_{D_i})^{*} \Big]\phi_i.
 \notag
\end{equation}
The first term gives rise to a compact operator since $(S^{\kappa_{i}}_{D_{i}}-S^{0}_{D_{i}}), (S^{\kappa_{0}}_{D_{i}}-S^{0}_{D_{i}})$ are compact and also for $l \neq i$, $S^{\kappa_{0}}_{D_{l}} \phi_{l}$ is compact. A similar argument
works for the second term as well, and therefore it follows that $T-T_{0}$ is a compact operator.\\
%
Also note that the operator $T_{0}$ is invertible.
This follows directly as in the proof of Theorem 11.4, Page 189, of \cite{Ammari-Kang-2}, observing that $ T_0$ is diagonal with a $2\times2$ operator, at the diagonal, of the form
\[(\phi_{l},\psi_{l}) \mapsto \Big(\Big(S^{0}_{D_{l}} \psi_{l}-S^{0}_{D_{l}} \phi_{l} \Big) \Big\vert_{\partial D_{l}},\ \frac{\rho_0}{\rho_{l}} \Big[\frac{1}{2} Id +(K^{0}_{D_{l}})^{*} \Big]\psi_{l}- \Big[-\frac{1}{2} Id+(K^{0}_{D_{l}})^{*} \Big] \phi_{l} \Big). \]
Therefore to prove that $T$ is invertible, it is sufficient to prove that $T$ is injective, as an application of the Fredholm alternative.\\
Let us therefore suppose that $T(\phi_{1},\psi_{1},...,\phi_{M},\psi_{M}) =0$. Then 
\begin{equation}
u(x)=\begin{cases}
      \sum^{M}_{l=1} S^{\kappa_{0}}_{D_{l}} \phi_{l}(x), \ x \in \mathbb{R}^{3}\diagdown \overline{\cup_{l=1}^{M} D_{l}}   \\
      S^{\kappa_{s}}_{D_{s}} \psi_{s}(x) , \ x \in D_{s}            \end{cases}
\notag
\end{equation}
 is the unique solution to the problem under consideration.\\
 Now 
 \[\int_{\partial D_{i}} \frac{\partial u}{\partial \nu}\Big\vert_{+} \bar{u}\ d\sigma=\frac{\rho_{0}}{\rho_{i}} \int_{\partial D_{i}} \frac{\partial u}{\partial \nu}\Big\vert_{-} \bar{u} d\sigma
 =\frac{\rho_{0}}{\rho_{i}} \int_{D_{i}} (\vert \nabla u\vert^{2}-\kappa_{i}^{2} \vert u\vert^{2}) \ dx
 \]
which is real and therefore
\[\text{Im} \ \int_{\partial D_{i}} \frac{\partial u}{\partial \nu} \Big\vert_{+} \bar{u}\ d \sigma=0, \]
which in turn implies that
$u=0$ in $\mathbb{R}^{3}\diagdown \cup_{l=1}^{M} D_{l}$.\\
Again, $u$ solves the equation
\[ 
\begin{aligned}
 &(\Delta + \kappa_{i}^{2}) u =0 \ \text{in} \ D_{i}, \\
 &u=\frac{\partial u}{\partial \nu^{i}}=0 \ \text{on} \ \partial D_{i}
 \end{aligned}
\]
and hence by unique continuation property, $u=0\  \text{in}\ D_{i}$.
Arguing similarly for each obstacle $D_{i}$, we obtain that $u = 0 \ \text{in}\ \mathbb{R}^{3}$. In particular, this implies that
\[\sum^{M}_{l=1} S^{\kappa_{0}}_{D_{l}} \phi_{l}(x)=0\ \text{on}\ \partial D_{i}.\]
Since $\Big(\Delta+\kappa_{0}^{2}\Big)(\sum^{M}_{l=1} S^{\kappa_{0}}_{D_{l}} \phi_{l})=0 \ \text{in}\ D_{i}$ and $\kappa_{0}^{2}$ is not a Dirichlet eigenvalue for $-\Delta$ on $D_{i}$ (which holds since $ a \ll 1$),
it follows that \[\sum^{M}_{l=1} S^{\kappa_{0}}_{D_{l}} \phi_{l}=0 \ \text{in} \ D_{i} \] and therefore since this is true for each $i$, we have
\begin{equation} 
\sum^{M}_{l=1} S^{\kappa_{0}}_{D_{l}} \phi_{l}=0\ \text{in}\ \mathbb{R}^{3}.
\label{apc35}
\end{equation}
Now for a fixed $i$, we know that whenever $l\neq i$, we have $\frac{\partial S^{\kappa_{0}}_{D_{l}} \phi_{l}}{\partial \nu^{i}}\Big\vert_{+}=\frac{\partial S^{\kappa_{0}}_{D_{l}}\phi_{l}}{\partial \nu^{i}}\Big\vert_{-}$
and therefore we obtain using \eqref{apc35},
\[\phi_{i}=\frac{\partial S^{\kappa_{0}}_{D_{i}} \phi_{i}}{\partial \nu^{i}}\Big\vert_{+}- \frac{\partial S^{\kappa_{0}}_{D_{i}} \phi_{i}}{\partial \nu^{i}}\Big\vert_{-}=0 \ \text{on}\ \partial D_{i}.\]
Next we prove that $\psi_{i}=0 \ \text{on}\ \partial D_{i}, \ \text{for all} \ i $.
Since we already know that $\phi_{i}=0 \ \forall \ i $, it follows from the form of the solution $u$ and the zero jump condition that
$S^{\kappa_{i}}_{D_{i}} \psi_{i}=0 \ \text{on} \ \partial D_{i}.$ 
Using this fact and the fact that $\Big(\Delta+\kappa^{2}_{i} \Big)  S^{\kappa_{i}}_{D_{i}} \psi_{i}=0 $ in $\mathbb{R}^{3}\diagdown \overline{D_{i}} $, we obtain that
$ S^{\kappa_{i}}_{D_{i}} \psi_{i}=0 \ \text{in} \ \mathbb{R}^{3}\diagdown \overline{D_{i}}$.
Recalling the fact that we have already proved that $u = 0 \ \text{in}\ D_{i}$ and hence $ S^{\kappa_{i}}_{D_{i}} \psi_{i}=0 \ \text{on} \ \partial D_{i},$ it therefore follows
that $S^{\kappa_{i}}_{D_{i}} \psi_{i}=0 \ \text{in}\ \mathbb{R}^{3}$. Now we can proceed just in the case of $\phi_{l}$ by using the Neumann jump of 
$ S^{\kappa_{i}}_{D_{i}} \psi_{i}$ across $D_{i}$ to obtain that $\psi_{i}=0$. 
\subsection{Proof of Lemma \ref{scaling-1}}\label{proof-scaling-1}

In order to prove \eqref{Sinv_int_difpts_psi}, we first set
$\left(S_{D_l}^0\right)^{-1}\Bigg( \int_{\partial D_l}|\cdot-t|^n \phi_{l}(t) d\sigma_{l}(t)\Bigg)=: \chi_{l}. $
As
\begin{equation}
\begin{aligned}
&\delta^{n+2} \int_{\partial B_{l}} \vert \hat{s}-\hat{t} \vert^{n} \hat{\phi}_{l}(\hat{t})\ d\hat{\sigma}_{l}(\hat{t})= \int_{\partial D_{l}} \vert s-t \vert^n \phi_{l}(t)\ d\sigma_{l}(t) =\Big[S^{0}_{D_{l}} \chi_{l} \Big](s) =\delta \Big[S^{0}_{B_{l}} \hat{\chi}_{l} \Big](\hat{s}), 
\end{aligned}
\end{equation}
then
\begin{align*}
\hat{\chi}_{l}= \delta^{n+1} (S^{0}_{B_{l}})^{-1} \Big(\int_{\partial B_{l}} \vert \cdot-\hat{t} \vert^{n} \hat{\phi}_{l}(\hat{t})\ d\hat{\sigma}_{l}(\hat{t}) \Big) , 
\end{align*}
where $[S^{0}_{B_{l}} \hat{\chi}_{l} ](\hat{s}):=\int_{\partial B_{l}} \frac{1}{4 \pi \vert \hat{s}-\hat{t} \vert} \hat{\chi}_{l}(\hat{t})\ d\hat{\sigma}_{l}(\hat{t})  $ and $d\hat{\sigma}_{l} $ denotes the surface measure on $\partial B_{l} $.\\
Therefore
\begin{align*}
\frac{1}{\delta} \norm{\chi_{l}}_{L^{2}(\partial D_{l})}= \norm{\hat{\chi}_{l}}_{L^{2}(\partial B_{l})} &=\delta^{n+1} \norm{(S^{0}_{B_{l}})^{-1}}_{\mathcal{L}(H^1(\partial B_{l}),L^2(\partial B_{l}))} \Big\Vert\int_{\partial B_{l}} \vert \cdot- \hat{t} \vert^{n}\ \hat{\phi}_{l}(\hat{t})\ d\hat{\sigma}_{l}(\hat{t}) \Big\Vert_{H^1(\partial B_{l})} \\
&=O(\delta^{n+1} \norm{\hat{\phi}_{l}}_{L^{2}(\partial B_{l})}) = O(\delta^{n} \norm{\phi_{l}}_{L^{2}(\partial D_{l})}),
\end{align*}
whence it follows that
\begin{align*}
\left\|\left(S_{D_l}^0\right)^{-1}\Bigg( \int_{\partial D_l}|\cdot-t|^n \phi_{l}(t) d\sigma_{l}(t)\Bigg)\right\|_{L^2(\partial D_l)}= \norm{\chi_{l}}_{L^{2}(\partial D_{l})} &= O\Big(\delta^{n+1} \norm{\phi_{l}}_{L^{2}(\partial D_{l})}\Big)\\
& = O\Big(a^{n+1} \norm{\phi_{l}}_{L^{2}(\partial D_{l})}\Big).
\end{align*}
Similarly, to prove \eqref{InvSsmzl--ori}, we write $(S^{0}_{D_{l}})^{-1} \Big((\cdot-z_{l})^{n} \Big)= \eta_{l} $. Then 
\begin{align*}
\delta^{n} \hat{s}^{n}= (s-z_{l})^{n} &= \Big[S^{0}_{D_{l}} \eta_{l} \Big](s)=\delta \Big[S^{0}_{B_{l}} \hat{\eta}_{l} \Big](\hat{s}),
\end{align*}
which implies
\begin{equation}
\hat{\eta}_{l}=\delta^{n-1} (S^{0}_{B_{l}})^{-1}(\cdot^{n}),
\notag
\end{equation}
and hence 
\begin{align*}
& \frac{1}{\delta} \norm{\eta_{l}}_{L^{2}(\partial D_{l})} = \norm{\hat{\eta}_{l}}_{L^{2}(\partial B_{l})}= \delta^{n-1} \norm{(S^{0}_{B_{l}})^{-1} (\cdot^{n})}_{L^{2}(\partial B_{l})}\\
\Rightarrow &\norm{(S^{0}_{D_{l}})^{-1} \Big((\cdot-z_{l})^{n} \Big)}_{L^{2}(\partial D_{l})}= \norm{\eta_{l}}_{L^{2}(\partial D_{l})} = \delta^{n} \norm{(S^{0}_{B_{l}})^{-1} (\cdot^{n})}_{L^{2}(\partial B_{l})} = O(a^{n}).
\end{align*}

\subsection{Proof of Lemma \ref{Claimest}}\label{Proof-Claimest}
 In order to prove \eqref{claim-single-Lay-dif-2}, we proceed as follows. 
 \begin{equation}
 \begin{aligned}
 \left[S^{\kappa_{0}}_{D_{m}}-S^{\kappa_{m}}_{D_{m}}\right]&\psi_{m}(x)
 =\int_{\partial D_m}[\Phi_{\kappa_0}-\Phi_{\kappa_m}](x,s)\psi_m(s)\,d\sigma_{m}(s)\\ 
 &=\frac{i}{4\pi}(\kappa_0-\kappa_m)\int_{\partial D_m}\psi_m(s)\,d\sigma_{m}(s)+
 \underbrace{\sum_{n=2}^\infty\frac{i^n(\kappa_0^n-\kappa_m^n)}{4\pi n!}\int_{\partial D_m}\vert{x-s}\vert^{n-1}\psi_m(s)\,d\sigma_{m}(s)}_{=: Err1_{m}}\quad\\
  &=\frac{i}{4\pi}(\kappa_0-\kappa_m)\int_{\partial D_m}\psi_m(s)\,d\sigma_{m}(s)+
{O\left(\sum_{n=2}^\infty\frac{(\kappa_0^n+\kappa_m^n)}{4\pi n!}a^{n-1}\int_{\partial D_m}\vert\psi_m(s) \vert \ d\sigma_{m}(s) \right)}\\
  &=\frac{i}{4\pi}(\kappa_0-\kappa_m)\int_{\partial D_m}\psi_m(s)\,d\sigma_{m}(s)+
O\left(\frac{a^{2}}{4\pi}\left[\frac{\kappa_0^2}{1-\kappa_0\,a}+\frac{\kappa_m^2}{1-\kappa_m\,a}\right]\|\psi_m\|\right)\\
  &=\frac{i}{4\pi}(\kappa_0-\kappa_m)\int_{\partial D_m}\psi_m(s)\,d\sigma_{m}(s)+
O\left(a^{2}\|\psi_m\|\right),
\end{aligned}
\label{claim-single-Lay-dif-prf-1}
\end{equation}
whence \eqref{claim-single-Lay-dif-2} follows. Note that here we use the fact that $\vert \kappa_{0}a \vert, \vert \kappa_{m} a \vert<\frac{1}{2} $, which holds since $a \ll 1 $.\\ 

In order to prove \eqref{claim-adjdblle-Lay-dif-prime}, we first observe that
 \begin{align*}
  \int_{\partial D_{l}}(K^{\kappa}_{D})^{*}\psi(s) d\sigma_{l}(s) = \int_{\partial D_{l}}\psi(s)(K^{\kappa}_{D})(1)(s) d\sigma_{l}(s)
  &= \int_{\partial D_{l}}\psi(s)\Bigg[\int_{\partial D_{l}}\frac{\partial}{\partial\nu_t}\Phi_{\kappa}(s,t) d\sigma_{l}(t)\Bigg]d\sigma_{l}(s) \\
  &= \int_{\partial D_{l}}\psi(s)\Bigg[\int_{ \partial D_{l}}\nabla_t\Phi_{\kappa}(s,t)\cdot\nu_t \ d\sigma_{l}(t)\Bigg]d\sigma_{l}(s).
 \end{align*}

 Then we can write

 \begin{align*}
 \int_{\partial D_l}&\left[(K^{\kappa_{0}}_{D_{l}})^{*} -(K^{\kappa_{l}}_{D_{l}})^{*}\right]\psi_l(s) d\sigma_{l}(s)
 =\int_{\partial D_l}\psi_l(s)\Bigg[\int_{ \partial D_l}\nabla_t\Big(\Phi_{\kappa_0}-\Phi_{\kappa_l}\Big)(s,t)\cdot\nu_t \ d\sigma_{l}(t) \Bigg]d\sigma_{l}(s) \\
  =&\int_{\partial D_l}\psi_l(s)\Bigg[\int_{ \partial D_l}\nabla_t\Bigg(\sum_{n=1}^\infty\frac{i^n(\kappa_{0}^n-\kappa_{l}^n)}{4\pi\; n!}\vert{s-t}\vert^{n-1}\Bigg) \cdot\nu_t \ d\sigma_{l}(t)\Bigg]d\sigma_{l}(s) \\
    =&\int_{\partial D_l}\psi_l(s)\Bigg[\int_{ \partial D_l}\Bigg(\sum_{n=2}^\infty\frac{i^n(\kappa_{l}^n-\kappa_{0}^n)}{4\pi\; n!}(n-1)\vert{s-t}\vert^{n-2}\frac{(s-t)}{\vert{s-t}\vert}\Bigg) \cdot\nu_t d\sigma_{l}(t) \Bigg]d\sigma_{l}(s) \\
 =&\frac{1}{8\pi}(\kappa_0^2-\kappa_l^2)\int_{\partial D_l}\psi_l(s)\left[\int_{ \partial D_l}\frac{(s-t)}{\vert{s-t}\vert} \cdot\nu_t \,d\sigma_{l}(t)\right]d\sigma_{l}(s) \\
 &+\sum_{n=3}^\infty\frac{i^n(\kappa_{l}^n-\kappa_{0}^n)}{4\pi\; n!}(n-1)\int_{\partial D_l}\psi_l(s)\Bigg[\int_{ \partial D_l}\Bigg(\vert{s-t}\vert^{n-2}\frac{(s-t)}{\vert{s-t}\vert}\Bigg) \cdot\nu_t d\sigma_{l}(t) \Bigg]d\sigma_{l}(s) \\
=&\frac{1}{8\pi}(\kappa_0^2-\kappa_l^2)\int_{\partial D_l}\psi_l(s)\left[\int_{ \partial D_l}\frac{(s-t)}{\vert{s-t}\vert} \cdot\nu_t \,d\sigma_{l}(t) \right]d\sigma_{l}(s)-\frac{i}{4 \pi}(\kappa_0^3-\kappa_l^3)\vert{D_l}\vert\int_{\partial D_l}\psi_l(s)d\sigma_{l}(s) \\
 &+\underbrace{\sum_{n=4}^\infty\frac{i^n(\kappa_{l}^n-\kappa_{0}^n)}{4\pi\; n!}(n-1)\int_{\partial D_l}\psi_l(s)\Bigg[\int_{ \partial D_l}\Bigg(\vert{s-t}\vert^{n-2}\frac{(s-t)}{\vert{s-t}\vert}\Bigg) \cdot\nu_t d\sigma_{l}(t) \Bigg]d\sigma_{l}(s)}_{=:Err2_{l}=O({a^5}\|\psi_l\|)},
 \end{align*}
 where the last step follows from the divergence theorem. Thus \eqref{claim-adjdblle-Lay-dif-prime} follows. \\

 Next we establish \eqref{claim-normsingle-Lay}. For this, we first observe that
\begin{align*}
\int_{\partial D_{l}} \Big[\frac{\partial S^{\kappa_{0}}_{D_{m}}  \phi_{m}}{\partial \nu^{l}}  \Big](s)\ d\sigma_{l}(s)  
 &= \int_{\partial D_{l}} \Big[\int_{\partial D_{m}} \frac{\partial \Phi_{\kappa_{0}}(s,t)}{\partial \nu^{l}(s)} \phi_{m}(t) \ d\sigma_{m}(t) \Big] d\sigma_{l}(s)  \\
 &= \int_{\partial D_{m}} \Big[\int_{\partial D_{l}} \frac{\partial \Phi_{\kappa_{0}}(s,t)}{\partial \nu^{l}(s)} d\sigma_{l}(s) \Big] \phi_{m}(t) d\sigma_{m}(t) \\
 &= \int_{\partial D_{m}} \Big[\int_{D_{l}} \Delta \Phi_{\kappa_0}(x,t)\ dx \Big] \phi_{m}(t)\ d\sigma_{m}(t) \\
 &= -\kappa_{0}^{2} \int_{\partial D_{m}} \Big[\int_{D_{l}} \Phi_{\kappa_{0}}(x,t)\ dx \Big] \phi_{m}(t)\ d\sigma_{m}(t).
 \end{align*}
 We can write this as
 \begin{align*}
 \int_{\partial D_{l}} \Big[\frac{\partial S^{\kappa_{0}}_{D_{m}}  \phi_{m}}{\partial \nu^{l}}  \Big](s)\ d\sigma_{l}(s)  &=-\kappa_{0}^{2} \int_{\partial D_{m}} \Phi_{\kappa_{0}}(z_{l},t) \vert D_{l} \vert \phi_{m}(t) d \sigma_{m}(t)\\
  & -\kappa_{0}^{2} \int_{\partial D_{m}} \nabla_x\Phi_{\kappa_{0}}(z_{l},t)\cdot\Big[ \int_{D_{l}}(x-z_l) dx\Big]  \phi_{m}(t) d \sigma_{m}(t)\\
 &-\underbrace{\kappa_{0}^{2} \int_{\partial D_{m}} \Big[ \int_{D_{l}} (\Phi_{\kappa_{0}}(x,t)-\Phi_{\kappa_{0}}(z_l,t)-(x-z_l)\cdot\nabla_x\Phi_{\kappa_{0}}(z_{l},t)) dx\Big]  \phi_{m}(t) d \sigma_{m}(t)}_{=:G_{1_{ml}}^{\kappa_0}=O(\frac{a^6}{d_{ml}^3}\|\phi_m\|)} 
 \end{align*}
 which further implies that
 \begin{align*}
 \int_{\partial D_{l}} \Big[\frac{\partial S^{\kappa_{0}}_{D_{m}} \phi_{m}}{\partial \nu^{l}}  \Big](s)\ d\sigma_{l}(s) &=-\kappa_{0}^{2} \int_{\partial D_{m}} \Phi_{\kappa_{0}}(z_{l},z_{m}) \vert D_{l} \vert \phi_{m}(t) d \sigma_{m}(t)  \\
 &-\kappa_{0}^{2} \int_{\partial D_{m}} \vert D_{l} \vert \nabla_{y} \Phi_{\kappa_{0}}(z_{l},z_{m})\cdot (t-z_{m}) \phi_{m}(t) d\sigma_{m}(t)\\
 & -\underbrace{\kappa_{0}^{2} \int_{\partial D_{m}} \vert D_{l} \vert \big[ \Phi_{\kappa_{0}}(z_{l},t)- \Phi_{\kappa_{0}}(z_{l},z_{m})-\nabla_{t} \Phi_{\kappa_{0}}(z_{l},z_{m})\cdot (t-z_{m})\big] \phi_{m}(t) d\sigma_{m}(t)}_{=:G_{2_{ml}}^{\kappa_0}=O(\frac{a^6}{d_{ml}^3}\|\phi_m\|)}-{G_{1_{ml}}^{\kappa_0}} \\
 &-\kappa_{0}^{2} \int_{\partial D_{m}} \nabla_x\Phi_{\kappa_{0}}(z_{l},z_m)\cdot\Big[ \int_{D_{l}}(x-z_l) dx\Big]  \phi_{m}(t) d \sigma_{m}(t)\\
  &-\underbrace{\kappa_{0}^{2} \int_{\partial D_{m}} [\nabla_x \Phi_{\kappa_{0}}(z_{l},t)-\nabla_x\Phi_{\kappa_{0}}(z_{l},z_m)]\cdot\Big[ \int_{D_{l}}(x-z_l) dx\Big]  \phi_{m}(t) d \sigma_{m}(t)}_{=:G_{3_{ml}}^{\kappa_0}=O(\frac{a^6}{d_{ml}^3}\|\phi_m\|)}
  \end{align*}
  and therefore
  \begin{align} \label{claim-normsingle-Lay-Prf}
 \int_{\partial D_{l}} \Big[\frac{\partial S^{\kappa_{0}}_{D_{m}} \phi_{m}}{\partial \nu^{l}}  \Big](s)\ d\sigma_{l}(s) &=-\kappa_{0}^{2} \Phi_{\kappa_{0}}(z_{l},z_{m}) \vert D_{l} \vert \int_{\partial D_{m}} \phi_{m}(t) d \sigma_{m}(t)\\
 & -\kappa_{0}^{2} \vert D_{l} \vert \nabla_{t} \Phi_{\kappa_{0}}(z_{l},z_{m})\cdot \int_{\partial D_{m}} (t-z_{m})  \phi_{m}(t) d\sigma_{m}(t)\nonumber\\
 &-\kappa_{0}^{2}\nabla_x\Phi_{\kappa_{0}}(z_{l},z_m)\cdot\Big[ \int_{D_{l}}(x-z_l) dx\Big] \int_{\partial D_{m}}   \phi_{m}(t) d \sigma_{m}(t) -\underbrace{[G_{1_{ml}}^{\kappa_0}+G_{2_{ml}}^{\kappa_0}+G_{3_{ml}}^{\kappa_0}]}_{=:Err4_{m}=O(\frac{a^6}{d_{ml}^3}\|\phi_m\|)}, \nonumber
 \end{align}
 whence \eqref{claim-normsingle-Lay} follows. \\

The proof of \eqref{claim-norm-inci} follows easily from the observation
\begin{align}\label{claim-norm-inci-prf}
\int_{\partial D_{l}} \frac{\partial u^{I}}{\partial \nu^{l}} 
=&-\kappa_0^2\,\vert{D_l}\vert\,u^{I}(z_l)\underbrace{-{\kappa_0^2}  \int_{ D_l} (u^{I}(y)-u^{I}(z_l)\, dy}_{=:Err5_{l}=O(a^4)}.
\end{align}
%
%
We shall next derive the approximations \eqref{claim- diff-phpsi} and \eqref{claim-int- diff-phpsi}. 
First of all, we note that
%
\begin{align}\label{exp-SDlDKO}
S_{D_l}^{d_{\kappa_0}}\phi_l(s)
&=\int_{\partial D_l} \frac{e^{i\kappa_0|s-y|}-1}{4\pi\, |s-y|}\phi_l(y)\,d\sigma_{l}{(y)}
=\int_{\partial D_l}\sum_{n=1}^\infty\frac{1}{4\pi}\frac{(i\kappa_0)^n}{n!}|s-y|^{n-1}\phi_l(y)\,d\sigma_{l}{(y)}\\
&=\frac{i\kappa_0}{4\pi}\int_{\partial D_l}\phi_l(y)\,d\sigma_{l}{(y)}+\int_{\partial D_l}\sum_{n=2}^\infty\frac{1}{4\pi}\frac{(i\kappa_0)^n}{n!}|s-y|^{n-1}\phi_l(y)\,d\sigma_{l}{(y)}\nonumber\\
&=\begin{cases}
\frac{i\kappa_0}{4\pi}Q_l+O(a^2\|\phi_l\|),\mbox{ in $L^\infty$ and $H^1$},\\
\frac{i\kappa_0}{4\pi}Q_l+O(a^3\|\phi_l\|),\mbox{ in $L^2$}.
\end{cases}\nonumber
\end{align}
%
From  \eqref{mult-obs-102r}, we have that on $\partial D_{l}$,
\begin{align*}
 \psi_{l}(s)&=\left(S^{\kappa_{l}}_{D_{l}}\right)^{-1}\left[u^{I} + \sum_{{m=1} }^{M} S^{\kappa_{0}}_{D_{m}} \phi_{m}\right](s)\\
 &=\left(S^{{0}}_{D_{l}}\right)^{-1}\left[u^{I} + S^{\kappa_{0}}_{D_{l}}\phi_{l}+\sum_{{m=1} \atop {m \neq l}}^{M} S^{\kappa_{0}}_{D_{m}} \phi_{m}\right](s)\\
 &\qquad +\left(S^{{0}}_{D_{l}}\right)^{-1}\left[\sum_{n=1}^\infty (-1)^{n} \left(S^{{d_{\kappa_l}}}_{D_{l}}\left(S^{{0}}_{D_{l}}\right)^{-1}\right)^n\left(u^{I} +S^{\kappa_{0}}_{D_{l}} \phi_l+ \sum_{{m=1}\atop m\neq l }^{M} S^{\kappa_{0}}_{D_{m}} \phi_{m}\right)\right](s)\\ 
  &=\left(S^{{0}}_{D_{l}}\right)^{-1}\left[u^{I}(z_l) +(s-z_l)\cdot\nabla u^I(z_l)+O(a^2) +S^{{0}}_{D_{l}}\phi_{l}+S^{d_{\kappa_{0}}}_{D_{l}}\phi_{l}+\sum_{{m=1} \atop {m \neq l}}^{M} S^{\kappa_{0}}_{D_{m}} \phi_{m}\right](s)\\
 &\qquad+\left(S^{{0}}_{D_{l}}\right)^{-1}\left[\sum_{n=1}^\infty (-1)^{n} \left(S^{{d_{\kappa_l}}}_{D_{l}}\left(S^{{0}}_{D_{l}}\right)^{-1}\right)^n\left(u^{I} +S^{{0}}_{D_{l}}\phi_{l}+S^{d_{\kappa_{0}}}_{D_{l}}\phi_{l}+ \sum_{{m=1}\atop m\neq l }^{M} S^{\kappa_{0}}_{D_{m}} \phi_{m}\right)\right](s),
 \end{align*}
 where the approximation above is in a point-wise sense and $S^{d_{\kappa_{l}}}_{D_{l}} \phi(s):=\int_{\partial D_{l}} \frac{e^{i \kappa_{l} \vert s-t \vert}-1}{4\pi \vert s-t \vert} \phi(t) \ d\sigma_{l}(t) $.\\
 Using \eqref{InvSsmzl--ori} and \eqref{exp-SDlDKO}, we can write this as
 \begin{align*}
  \psi_{l}(s) &=\phi_l(s)+\left(S^{{0}}_{D_{l}}\right)^{-1}\left(u^{I}(z_l)\right)  +\left(S^{{0}}_{D_{l}}\right)^{-1}\left(S^{d_{\kappa_{0}}}_{D_{l}}-S^{d_{\kappa_{l}}}_{D_{l}}\right)\phi_{l}(s)+O(a)+O(a^2\|\phi_l\|) \\
   &+\left[\left(S^{{0}}_{D_{l}}\right)^{-1}\sum_{{m=1} \atop {m \neq l}}^{M} S^{\kappa_{0}}_{D_{m}} \phi_{m}\right]+\left(S^{{0}}_{D_{l}}\right)^{-1}\left[\sum_{n=1}^\infty (-1)^{n} \left(S^{{d_{\kappa_l}}}_{D_{l}}\left(S^{{0}}_{D_{l}}\right)^{-1}\right)^n\left( \sum_{{m=1}\atop m\neq l }^{M} S^{\kappa_{0}}_{D_{m}} \phi_{m}\right)\right]
   \end{align*}
with error estimates in $L^2$ sense.
This further implies, using $(S^{d_{\kappa_0}}_{D_l}-S^{d_{\kappa_{l}}}_{D_{l}})\phi_{l}= (S^{\kappa_0}_{D_l}-S^{\kappa_l}_{D_l})\phi_{l} $ and \eqref{claim-single-Lay-dif-2}, that   
\begin{align*}    
 \psi_{l}(s)  &=\phi_l(s)+\left(S^{{0}}_{D_{l}}\right)^{-1}\left(u^{I}(z_l)\right)+\frac{i(\kappa_0-\kappa_l)}{4\pi}Q_l\left(S^{{0}}_{D_{l}}\right)^{-1}(1)(s)\\
 &\qquad +\left(S^{{0}}_{D_{l}}\right)^{-1}\left(\int_{\partial D_l}\sum_{n=2}^\infty\frac{1}{4\pi}\frac{i^n(\kappa_0^n-\kappa_l^n)}{n!}|s-y|^{n-1}\phi_l(y)\,d\sigma_{l}{(y)}\right)(s)\\
   &\qquad +\left[\left(S^{{0}}_{D_{l}}\right)^{-1}\sum_{{m=1} \atop {m \neq l}}^{M} S^{\kappa_{0}}_{D_{m}} \phi_{m}\right]+\left(S^{{0}}_{D_{l}}\right)^{-1}\left[\sum_{n=1}^\infty (-1)^{n} \left(S^{{d_{\kappa_l}}}_{D_{l}}\left(S^{{0}}_{D_{l}}\right)^{-1}\right)^n\left(\sum_{{m=1}\atop m\neq l }^{M} S^{\kappa_{0}}_{D_{m}} \phi_{m}\right)\right]\\
  &\qquad +O(a)+O(a^2\|\phi_l\|), \mbox{ in } L^2,
\end{align*}
and hence, keeping only the first terms in the two infinite series,
\begin{align*}    
\psi_{l}(s) &=\phi_l(s)+\left(S^{{0}}_{D_{l}}\right)^{-1}\left(u^{I}(z_l)\right)+\frac{i(\kappa_0-\kappa_l)}{4\pi}Q_l\left(S^{{0}}_{D_{l}}\right)^{-1}(1)(s)\\
&\qquad +\frac{i^2(\kappa_0^2-\kappa_l^2)}{8\pi}\left(S^{{0}}_{D_{l}}\right)^{-1}\left(\int_{\partial D_l}|\cdot-y|\phi_l(y)\,d\sigma_{l}{(y)}\right)(s) \\
   &\qquad +\left[\left(S^{{0}}_{D_{l}}\right)^{-1}\sum_{{m=1} \atop {m \neq l}}^{M} S^{\kappa_{0}}_{D_{m}} \phi_{m}\right]-\left(S^{{0}}_{D_{l}}\right)^{-1}\left[S^{{d_{\kappa_l}}}_{D_{l}}\left(S^{{0}}_{D_{l}}\right)^{-1}\left(\sum_{{m=1}\atop m\neq l }^{M} S^{\kappa_{0}}_{D_{m}} \phi_{m}\right)\right]\\
   &\qquad +O\left(\sum_{{m=1}\atop m\neq l }^{M}\frac{a^3}{d_{lm}^2}\|\phi_m\|\right)+O(a)+O(a^2\|\phi_l\|). 
\end{align*}
Using \eqref{Sinv_int_difpts_psi}, we can write this as
\begin{align}    
 \psi_{l}(s) &=\phi_l(s)+\left(S^{{0}}_{D_{l}}\right)^{-1}\left(u^{I}(z_l)\right)  +\frac{i(\kappa_0-\kappa_l)}{4\pi}Q_l\left(S^{{0}}_{D_{l}}\right)^{-1}(1)(s)+\left[\left(S^{{0}}_{D_{l}}\right)^{-1}\sum_{{m=1} \atop {m \neq l}}^{M} S^{\kappa_{0}}_{D_{m}} \phi_{m}\right]\nonumber\\
         &-\left(S^{{0}}_{D_{l}}\right)^{-1}\left[S^{{d_{\kappa_l}}}_{D_{l}}\left(S^{{0}}_{D_{l}}\right)^{-1}\left(\sum_{{m=1}\atop m\neq l }^{M} S^{\kappa_{0}}_{D_{m}} \phi_{m}\right)\right]
   +O\left(\sum_{{m=1}\atop m\neq l }^{M}\frac{a^3}{d_{lm}^2}\|\phi_m\|\right)+O(a)+O(a^2\|\phi_l\|) \mbox{ in } L^2.
\label{mult-obs-102r-neq}
\end{align}

For a better understanding of the dominating terms in \eqref{mult-obs-102r-neq}, we next estimate the last two terms.\\ 
Using Taylor series expansion, for $s\in \partial D_l$ and $t\in \partial D_m$, we can write
\begin{align}\label{taylor-fund}
\Phi_{\kappa_0}(s,t)&=\Phi_{\kappa_0}(z_l,t)+(s-z_l)\cdot\nabla_s\Phi_{\kappa_0}(z_l,t)+\frac{1}{2}(s-z_l)^2 \cdot \nabla_s\nabla_s\Phi_{\kappa_0}(z_l,t)\\
&\qquad+\sum_{|\alpha|=3}\frac{|\alpha|}{\alpha!}(s-z_l)^\alpha\int_{0}^1(1-\beta)^{|\alpha|-1}D^{\alpha}_s\Phi_{\kappa_0}(z_l+\beta(s-z_l),t) d\beta\nonumber\\
&=\Phi_{\kappa_0}(z_l,z_m)+(t-z_m)\cdot\nabla_t\Phi_{\kappa_0}(z_l,z_m)+(s-z_l)\cdot\nabla_s\Phi_{\kappa_0}(z_l,z_m)\nonumber\\
&\qquad+\frac{1}{2}(s-z_l)^2 \cdot \nabla_s\nabla_s\Phi_{\kappa_0}(z_l,t)\nonumber\\
&\qquad+\sum_{|\alpha|=2}\frac{|\alpha|}{\alpha!}(t-z_m)^\alpha\int_{0}^1(1-r)^{|\alpha|-1}D^{\alpha}_t\Phi_{\kappa_0}(z_l,z_m+r(t-z_m))\, dr\nonumber\\
&\qquad+(s-z_l)\cdot\sum_{|\alpha|=1}\frac{|\alpha|}{\alpha!}(t-z_m)^\alpha\int_{0}^1(1-r)^{|\alpha|-1}D^{\alpha}_t\nabla_s\Phi_{\kappa_0}(z_l,z_m+r(t-z_m))\, dr\nonumber\\
&\qquad+\sum_{|\alpha|=3}\frac{|\alpha|}{\alpha!}(s-z_l)^\alpha\int_{0}^1(1-\beta)^{|\alpha|-1}D^{\alpha}_s\Phi_{\kappa_0}(z_l+\beta(s-z_l),t) \,d\beta,\nonumber
\end{align}
which gives us
\begin{align}\label{expansion-SDml}
(S^{\kappa_{0}}_{D_{m}} \phi_{m})(s)
&=\Phi_{\kappa_0}(z_l,z_m)\,Q_m+\nabla_t\Phi_{\kappa_0}(z_l,z_m)\,\cdot\,V_m+(s-z_l)\cdot\nabla_s\Phi_{\kappa_0}(z_l,z_m)\,Q_m\\
&+\int_{\partial D_m}\Bigg(\frac{1}{2}(s-z_l)^2 \cdot \nabla_s\nabla_s\Phi_{\kappa_0}(z_l,t)\nonumber\\
&+\sum_{|\alpha|=2}\frac{|\alpha|}{\alpha!}(t-z_m)^\alpha\int_{0}^1(1-r)^{|\alpha|-1}D^{\alpha}_t\Phi_{\kappa_0}(z_l,z_m+r(t-z_m))\, dr\nonumber\\
&+(s-z_l)\cdot\sum_{|\alpha|=1}\frac{|\alpha|}{\alpha!}(t-z_m)^\alpha\int_{0}^1(1-r)^{|\alpha|-1}D^{\alpha}_t\nabla_s\Phi_{\kappa_0}(z_l,z_m+r(t-z_m))\, dr\nonumber\\
&+\sum_{|\alpha|=3}\frac{|\alpha|}{\alpha!}(s-z_l)^\alpha\int_{0}^1(1-\beta)^{|\alpha|-1}D^{\alpha}_s\Phi_{\kappa_0}(z_l+\beta(s-z_l),t) \,d\beta \Bigg)\phi_m(t) d\sigma_{m}(t).\nonumber
\end{align}
Using \eqref{InvSsmzl--ori} and \eqref{expansion-SDml}, we have
\begin{align}\label{xyz}
\left(S_{D_l}^0\right)^{-1}(S^{\kappa_{0}}_{D_{m}} \phi_{m})(s)
&=\Phi_{\kappa_0}(z_l,z_m)\,Q_m \left(S_{D_l}^0\right)^{-1}(1)(s)+\nabla_t\Phi_{\kappa_0}(z_l,z_m)\,\cdot\,V_m\left(S_{D_l}^0\right)^{-1}(1)(s)\\
&\qquad+\Big[\left(S_{D_l}^0\right)^{-1}(\cdot-z_l)\Big](s)\cdot\nabla_s\Phi_{\kappa_0}(z_l,z_m)\,Q_m\nonumber\\
&\quad+\left(S_{D_l}^0\right)^{-1}\Bigg(\int_{\partial D_m}\Bigg[\frac{1}{2}(\cdot-z_l)^2 \cdot \nabla_s\nabla_s\Phi_{\kappa_0}(z_l,t)\nonumber\\
&\qquad+\sum_{|\alpha|=2}\frac{|\alpha|}{\alpha!}(t-z_m)^\alpha\int_{0}^1(1-r)^{|\alpha|-1}D^{\alpha}_t\Phi_{\kappa_0}(z_l,z_m+r(t-z_m))\, dr\nonumber\\
&\qquad+(s-z_l)\cdot\sum_{|\alpha|=1}\frac{|\alpha|}{\alpha!}(t-z_m)^\alpha\int_{0}^1(1-r)^{|\alpha|-1}D^{\alpha}_t\nabla_s\Phi_{\kappa_0}(z_l,z_m+r(t-z_m))\, dr\nonumber\\
&\qquad+\sum_{|\alpha|=3}\frac{|\alpha|}{\alpha!}(\cdot-z_l)^\alpha\int_{0}^1(1-\beta)^{|\alpha|-1}D^{\alpha}_s\Phi_{\kappa_0}(z_l+\beta(\cdot-z_l),t) \,d\beta \Bigg]\phi_m(t) d\sigma_{m}(t)\Bigg)(s).\nonumber
\end{align}
Therefore 
\begin{align}\label{expansion-SinvlSDm}
\left(S_{D_l}^0\right)^{-1}(S^{\kappa_{0}}_{D_{m}} \phi_{m})(s)
&=\Phi_{\kappa_0}(z_l,z_m)\,Q_m \left(S_{D_l}^0\right)^{-1}(1)(s)+\nabla_t\Phi_{\kappa_0}(z_l,z_m)\,\cdot\,V_m\left(S_{D_l}^0\right)^{-1}(1)(s)\\
&\qquad+\Big[\left(S_{D_l}^0\right)^{-1}(\cdot-z_l)\Big](s)\cdot\nabla_s\Phi_{\kappa_0}(z_l,z_m)\,Q_m+O\left(\frac{a^3}{d^3_{ml}}\|\phi_m\|\right) \mbox{ in } L^2,\nonumber
\end{align}
where the last term in \eqref{xyz}, which is of order $O\left(\frac{a^4}{d^4_{ml}}\|\phi_m\|\right) $, is absorbed in $O\left(\frac{a^3}{d^3_{ml}}\|\phi_m\|\right) $ as $\frac{a}{d_{ml}}=O(1) $. \\
To estimate the other term in \eqref{mult-obs-102r-neq}, we note that
\begin{align*}
\left(S_{D_l}^0\right)^{-1}S_{D_l}^{d_{\kappa_l}}\left(S_{D_l}^0\right)^{-1}(S^{\kappa_{0}}_{D_{m}} \phi_{m})(s)
&=
\left(S_{D_l}^0\right)^{-1}\Bigg(\frac{i\kappa_l}{4\pi}\int_{\partial D_l}\left(S_{D_l}^0\right)^{-1}(S^{\kappa_{0}}_{D_{m}} \phi_{m})(t)\,d\sigma_{l}(t)\\
&\qquad+\int_{\partial D_l}\sum_{n=2}^\infty\frac{1}{4\pi}\frac{(i\kappa_l)^n}{n!}|s-t|^{n-1}\left(S_{D_l}^0\right)^{-1}(S^{\kappa_{0}}_{D_{m}} \phi_{m})(t)\,d\sigma_{l}(t)\Bigg)\\
&=
\left(S_{D_l}^0\right)^{-1}\Bigg(\frac{i\kappa_l}{4\pi}\int_{\partial D_l}\left(S_{D_l}^0\right)^{-1}(S^{\kappa_{0}}_{D_{m}} \phi_{m})(t)\,d\sigma_{l}(t)\\
&\qquad+\frac{(i\kappa_l)^2}{8\pi}\int_{\partial D_l}|s-t|\left(S_{D_l}^0\right)^{-1}(S^{\kappa_{0}}_{D_{m}} \phi_{m})(t)\,d\sigma_{l}(t)\\
&\qquad+\int_{\partial D_l}\sum_{n=3}^\infty\frac{1}{4\pi}\frac{(i\kappa_l)^n}{n!}|s-t|^{n-1}\left(S_{D_l}^0\right)^{-1}(S^{\kappa_{0}}_{D_{m}} \phi_{m})(t)\,d\sigma_{l}(t)\Bigg)\\
&=
\left(S_{D_l}^0\right)^{-1}\Bigg(\frac{i\kappa_l}{4\pi}\int_{\partial D_l}\left(S_{D_l}^0\right)^{-1}(S^{\kappa_{0}}_{D_{m}} \phi_{m})(t)\,d\sigma_{l}(t)\\
&+\frac{(i\kappa_l)^2}{8\pi}\int_{\partial D_l}|s-t|\left(S_{D_l}^0\right)^{-1}(S^{\kappa_{0}}_{D_{m}} \phi_{m})(t)\,d\sigma_{l}(t)\Bigg)
+O(\frac{a^3}{d_{ml}^2}\|\phi_m\|), \mbox{ in $L^2$ }.
\end{align*}
Using \eqref{expansion-SinvlSDm}, we can further write this as
\begin{align*}
&\left(S_{D_l}^0\right)^{-1}S_{D_l}^{d_{\kappa_l}}\left(S_{D_l}^0\right)^{-1}(S^{\kappa_{0}}_{D_{m}} \phi_{m})(s)\\
&\quad =
\frac{i\kappa_l}{4\pi}\left(S_{D_l}^0\right)^{-1}\Bigg(\Phi_{\kappa_0}(z_l,z_m)\,Q_m Cap_l+\nabla_t\Phi_{\kappa_0}(z_l,z_m)\,\cdot\,V_m Cap_l\\
&\qquad+\nabla_s\Phi_{\kappa_0}(z_l,z_m)\,Q_m\cdot\int_{\partial D_l}\Big[\left(S_{D_l}^0\right)^{-1}(\cdot-z_l)\Big](t) d\sigma_{l}(t)+O\left(\frac{a^4}{d^3_{ml}}\|\phi_m\|\right)_{L^\infty}\Bigg)\nonumber\\
&\qquad+\frac{(i\kappa_l)^2}{8\pi}\left(S_{D_l}^0\right)^{-1}\Bigg(\int_{\partial D_l}|s-t|\left(S_{D_l}^0\right)^{-1}(S^{\kappa_{0}}_{D_{m}} \phi_{m})(t)\,d\sigma_{l}(t)\Bigg)+O\Big(\frac{a^3}{d_{ml}^2}\|\phi_m\|\Big) \\
&\quad =
\frac{i\kappa_l}{4\pi}\Bigg(\Phi_{\kappa_0}(z_l,z_m)\,Q_m Cap_l+\nabla_t\Phi_{\kappa_0}(z_l,z_m)\,\cdot\,V_m Cap_l \\
&\qquad+\nabla_s\Phi_{\kappa_0}(z_l,z_m)\,Q_m\cdot\int_{\partial D_l}\Big[\left(S_{D_l}^0\right)^{-1}(\cdot-z_l)\Big](t) d\sigma_{l}(t)+O\left(\frac{a^4}{d^3_{ml}}\|\phi_m\|\right)_{L^\infty}\Bigg)\left(S_{D_l}^0\right)^{-1}(1)\\
&\qquad+\frac{(i\kappa_l)^2}{8\pi}\left(S_{D_l}^0\right)^{-1}\Bigg(\int_{\partial D_l}|s-t|\left(S_{D_l}^0\right)^{-1}(S^{\kappa_{0}}_{D_{m}} \phi_{m})(t)\,d\sigma_{l}(t)\Bigg)+O\Big(\frac{a^3}{d_{ml}^2}\|\phi_m\|\Big), \mbox{ in $L^2$ },
\end{align*}
where by $(\cdot)_{L^{\infty}} $ we mean the point-wise error estimate. Therefore, using \eqref{InvSsmzl--ori} and the fact $V_m=O(a^2 \norm{\phi_m}) $, we get 
\begin{align}\label{expansion-SinvDdSinvlSDm}
\left(S_{D_l}^0\right)^{-1}S_{D_l}^{d_{\kappa_l}}\left(S_{D_l}^0\right)^{-1}(S^{\kappa_{0}}_{D_{m}} \phi_{m})(s) &=
\frac{i\kappa_l}{4\pi}\Phi_{\kappa_0}(z_l,z_m)\,Q_m Cap_l\left(S_{D_l}^0\right)^{-1}(1)(s)\\
&+\frac{(i\kappa_l)^2}{8\pi}\left(S_{D_l}^0\right)^{-1}\Bigg(\int_{\partial D_l}|s-t|\left(S_{D_l}^0\right)^{-1}(S^{\kappa_{0}}_{D_{m}} \phi_{m})(t)\,d\sigma_{l}(t)\Bigg)+O(\frac{a^3}{d_{ml}^2}\|\phi_m\|)\nonumber\\
&=
\frac{i\kappa_l}{4\pi}\Phi_{\kappa_0}(z_l,z_m)\,Q_m Cap_l\left(S_{D_l}^0\right)^{-1}(1)(s)+O(\frac{a^3}{d_{ml}^2}\|\phi_m\|), \mbox{ in $L^2$ },\nonumber
\end{align}
where we have also used the estimate 
\[
 \left\|\left(S_{D_l}^0\right)^{-1}\Bigg( \int_{\partial D_l}|\cdot-t| \left(S_{D_l}^0\right)^{-1}(S^{\kappa_{0}}_{D_{m}} \phi_{m})(t)) d\sigma_{l}(t)\Bigg)(s)\right\|_{L^2(\partial D_l)}
 =O\left(\frac{a^4}{d_{ml}}\|{\phi}_m\|\right),
 \]
the proof of which follows by a similar argument as in the proof of lemma \ref{scaling-1}. \\

Now, by making use of (\ref{expansion-SinvlSDm}-\ref{expansion-SinvDdSinvlSDm}) in \eqref{mult-obs-102r-neq}, we obtain
\begin{align}\label{mult-obs-102r3-1-3}
 \psi_{l}(s)
         =&\phi_l(s)+\left(S^{{0}}_{D_{l}}\right)^{-1}\left(u^{I}(z_l)\right)  +\frac{i(\kappa_0-\kappa_l)}{4\pi}Q_l\left(S^{{0}}_{D_{l}}\right)^{-1}(1)(s)\\
&+\sum_{{m=1} \atop {m \neq l}}^{M}\Bigg[\Phi_{\kappa_0}(z_l,z_m)\,Q_m \left(S_{D_l}^0\right)^{-1}(1)(s)+\nabla_t\Phi_{\kappa_0}(z_l,z_m)\,\cdot\,V_m\left(S_{D_l}^0\right)^{-1}(1)(s)\nonumber\\
&\qquad+\Big[\left(S_{D_l}^0\right)^{-1}(\cdot-z_l)\Big](s)\cdot\nabla_s\Phi_{\kappa_0}(z_l,z_m)\,Q_m-\frac{i\kappa_l}{4\pi}\Phi_{\kappa_0}(z_l,z_m)\,Q_m Cap_l\left(S_{D_l}^0\right)^{-1}(1)\Bigg]\nonumber\\
         &+O(a)+O(a^2\|\phi_l\|)+O\left(\sum_{{m=1}\atop m\neq l }^{M}\frac{a^3}{d_{lm}^3}\|\phi_m\|\right), \mbox{ in } L^2,\nonumber
\end{align}
whence \eqref{claim- diff-phpsi} follows.
The approximation \eqref{claim-int- diff-phpsi} can then be obtained by integrating \eqref{mult-obs-102r3-1-3} on $\partial D_l$ as follows.
\begin{align}\label{mult-obs-102r3-1-4}
\int_{\partial D_l}(\psi_{l}-\phi_{l})
=&Cap_l\ u^{I}(z_l)+\frac{i(\kappa_0-\kappa_l)}{4\pi}Q_lCap_l\\
&+\sum_{{m=1} \atop {m \neq l}}^{M}\Bigg[\Phi_{\kappa_0}(z_l,z_m)\,Q_m Cap_l+\nabla_t\Phi_{\kappa_0}(z_l,z_m)\,\cdot\,V_mCap_l\nonumber\\
&\qquad-\frac{i\kappa_l}{4\pi}\Phi_{\kappa_0}(z_l,z_m)\,Q_m Cap_l^2+\nabla_s\Phi_{\kappa_0}(z_l,z_m)\,Q_m\cdot\int_{\partial D_l}\left(S_{D_l}^0\right)^{-1}(\cdot-z_l)(s) d\sigma_{l}(s) \Bigg]      \nonumber\\
&+\underbrace{O\left(a^2+a^3\|\phi_l\|+\sum_{{m=1}\atop m\neq l }^{M}\frac{a^4}{d_{lm}^3}\|\phi_m\|\right)}_{=:Err7_{l}},\nonumber
\end{align}
where $Err7_{l} $ is in the point-wise sense.
%
%
\subsection{Proof of lemma \ref{claim-g}}\label{proof-claim-g}
The proof of \eqref{claim-adjdblle-Lay-dif-g} follows as in the case of \eqref{claim-adjdblle-Lay-dif-prime}.
To prove \eqref{def-Sg-c}, we note that from the definition of $g_l$, we have 
\begin{align*}
 g_l (s) =& \frac{i}{4\pi}(\kappa_0-\kappa_l)(S^{\kappa_{0}}_{D_{l}})^{-1}\left(\int_{\partial D_l}\psi_l\right)(s) \\
 =&\frac{i}{4\pi}(\kappa_0-\kappa_l)\left[(S^{0}_{D_{l}})^{-1}+\sum_{n=1}^{\infty}
 (-1)^{n}(S^{0}_{D_{l}})^{-1}\left(S^{d_{\kappa_{0}}}_{D_{l}}(S^{0}_{D_{l}})^{-1}\right)^n\right]\left(\int_{\partial D_l}\psi_l\right)(s) \\
  =&\frac{i}{4\pi}(\kappa_0-\kappa_l)(S^{0}_{D_{l}})^{-1}\left(\int_{\partial D_l}\psi_l\right)(s)+O(a^2\|\psi_l\|), \mbox{ in } L^2,
  \end{align*}
where we use \eqref{Sinv_int_difpts_psi} applied to $\psi_{l} $ with $n=0$ and the fact that for any $f \in L^{2}(\partial D_{l}) $, $\norm{S^{d_{\kappa_{0}}}_{D_l} f}_{H^{1}(\partial D_{l})}=O(a^2 \norm{f}_{L^2(\partial D_{l})}) $.  \\
Using \eqref{mult-obs-102r3-1-3} and \eqref{mult-obs-102r3-1-4}, we can write this as 
\begin{align}\label{est-calc-Sg}  
 g_l(s) =&\frac{i}{4\pi}(\kappa_0-\kappa_l)\left[\int_{\partial D_l}\psi_l\right](S^{0}_{D_{l}})^{-1}\left(1\right)(s)+O\left(a^2+a^2\|\phi_l\|+\sum_{m\neq\,l}\frac{a^3}{d_{ml}}\norm{\phi_m}\right), \mbox{ in } L^2 \\
        =&\frac{i}{4\pi}(\kappa_0-\kappa_l)\left[Cap_l\ u^{I}(z_l)+\int_{\partial D_l}\phi_l+\sum_{m\neq\,l}\Phi_{\kappa_0}(z_l,z_m)\,Q_m Cap_l\right](S^{0}_{D_{l}})^{-1}\left(1\right)(s)\nonumber\\
        &+O\left(a^2+a^2\|\phi_l\|+\sum_{m\neq\,l}\frac{a^3}{d^2_{ml}}\|\phi_m\|\right), \mbox{ in } L^2,\nonumber
\end{align}
whence \eqref{def-Sg-c} follows.
\subsection{Proof of Proposition \ref{Est-intAlminhatAlphil}}\label{proof-Al-hat}
First of all, we note that using the definition of $g_l$ and $\tilde{g}_l$ and \eqref{claim-single-Lay-dif-2}, \eqref{mult-obs-108r-int} can be rewritten as
\begin{align}\label{mult-obs-108r1}
\phi_{l}=&  -[\lambda_l \ Id + (K^{{0}}_{D_{l}})^{*}]^{-1} \Big[ \frac{\partial u^{I}}{\partial \nu^{l}}
+\sum_{{m=1} \atop {m \neq l}}^{M} \frac{\partial (S^{\kappa_{0}}_{D_{m}} \phi_{m})}{\partial \nu^{l}}\Big\vert_{\partial D_{l}} +[(K^{{\kappa_0}}_{D_{l}})^{*} - (K^{{0}}_{D_{l}})^{*}] \phi_{l}\\
& +\Big(1-\frac{\rho_{l}}{\rho_{0}} \Big)^{-1}[(K^{\kappa_{l}}_{D_{l}})^{*}-(K^{\kappa_{0}}_{D_{l}})^{*}] \psi_{l}  \nonumber\\
& +\Big(1-\frac{\rho_{l}}{\rho_{0}} \Big)^{-1}[\frac{1}{2}Id+(K^{\kappa_{0}}_{D_{l}})^{*}] \left(S_{D_l}^{\kappa_0}\right)^{-1}\left(S_{D_l}^{\kappa_0}-S_{D_l}^{\kappa_l}\right)\psi_l \Big], 
\mbox{  on } \partial D_l.\nonumber
\end{align}
 Then we have
 \begin{equation}
\begin{aligned}
& \int_{\partial D_l}\Big(A_l(s)-\hat{A}_l\Big)\phi_{l}(s)\ d\sigma_{l}(s)\\
&=-\int_{\partial D_l} \Big[\Big(\lambda_l\ Id+K^{{0}}_{D_{l}} \Big)^{-1}\Big(A_l(\cdot)-\hat{A}_l\Big)\Big](s)  \left[ \frac{\partial u^{I}}{\partial \nu^{l}} (s) + \Big(1-\frac{\rho_{l}}{\rho_{0}} \Big)^{-1}\Bigg([(K^{\kappa_{l}}_{D_{l}})^{*}-(K^{\kappa_{0}}_{D_{l}})^{*}] \psi_{l}(s) \right.\\
&\left.+[\frac{1}{2}Id+(K^{\kappa_{0}}_{D_{l}})^*]  \left(S_{D_l}^{\kappa_0}\right)^{-1}\left(S_{D_l}^{\kappa_0}-S_{D_l}^{\kappa_l}\right)\psi_l (s)
 \Bigg)+\sum_{{m=1} \atop {m \neq l}}^{M} \frac{\partial (S^{\kappa_{0}}_{D_{m}} \phi_{m})}{\partial \nu^{l}}\Big\vert_{\partial D_{l}}(s)+[(K^{{\kappa_0}}_{D_{l}})^{*} - (K^{{0}}_{D_{l}})^{*}] \phi_{l}(s)\right] d\sigma_{l}(s).
\end{aligned}
\label{mult-obs-108r-pr-beg}
\end{equation}
To estimate the terms in the right hand side of equation \eqref{mult-obs-108r-pr-beg}, we make use of \eqref{estimate-Al} and the facts that $A_{l}-\hat{A}_{l} \in L^{2}_{0}$, and the operator $\Big(\lambda_l\ Id +K^{{0}}_{D_{l}}\Big)^{-1}$ maps $L_0^2$ to $L_0^2$ and is uniformly bounded with respect to $a$ and proceed as follows.
%
By Cauchy-Schwarz inequality and as $\norm{A_l-\hat{A}_l}_{L^2(\partial D_l)}=O(a^3) $, we have
  \begin{align} \label{clarify-1}
     &\int_{\partial D_l}\Big[\Big(\lambda_{l} Id+ K^{0}_{D_{l}} \Big)^{-1}(A_l(\cdot)-\hat{A}_l)\Big] (s) \frac{\partial u^{I}}{\partial \nu^l}(s) d\sigma_{l}(s) =O(a^4),
  \end{align}
%
and as $\norm{(K^{\kappa_{l}}_{D_{l}})^{*}-(K^{\kappa_{0}}_{D_{l}})^{*}}_{L^2(\partial D_l)}=O(a^2) $, we get 
  \begin{align}\label{clarify-1prps}
     \int_{\partial D_l} \Big[\Big(\lambda_{l} Id+ K^{0}_{D_{l}} \Big)^{-1}(A_l(\cdot)-\hat{A}_l)\Big] (s)\  [(K^{\kappa_{l}}_{D_{l}})^{*}-(K^{\kappa_{0}}_{D_{l}})^{*}] \psi_{l}(s) d\sigma_{l}(s) &=O(a^3\cdot\,a^2\|\psi_l\|)=O(a^5\|\psi_l\|).
  \end{align}
  Similarly, we have
    \begin{align}\label{clarify-1prph}
     \int_{\partial D_l} \Big[\Big(\lambda_{l} Id+ K^{0}_{D_{l}} \Big)^{-1}(A_l(\cdot)-\hat{A}_l)\Big] (s)\  [(K^{\kappa_{0}}_{D_{l}})^{*}-(K^{{0}}_{D_{l}})^{*}] \phi_{l}(s) d\sigma_{l}(s)&=O(a^5\|\phi_l\|).
  \end{align}
 %
 Next for $s \in \partial D_{l} $, we use the expansion $\frac{\partial (S^{\kappa_{0}}_{D_{m}} \phi_{m})}{\partial \nu^{l}} (s)=\nabla_x\Phi_{\kappa_0}(z_l,z_m)\cdot\nu_l(s) Q_m+O\left(\frac{a^3}{d_{ml}^3}\|\phi_m\|\right)$ with the error estimate in the $L^2$ sense, 
 to deduce that
   \begin{align}\label{clarify-2}
     &\int_{\partial D_l}\Big[\Big(\lambda_{l} Id+ K^{0}_{D_{l}} \Big)^{-1}(A_l(\cdot)-\hat{A}_l)\Big] (s)\  \sum_{{m=1} \atop {m \neq l}}^{M} \frac{\partial (S^{\kappa_{0}}_{D_{m}} \phi_{m})}{\partial \nu^{l}} (s)    d\sigma_{l}(s)\\
     &\qquad \qquad = R_l\cdot \sum_{{m=1} \atop {m \neq l}}^{M} \nabla_s\Phi_{\kappa_0}(z_l,z_m) Q_m+O\left(a^3\sum_{{m=1}\atop {m\neq l}}^{M}\frac{a^3}{d_{ml}^3}\|\phi_m\|\right),\nonumber
  \end{align}
  where the term $R_l$ behaves as $O(a^4)$ and is defined as
  \begin{align} \label{def-Rl}
  R_l:=\int_{\partial D_l} \Big[\Big(\lambda_{l} Id+ K^{0}_{D_{l}} \Big)^{-1}(A_l(\cdot)-\hat{A}_l)\Big] (s)\nu_l(s)\,d\sigma_{l}(s).
  \end{align}
%
%
The remaining term can be dealt with in the following manner. First we write 
\begin{align*}
&\int_{\partial D_l}\Big[\Big(\lambda_l Id+K^{{0}}_{D_{l}}\Big)^{-1}\Big(A_l(\cdot)-\hat{A}_l\Big)\Big](s)\ [\frac{1}{2}Id+(K^{\kappa_{0}}_{D_{l}})^{*}] \left(S_{D_l}^{\kappa_0}\right)^{-1}\left(S_{D_l}^{\kappa_0}-S_{D_l}^{\kappa_l}\right)\psi_l (s)\ d\sigma_{l}(s)\\
&=\Big(\frac{1}{2}-\lambda_l\Big)\int_{\partial D_l}\Big[\Big(\lambda_l Id+K^{{0}}_{D_{l}}\Big)^{-1}\Big(A_l(\cdot)-\hat{A}_l\Big)\Big](s) \  \left(S_{D_l}^{\kappa_0}\right)^{-1}\left(S_{D_l}^{\kappa_0}-S_{D_l}^{\kappa_l}\right)\psi_l (s)\ d\sigma_{l}(s)\\
&\qquad+\int_{\partial D_l}\Big(A_l(s)-\hat{A}_l\Big)\ \left(S_{D_l}^{\kappa_0}\right)^{-1}\left(S_{D_l}^{\kappa_0}-S_{D_l}^{\kappa_l}\right)\psi_l (s)\ d\sigma_{l}(s)\\
&\qquad+\int_{\partial D_l} \Big[\Big(\lambda_l Id+K^{{0}}_{D_{l}}\Big)^{-1}\Big(A_l(\cdot)-\hat{A}_l\Big)\Big](s) \ [(K^{\kappa_{0}}_{D_{l}})^{*}-(K^{{0}}_{D_{l}})^{*}]  \left(S_{D_l}^{\kappa_0}\right)^{-1}\left(S_{D_l}^{\kappa_0}-S_{D_l}^{\kappa_l}\right)\psi_l (s)\ d\sigma_{l}(s)\\
&=\Big(\frac{1}{2}-\lambda_l\Big)\int_{\partial D_l} \Big[\Big(\lambda_l Id+K^{{0}}_{D_{l}}\Big)^{-1}\Big(A_l(\cdot)-\hat{A}_l\Big)\Big](s)\   \sum_{n=0}^{\infty}(-1)^{n}\left(\left(S_{D_l}^{0}\right)^{-1}S_{D_l}^{d_{\kappa_0}}\right)^n\left(S_{D_l}^{0}\right)^{-1}\left(S_{D_l}^{\kappa_0}-S_{D_l}^{\kappa_l}\right)\psi_l (s)d\sigma_{l}(s)\\
&\qquad+\int_{\partial D_l}\Big(A_l(\cdot)-\hat{A}_l\Big)(s) \  \left[Id+\sum_{n=1}^{\infty}(-1)^{n}\left(\left(S_{D_l}^{0}\right)^{-1}S_{D_l}^{d_{\kappa_0}}\right)^n\right]\left(S_{D_l}^{0}\right)^{-1}\left(S_{D_l}^{\kappa_0}-S_{D_l}^{\kappa_l}\right)\psi_l (s)d\sigma_{l}(s)\\
&\qquad+\int_{\partial D_l}\Big[\Big(\lambda_l Id+K^{{0}}_{D_{l}}\Big)^{-1}\Big(A_l(\cdot)-\hat{A}_l\Big)\Big](s)\ [(K^{\kappa_{0}}_{D_{l}})^{*}-(K^{{0}}_{D_{l}})^{*}]  \left(S_{D_l}^{\kappa_0}\right)^{-1}\left(S_{D_l}^{\kappa_0}-S_{D_l}^{\kappa_l}\right)\psi_l (s)d\sigma_{l}(s)\\
&=\frac{\rho_l}{\rho_l-\rho_0}O(a^4\|\psi_l\|)+O(a^5\|\psi_l\|)+O(a^6\|\psi_l\|)+\int_{\partial D_l}\Big(A_l(s)-\hat{A}_l\Big)\   \left(S_{D_l}^{0}\right)^{-1}\left(S_{D_l}^{\kappa_0}-S_{D_l}^{\kappa_l}\right)\psi_l (s)d\sigma_{l}(s)\\
 &=\frac{i}{4\pi}(\kappa_0-\kappa_l)\Big(\int_{\partial D_l}\psi_l(t)\ d\sigma_{l}(t) \Big)\, \int_{\partial D_l}\Big(A_l(s)-\hat{A}_l\Big)\ \left(S_{D_l}^{0}\right)^{-1}(1) (s) d\sigma_{l}(s) +O(a^5\|\psi_l\|)\\
 &\qquad \qquad +\int_{\partial D_l}\Big(A_l(s)-\hat{A}_l\Big)\  \left(S_{D_l}^{0}\right)^{-1} \left[
 \sum_{n=2}^\infty\frac{i^n(\kappa_0^n-\kappa_l^n)}{4\pi n!}\int_{\partial D_l}\vert{t-s}\vert^{n-1}\psi_l(t)\,d\sigma_{l}(t)\right]d\sigma_{l}(s),
 \end{align*}
 where we use lemma \ref{Claimest} and the fact that $\rho_l \sim a^{1+\gamma}, \gamma \geq 0 $. \\
 Using \eqref{Sinv_int_difpts_psi} and \eqref{int-invs-Al}, we can then deduce
 \begin{align}\label{clarify-3}
 &\int_{\partial D_l}\Big[\Big(\lambda_l Id+K^{{0}}_{D_{l}}\Big)^{-1}\Big(A_l(\cdot)-\hat{A}_l\Big)\Big](s)\ [\frac{1}{2}Id+(K^{\kappa_{0}}_{D_{l}})^{*}] \left(S_{D_l}^{\kappa_0}\right)^{-1}\left(S_{D_l}^{\kappa_0}-S_{D_l}^{\kappa_l}\right)\psi_l (s)\ d\sigma_{l}(s)\\
  &=\frac{i}{4\pi}(\kappa_0-\kappa_l)\Big(\int_{\partial D_l}\psi_l(t) d\sigma_{l}(t)\Big)\int_{\partial D_l}\Big(A_l(s)-\hat{A}_l\Big)\  \left(S_{D_l}^{0}\right)^{-1}(1) (s) \ d\sigma_{l}(s) +O(a^5\|\psi_l\|)+O(a^3{\cdot}a^2\|\psi_l\|)\nonumber\\
  &=-\frac{i}{4\pi}(\kappa_0-\kappa_l)\Big(8\pi\vert{D_l}\vert+\hat{A}_lCap_l\Big)\int_{\partial D_l}\psi_l(t)\ d\sigma_{l}(t)
  +O(a^5\|\psi_l\|).\nonumber
\end{align}
%
%
Now by substituting (\ref{clarify-1}-\ref{clarify-3}) in \eqref{mult-obs-108r-pr-beg} and using \eqref{claim-int- diff-phpsi}, we have
\begin{align}\label{mult-obs-108r-pr}
& \int_{\partial D_l}\Big(A_l(s)-\hat{A}_l\Big)\phi_{l}(s) d\sigma_{l}(s)\\
&=-R_l\cdot \sum_{{m=1} \atop {m \neq l}}^{M} \nabla_s\Phi_{\kappa_0}(z_l,z_m) Q_m
+\Big(1-\frac{\rho_{l}}{\rho_{0}} \Big)^{-1} \frac{i}{4\pi}(\kappa_0-\kappa_l)\Big(8\pi\vert{D_l}\vert+\hat{A}_lCap_l\Big)\int_{\partial D_l}\psi_l(s)\ d\sigma_{l}(s) \nonumber\\
&\qquad
 +O\left(a^4+a^3\sum_{{m=1}\atop {m\neq l}}^{M}\frac{a^3}{d_{ml}^3}\|\phi_m\|+a^5\|\phi_l\|+a^5\|\psi_l\|\right)\nonumber\\\
&=-R_l\cdot \sum_{{m=1} \atop {m \neq l}}^{M} \nabla_s\Phi_{\kappa_0}(z_l,z_m) Q_m\nonumber\\\  
&\qquad +\Big(1-\frac{\rho_{l}}{\rho_{0}} \Big)^{-1} \frac{i}{4\pi}(\kappa_0-\kappa_l)\,\big[Q_l+\sum_{m\neq\,l}Cap_l\Phi_{\kappa_0}(z_l,z_m)Q_m\big]\,\Big(8\pi\vert{D_l}\vert+\hat{A}_lCap_l\Big)\nonumber\\\
&\qquad
 +O\left(a^4+a^3\sum_{{m=1}\atop {m\neq l}}^{M}\frac{a^3}{d_{ml}^3}\|\phi_m\|+a^5\|\phi_l\|+a^5\|\psi_l\|\right),\nonumber\
\end{align}
whence \eqref{mult-obs-108r-pr-th} follows.
\subsection{Proof of proposition \ref{estimate-Vm}}\label{Vm} 
We note that using \eqref{mult-obs-108r1}, we can write 
 \begin{align}\label{mult-obs-108r3}
\int_{\partial D_l}(s-z_l)_p\;&\phi_{l}\;d\sigma_{l}(s)=  -\underbrace{\int_{\partial D_l}(s-z_l)_p\;[\lambda_l Id+ (K^{{0}}_{D_{l}})^{*}]^{-1}\frac{\partial u^{I}}{\partial \nu^{l}}\;d\sigma_{l}(s)}_{:=E^p_l}\\
&-\underbrace{\sum_{{m=1} \atop {m \neq l}}^{M}\int_{\partial D_l}(s-z_l)_p\;[\lambda_l Id+ (K^{{0}}_{D_{l}})^{*}]^{-1} \frac{\partial (S^{\kappa_{0}}_{D_{m}} \phi_{m})}{\partial \nu^{l}}\Big\vert_{\partial D_{l}} \;d\sigma_{l}(s)}_{:=B^p_{\phi{l}}}\nonumber\\
&-\underbrace{\int_{\partial D_l}(s-z_l)_p\;[\lambda_l Id+ (K^{{0}}_{D_{l}})^{*}]^{-1}[(K^{{\kappa_0}}_{D_{l}})^{*} - (K^{{0}}_{D_{l}})^{*}] \phi_{l}\;d\sigma_{l}(s)}_{=:K^p_{\phi{l}}}\nonumber\\
&-\Big(1-\frac{\rho_{l}}{\rho_{0}} \Big)^{-1}\underbrace{\int_{\partial D_l}(s-z_l)_p\;[\lambda_l Id+ (K^{{0}}_{D_{l}})^{*}]^{-1}  [(K^{\kappa_{l}}_{D_{l}})^{*}-(K^{\kappa_{0}}_{D_{l}})^{*}] \psi_{l}  \;d\sigma_{l}(s)}_{=:K^p_{\psi{l}}}\nonumber\\
&-\Big(1-\frac{\rho_{l}}{\rho_{0}} \Big)^{-1}\underbrace{\int_{\partial D_l}(s-z_l)_p\;[\lambda_l Id+ (K^{{0}}_{D_{l}})^{*}]^{-1} [\frac{1}{2}Id+(K^{\kappa_{0}}_{D_{l}})^{*}] \left(S_{D_l}^{\kappa_0}\right)^{-1}\left(S_{D_l}^{\kappa_0}-S_{D_l}^{\kappa_l}\right)\psi_l \ d\sigma_{l}(s) }_{=:K^p_{Gg{l}}}. \nonumber
\end{align}
Let us recall that $\overline{z}_{l}=\frac{1}{\vert \partial D_{l} \vert} \int_{\partial D_{l}} s \ d\sigma_{l}(s) .$ By definition, we have $(s-\overline{z}_{l})\in (L^{2}_{0}(\partial D_{l}))^3 $. In the sequel, we will repeatedly use this property.\\
 To estimate $E^{p}_{l}$, using $\overline{z}_{l} $ we write 
 \begin{align} \label{corr1000}
 E^{p}_{l}&:= \int_{\partial D_{l}} (s-z_{l})_{p}\ [\lambda_{l} Id+ (K^{0}_{D_{l}})^*]^{-1} \frac{\partial u^{I}}{\partial \nu^{l}} \ d\sigma_{l}(s) \\
 &= \underbrace{\int_{\partial D_{l}} (s-\overline{z}_{l})_{p}\ [\lambda_{l}Id+ (K^{0}_{D_{l}})^*]^{-1} \frac{\partial u^{I}}{\partial \nu^{l}} \ d\sigma_{l}(s)}_{I} + \underbrace{\int_{\partial D_{l}} (\overline{z}_{l}-z_{l})_{p}\ [\lambda_{l}Id+ (K^{0}_{D_{l}})^*]^{-1} \frac{\partial u^{I}}{\partial \nu^{l}} \ d\sigma_{l}(s)}_{II}.\nonumber
 \end{align}
To estimate the term $II$, we proceed as follows. We note that
\begin{align}\label{corr1001}
II &= \int_{\partial D_{l}} (\overline{z}_{l}-z_{l})_{p}\ [\lambda_{l}Id+ (K^{0}_{D_{l}})^*]^{-1} \frac{\partial u^{I}}{\partial \nu^{l}} \ d\sigma_{l}(s) 
=(\overline{z}_{l}-z_{l})_{p} \int_{\partial D_{l}}\ \underbrace{[\lambda_{l}Id+ (K^{0}_{D_{l}})^*]^{-1} \frac{\partial u^{I}}{\partial \nu^{l}}}_{f} \ d\sigma_{l}(s),
\end{align}
and we can write \[[\lambda_{l}Id+ (K^{0}_{D_{l}})^*]f= \frac{\partial u^I}{\partial \nu^{l}}. \]
Now integrating over $\partial D_{l}$, we find that
\begin{align}\label{corr1002}
&\int_{\partial D_{l}} [\lambda_{l}Id+ (K^{0}_{D_{l}})^*]f(s)\ d\sigma_{l}(s) =\int_{\partial D_{l}} \frac{\partial u^I}{\partial \nu^{l}}(s)\ d\sigma_{l}(s) = -\kappa_0^2 \int_{D_{l}} u^I(x)\ dx \nonumber\\
\Rightarrow &\Big(\lambda_{l}-\frac{1}{2} \Big)\int_{\partial D_{l}} f(s)\ d\sigma_{l}(s)= \int_{\partial D_{l}} f(s) \  [\lambda_{l}Id+ K^{0}_{D_{l}}](1) \ d\sigma_{l}(s) = -\kappa_0^2 \int_{D_{l}} u^I(x)\ dx \nonumber\\
\Rightarrow &\int_{\partial D_{l}} f(s)\ d\sigma_{l}(s) = - \Big(\lambda_{l}-\frac{1}{2} \Big)^{-1} \kappa_0^{2} \int_{D_{l}} u^I(x)\ dx.
\end{align}
Using this in \eqref{corr1001}, it follows that
\begin{equation}
II= -(\overline{z}_{l}-z_{l})_{p} \Big(\lambda_{l}-\frac{1}{2} \Big)^{-1} \kappa_0^{2} \int_{D_{l}} u^I(x)\ dx.
\label{corr1003}
\end{equation}
In order to estimate $I$, we note that,
\begin{equation}
I= \int_{\partial D_{l}} (s-\overline{z}_{l})_{p} \ [\lambda_{l} Id+ (K^{0}_{D_{l}})^{*}]^{-1} \frac{\partial u^{I}}{\partial \nu^{l}} \ d\sigma_{l}(s)=\int_{\partial D_{l}} [\lambda_{l} Id+ K^{0}_{D_{l}}]^{-1}(s-\overline{z}_{l})_{p}\ \frac{\partial u^{I}}{\partial \nu^{l}} \ d\sigma_{l}(s).
\notag
\end{equation}
We recall that $\left[\lambda Id+ K^{0}_{D_l} \right]^{-1}: L^{2}_{0}(\partial D_l) \rightarrow  L^{2}_{0}(\partial D_l)$ is uniformly bounded with respect to $\lambda $ for $\vert \lambda \vert \in [\frac{1}{2}, +\infty ) $. Using the fact that $(s-\overline{z}_{l}) $ is mean-free, we can now conclude that
\begin{align}\label{corr1010}
\vert I \vert &\leq \norm{[\lambda_{l} Id+ K^{0}_{D_{l}}]^{-1}(s-\overline{z}_{l})_{p}}_{L^2(\partial D_l)} \Big\Vert\frac{\partial u^{I}}{\partial \nu^{l}}\Big\Vert_{L^2(\partial D_l)} \leq C \norm{(s-\overline{z}_{l})_{p}}_{L^2(\partial D_l)}\Big\Vert \frac{\partial u^{I}}{\partial \nu^{l}} \Big\Vert_{L^2(\partial D_l)} = O(a^3).
\end{align}
Combining \eqref{corr1003} and \eqref{corr1010}, we can therefore write
\begin{align}\label{Ep}
-E^{p}_{l}&= (\overline{z}_{l}-z_{l})_{p} \Big(\lambda_{l}-\frac{1}{2} \Big)^{-1} \kappa_0^{2} \int_{D_{l}} u^I(x)\ dx+O(a^3)\\
&=(\overline{z}_{l}-z_{l})_{p} \Big(\lambda_{l}-\frac{1}{2} \Big)^{-1} \kappa_0^{2} u^{I}(z_{l}) \vert D_{l} \vert+ O(a^3) + O(a^5 \rho_{l}^{-1})=O(a^{3-\gamma}),\nonumber
\end{align}
as $\lambda_l -\frac{1}{2} \sim \rho_l $ and we assume that $\rho_{l}\simeq a^{1+\gamma}$, with $\gamma \geq 0$. \\
Next let us estimate the term $K^p_{Gg{l}} $. 
 \begin{align*}
K^p_{Gg{l}}&=\int_{\partial D_l}(s-z_l)_p\;[\lambda_l Id + (K^{{0}}_{D_{l}})^{*}]^{-1} \Big[\frac{1}{2}Id+(K^{\kappa_{0}}_{D_{l}})^{*} \Big] \left(S_{D_l}^{\kappa_0}\right)^{-1}\left(S_{D_l}^{\kappa_0}-S_{D_l}^{\kappa_l}\right)\psi_l(s)\ d\sigma_{l}(s)\\
&=\int_{\partial D_l}[\lambda_l Id + K^{{0}}_{D_{l}}]^{-1}(s-z_l)_{p} \cdot \Big[\frac{1}{2}Id+(K^{\kappa_{0}}_{D_{l}})^{*} \Big] \left(S_{D_l}^{\kappa_0}\right)^{-1}\left(S_{D_l}^{\kappa_0}-S_{D_l}^{\kappa_l}\right)\psi_l(s)\ d\sigma_{l}(s)\\
&=\Big(\frac{1}{2}-\lambda_l\Big) \int_{\partial D_l}[\lambda_l Id+ K^{{0}}_{D_{l}}]^{-1}(s-z_l)_{p} \cdot \left(S_{D_l}^{\kappa_0}\right)^{-1}\left(S_{D_l}^{\kappa_0}-S_{D_l}^{\kappa_l}\right)\psi_l(s) \ d\sigma_{l}(s) \\
&\quad +\int_{\partial D_l}(s-z_l)_p\;  \left(S_{D_l}^{\kappa_0}\right)^{-1}\left(S_{D_l}^{\kappa_0}-S_{D_l}^{\kappa_l}\right)\psi_l(s) \ d\sigma_{l}(s)\\
&\quad +\int_{\partial D_l}[\lambda_lId + K^{{0}}_{D_{l}}]^{-1}(s-z_l)_p \cdot [(K^{\kappa_{0}}_{D_{l}})^{*}-(K^{{0}}_{D_{l}})^{*}] \left(S_{D_l}^{\kappa_0}\right)^{-1}\left(S_{D_l}^{\kappa_0}-S_{D_l}^{\kappa_l}\right)\psi_l(s)\ d\sigma_{l}(s)\\
&=\frac{1}{2}\frac{\rho_l}{\rho_l-\rho_0}\int_{\partial D_l} (s-z_l)_{p}  [\lambda_l Id + (K^{{0}}_{D_{l}})^{*}]^{-1} \left(S_{D_l}^0\right)^{-1}\sum_{n=0}^{\infty}\left(S_{D_l}^{d_{\kappa_0}}\left(S_{D_l}^{0}\right)^{-1}\right)^n  \left(S_{D_l}^{\kappa_0}-S_{D_l}^{\kappa_l}\right)\psi_l(s)\ d\sigma_{l}(s)\\
&\quad +\int_{\partial D_l} \left(S_{D_l}^{0}\right)^{-1}(s-z_l)_p \cdot  \sum_{n=0}^{\infty}\left(S_{D_l}^{d_{\kappa_0}}\left(S_{D_l}^{0}\right)^{-1}\right)^n\left(S_{D_l}^{\kappa_0}-S_{D_l}^{\kappa_l}\right)\psi_l(s)\ d\sigma_{l}(s) \\
&\quad +\int_{\partial D_l} (s-z_l)_p [\lambda_l Id + (K^{{0}}_{D_{l}})^{*}]^{-1} [(K^{\kappa_{0}}_{D_{l}})^{*}-(K^{{0}}_{D_{l}})^{*}] \left(S_{D_l}^{\kappa_0}\right)^{-1}\left(S_{D_l}^{\kappa_0}-S_{D_l}^{\kappa_l}\right)\psi_l(s)\ d\sigma_{l}(s).
\end{align*}
Then using \eqref{InvSsmzl--ori}, we can deduce that
\begin{align}\label{KpG}
K^p_{Gg{l}}&=O(\rho_{l} a^{2} \rho_{l}^{-1}  a^{-1} a^{2}\norm{\psi_{l}})+O(a\ a^2 \norm{\psi_{l}})+O( a^2 \rho_{l}^{-1} a^2 a^{-1} a^2 \|\psi_l\|)\\
&=O(a^3\|\psi_l\|)+O(a^{5} \rho_{l}^{-1} \norm{\psi_{l}}).\nonumber
\end{align}
%
Next, we shall estimate the term $K^{p}_{\phi_{l}}$. 
\begin{align*}
K^p_{\phi_{l}}&=\int_{\partial D_l}(s-z_l)_p\;[\lambda_l Id + (K^{{0}}_{D_{l}})^{*}]^{-1}[(K^{\kappa_{0}}_{D_{l}})^{*}-(K^{0}_{D_{l}})^{*}]\phi_l(s)\ d\sigma_{l}(s)\\
&=\underbrace{\int_{\partial D_l}(s-\overline{z}_l)_p\;[\lambda_l Id + (K^{{0}}_{D_{l}})^{*}]^{-1}[(K^{\kappa_{0}}_{D_{l}})^{*}-(K^{0}_{D_{l}})^{*}]\phi_l(s)\ d\sigma_{l}(s)}_{K^{p,1}_{\phi_{l}}}\\
&\qquad \qquad \qquad+ \underbrace{\int_{\partial D_l}(\overline{z}_{l}-z_l)_p\;[\lambda_l Id + (K^{{0}}_{D_{l}})^{*}]^{-1}[(K^{\kappa_{0}}_{D_{l}})^{*}-(K^{0}_{D_{l}})^{*}]\phi_l(s)\ d\sigma_{l}(s)}_{K^{p,2}_{\phi_{l}}}.
\end{align*}
Now, the first term in the right hand side above can be estimated in the following manner using the fact that $(s-\overline{z}_{l})_{p} \in L^{2}_{0}(\partial D_{l}) $.
\begin{align*}
K^{p,1}_{\phi_{l}}&=\int_{\partial D_l}(s-\overline{z}_l)_p\;[\lambda_l Id + (K^{{0}}_{D_{l}})^{*}]^{-1}[(K^{\kappa_{0}}_{D_{l}})^{*}-(K^{0}_{D_{l}})^{*}]\phi_l(s)\ d\sigma_{l}(s)\\
&\qquad=\int_{\partial D_l}[\lambda_l Id + K^{{0}}_{D_{l}}]^{-1}(s-\overline{z}_l)_p\cdot [(K^{\kappa_{0}}_{D_{l}})^{*}-(K^{0}_{D_{l}})^{*}] \phi_l(s)\ d\sigma_{l}(s)\\
&\qquad=O(a^2\,a^2\|\phi_l\|)=O(a^4\|\phi_l\|).
\end{align*}
To estimate the second term $K^{p,2}_{\phi_{l}} $, we observe that
\begin{align*}
K^{p,2}_{\phi_{l}}&=(\overline{z}_{l}-z_{l})_{p} \int_{\partial D_{l}} [\lambda_{l}Id+K^{0}_{D_{l}}]^{-1}(1) \cdot [(K^{\kappa_{0}}_{D_{l}})^{*}-(K^{0}_{D_{l}})^{*}] \phi_{l}(s)\ d\sigma_{l}(s)  \\
&= (\overline{z}_{l}-z_{l})_{p} \int_{\partial D_{l}} [K^{\kappa_{0}}_{D_{l}}-K^{0}_{D_{l}}][\lambda_{l}Id+K^{0}_{D_{l}}]^{-1}(1) \cdot \phi_{l}(s)\ d\sigma_{l}(s).
\end{align*}
Let us denote $h:= [\lambda_{l} Id+K^{0}_{D_{l}}]^{-1} (1) $. Then 
\begin{align*}
K^{p,2}_{\phi_{l}}&= (\overline{z}_{l}-z_{l})_{p} \int_{\partial D_{l}} \phi_{l}(s) \Big[\int_{\partial D_{l}} (\nabla_{t}\Phi_{\kappa_{0}}(s,t)-\nabla_{t}\Phi_{0}(s,t))\cdot \nu^{l}(t) h(t)\ d\sigma_{l}(t) \Big] d\sigma_{l}(s) \\
&=(\overline{z}_{l}-z_{l})_{p} \int_{\partial D_{l}} \phi_{l}(s) \Big[\int_{\partial D_{l}} \Big(\sum_{n=2}^{\infty} -\frac{\kappa^{n}_{0} i^{n}}{4 \pi n!} (n-1) \vert s-t \vert^{n-2} \frac{(s-t)}{\vert s-t \vert}  \Big)\cdot \nu^{l}(t) h(t) d\sigma_{l}(t) \Big] d\sigma_{l}(s) \\
&=\frac{1}{8 \pi} \kappa^{2}_{0} (\overline{z}_{l}-z_{l})_{p} \int_{\partial D_{l}} \phi_{l}(s) \Big[\int_{\partial D_{l}} \frac{(s-t)}{\vert s-t\vert}\cdot \nu^{l}(t) h(t)\ d\sigma_{l}(t) \Big] d\sigma_{l}(s) \\
&\quad + \frac{i \kappa_{0}^{3}}{12 \pi} (\overline{z}_{l}-z_{l})_{p} \int_{\partial D_{l}} \phi_{l}(s) \Big[\int_{\partial D_{l}} [(s-t)\cdot \nu^{l}(t)] h(t) d\sigma_{l}(t) \Big] d\sigma_{l}(s) \\
&\quad +(\overline{z}_{l}-z_{l})_{p} \int_{\partial D_{l}} \phi_{l}(s) \Big[\int_{\partial D_{l}} \Big(\sum_{n=4}^{\infty} -\frac{\kappa^{n}_{0} i^{n}}{4 \pi n!} (n-1) \vert s-t \vert^{n-2} \frac{(s-t)}{\vert s-t \vert}  \Big)\cdot \nu^{l}(t) h(t) d\sigma_{l}(t) \Big] d\sigma_{l}(s).
\end{align*}
Note that the third term in the right hand side of the above identity is of order $O(a^{6} \rho_{l}^{-1} \norm{\phi_{l}}) $. To understand the first two terms better, we now proceed as follows.\\
Let us write 
\[\int_{\partial D_{l}} [(s-t)\cdot \nu^{l}(t)] h(t) d\sigma_{l}(t)= \int_{\partial D_{l}} \underbrace{[\lambda_{l} Id+(K^{0}_{D_{l}})^{*}]^{-1} [(s-t)\cdot \nu^{l}(t)]}_{f_{1}}\ d\sigma_{l}(t).\]
Now $[\lambda_{l} Id+(K^{0}_{D_{l}})^{*}]f_{1}= (s-t)\cdot \nu^{l}(t) $ and hence 
\begin{align*}
&\int_{\partial D_{l}} [\lambda_{l} Id+(K^{0}_{D_{l}})^{*}] f_{1}(t)\ d\sigma_{l}(t) = \int_{\partial D_{l}} (s-t)\cdot \nu^{l}(t) d\sigma_{l}(t) \\
&\Longrightarrow \Big(\lambda_{l}-\frac{1}{2}\Big) \int_{\partial D_{l}} f_{1}(t)\ d\sigma_{l}(t)=\int_{\partial D_{l}} f_{1}(t) [\lambda_{l} Id+K^{0}_{D_{l}}](1)\ d\sigma_{l}(t)= \int_{\partial D_{l}} (s-t)\cdot \nu^{l}(t) d\sigma_{l}(t) =-3 \vert D_{l} \vert  \\
&\Longrightarrow \int_{\partial D_{l}} [(s-t)\cdot \nu^{l}(t)]h(t) d\sigma_{l}(t) =-3 \Big(\lambda_{l}-\frac{1}{2}\Big)^{-1} \vert D_{l} \vert .
\end{align*}
Therefore
\begin{align*}
&\frac{i \kappa_{0}^{3}}{12 \pi} (\overline{z}_{l}-z_{l})_{p} \int_{\partial D_{l}} \phi_{l}(s) \Big[\int_{\partial D_{l}} [(s-t)\cdot \nu^{l}(t)] h(t) d\sigma_{l}(t)\Big] d\sigma_{l}(s) =-\frac{i \kappa_{0}^{3}}{4 \pi} \Big(\lambda_{l}-\frac{1}{2} \Big)^{-1} \vert D_{l} \vert (\overline{z}_{l}-z_{l})_{p} \int_{\partial D_{l}} \phi_{l}(s) d \sigma_{l}(s).
\end{align*}
Next, we consider the term $\frac{1}{8 \pi} \kappa^{2}_{0} (\overline{z}_{l}-z_{l})_{p} \int_{\partial D_{l}} \phi_{l}(s) \Big[\int_{\partial D_{l}} \frac{(s-t)}{\vert s-t\vert}\cdot \nu^{l}(t) h(t)\ d\sigma_{l}(t) \Big] d\sigma_{l}(s)$. Again we write
\begin{equation}
\int_{\partial D_{l}} \frac{(s-t)}{\vert s-t\vert}\cdot \nu^{l}(t) h(t)\ d\sigma_{l}(t) = \int_{\partial D_{l}} \underbrace{[\lambda_{l} Id+ (K^{0}_{D_{l}})^{*}]^{-1} \Big(\frac{(s-t)}{\vert s-t \vert}\cdot \nu^{l}(t) \Big)}_{f_{2}}\ d\sigma_{l}(t),
\notag
\end{equation}
then arguing as in the previous case, we obtain
\begin{align*}
&\int_{\partial D_{l}} f_{2}(t) \ d\sigma_{l}(t) = \Big(\lambda_{l}-\frac{1}{2} \Big)^{-1} \int_{\partial D_{l}} \frac{s-t}{\vert s-t \vert} \cdot \nu^{l}(t) \ d\sigma_{l}(t)=\Big(\lambda_{l}-\frac{1}{2} \Big)^{-1} A_{l}(s) 
\end{align*}
and hence using \eqref{mult-obs-108r-pr-th}, it follows that
\begin{align*}
&\frac{1}{8 \pi} \kappa^{2}_{0} (\overline{z}_{l}-z_{l})_{p} \int_{\partial D_{l}} \phi_{l}(s) \Big[\int_{\partial D_{l}} \frac{(s-t)}{\vert s-t\vert}\cdot \nu^{l}(t) h(t)\ d\sigma_{l}(t) \Big] d\sigma_{l}(s)\\
&=\frac{1}{8 \pi} \kappa_{0}^{2} (\overline{z}_{l}-z_{l})_{p} (\lambda_{l}-\frac{1}{2})^{-1} \int_{\partial D_{l}} A_{l}(s) \phi_{l}(s)\ d\sigma_{l}(s) \\
&=\frac{1}{8 \pi} \kappa_{0}^{2} (\overline{z}_{l}-z_{l})_{p} (\lambda_{l}-\frac{1}{2})^{-1} \hat{A}_{l} Q_{l}+\frac{1}{8 \pi} \kappa_{0}^{2} (\overline{z}_{l}-z_{l})_{p} (\lambda_{l}-\frac{1}{2})^{-1} \int_{\partial D_{l}} (A_{l}-\hat{A}_{l})(s) \phi_{l}(s)\ d\sigma_{l}(s) \\
&=\frac{1}{8 \pi} \kappa_{0}^{2} (\overline{z}_{l}-z_{l})_{p} (\lambda_{l}-\frac{1}{2})^{-1} \hat{A}_{l} Q_{l}  +\frac{1}{8 \pi} \kappa_{0}^{2} (\overline{z}_{l}-z_{l})_{p} (\lambda_{l}-\frac{1}{2})^{-1} \Big[-R_l\cdot \sum_{{m=1} \atop {m \neq l}}^{M} \nabla_s\Phi_{\kappa_0}(z_l,z_m) Q_m\\
 &\qquad \qquad \qquad \qquad \qquad \qquad-\frac{\rho_{0}}{\rho_{l}-\rho_{0}}  \frac{i}{4\pi}(\kappa_0-\kappa_l)\big[Q_l+\sum_{m\neq\,l}Cap_l\Phi_{\kappa_0}(z_l,z_m)Q_m\big]\Big(8\pi\vert{D_l}\vert+\hat{A}_lCap_l\Big)\\
 &\qquad \qquad \qquad \qquad \qquad \qquad \qquad+O\left(a^4+a^3\sum_{{m=1}\atop {m\neq l}}^{M}\frac{a^3}{d_{ml}^3}\|\phi_m\|+a^5\|\phi_l\|+a^5\|\psi_l\|\right) \Big].
\end{align*}
Combining the estimates above, we finally have
\begin{align}\label{Kphi}
-K^{p}_{\phi_{l}}&=\frac{i \kappa_{0}^{3}}{4 \pi} \Big(\lambda_{l}-\frac{1}{2} \Big)^{-1} \vert D_{l} \vert (\overline{z}_{l}-z_{l})_{p} Q_{l} 
-\frac{1}{8 \pi} \kappa_{0}^{2} (\overline{z}_{l}-z_{l})_{p} (\lambda_{l}-\frac{1}{2})^{-1} \hat{A}_{l} Q_{l} \\
&\qquad +\frac{1}{8 \pi} \kappa_{0}^{2} (\overline{z}_{l}-z_{l})_{p} (\lambda_{l}-\frac{1}{2})^{-1} \Big[R_l\cdot \sum_{{m=1} \atop {m \neq l}}^{M} \nabla_s\Phi_{\kappa_0}(z_l,z_m) Q_m\nonumber\\
 &\qquad \qquad \qquad+\frac{\rho_{0}}{\rho_{l}-\rho_{0}}  \frac{i}{4\pi}(\kappa_0-\kappa_l)\big[Q_l+\sum_{m\neq\,l}Cap_l\Phi_{\kappa_0}(z_l,z_m)Q_m\big]\Big(8\pi\vert{D_l}\vert+\hat{A}_lCap_l\Big)\nonumber\\
 &\qquad \qquad \qquad -O\left(a^4+a^3\sum_{{m=1}\atop {m\neq l}}^{M}\frac{a^3}{d_{ml}^3}\|\phi_m\|+a^5\|\phi_l\|+a^5\|\psi_l\|\right) \Big]+ O(a^{4} \norm{\phi_{l}})+O(a^{6} \rho_{l}^{-1} \norm{\phi_{l}})\nonumber\\
&= -\frac{1}{8 \pi} \kappa_{0}^{2} (\overline{z}_{l}-z_{l})_{p} (\lambda_{l}-\frac{1}{2})^{-1} \hat{A}_{l} Q_{l}+ O\left(a^{4-\gamma}+a^{3-\gamma}\sum_{{m=1}\atop {m\neq l}}^{M}\frac{a^3}{d_{ml}^3}\|\phi_m\|+a^{5-\gamma}\|\phi_l\|+a^{5-\gamma}\|\psi_l\|\right)\nonumber \\
 &\qquad +O(a^{4-\gamma} \norm{\phi_{l}})+O(\sum_{{m=1}\atop {m\neq l}}^{M}\frac{a^{5-\gamma}}{d_{ml}^2}\|\phi_m\|)+O(\sum_{{m=1}\atop {m\neq l}}^{M}\frac{a^{5-\gamma}}{d_{ml}}\|\phi_m\|).\nonumber
\end{align}
%
Proceeding as in the case of $K^{p}_{\phi_{l}} $, we can write
\begin{align*}
-K^{p}_{\psi_{l}}&= \frac{i \kappa_{0}^{3}}{4 \pi} \Big(\lambda_{l}-\frac{1}{2} \Big)^{-1} \vert D_{l} \vert (\overline{z}_{l}-z_{l})_{p} \int_{\partial D_{l}} \psi_{l}(s)\ d\sigma_{l}(s)-\frac{1}{8 \pi} \kappa_{0}^{2} (\overline{z}_{l}-z_{l})_{p} (\lambda_{l}-\frac{1}{2})^{-1} \hat{A}_{l} \int_{\partial D_{l}} \psi_{l}(s)\ d\sigma_{l}(s) \\
&\qquad -\frac{1}{8 \pi} \kappa_{0}^{2} (\overline{z}_{l}-z_{l})_{p} (\lambda_{l}-\frac{1}{2})^{-1} \int_{\partial D_{l}} (A_{l}-\hat{A}_{l})(s) \psi_{l}(s)\ d\sigma_{l}(s)  +O(a^{4} \norm{\psi_{l}}) +O(a^{6} \rho_{l}^{-1} \norm{\psi_{l}}),
\end{align*}
and therefore using \eqref{claim- diff-phpsi},\eqref{claim-int- diff-phpsi} and \eqref{mult-obs-108r-pr-th}, it follows that
\begin{align*}
-K^{p}_{\psi_{l}}&= O(a^{6} \rho_{l}^{-1} \norm{\psi_{l}})+O(a^{4} \norm{\psi_{l}})\\
&+ \Big(\lambda_{l}-\frac{1}{2} \Big)^{-1}  (\overline{z}_{l}-z_{l})_{p} \Big[\frac{i \kappa^{3}_{0}}{4 \pi}\vert D_{l} \vert-\frac{1}{8 \pi}\kappa^{2}_{0} \hat{A}_{l} \Big] \Big[Q_{l} +Cap_{l}\ u^{I}(z_{l})
+\frac{i(\kappa_{0}-\kappa_{l})}{4 \pi} Q_{l} Cap_{l} \\
&+\sum_{{m=1} \atop {m \neq l}}^{M} \Big(\Phi_{\kappa_{0}}(z_{l},z_{m}) Q_{m} Cap_{l}+\nabla_{t} \Phi_{\kappa_{0}}(z_{l},z_{m})\cdot V_{m} Cap_{l}-\frac{i \kappa_{l}}{4 \pi}\Phi_{\kappa_{0}}(z_{l},z_{m}) Q_{m} C_{l}^{2}\\
&+\nabla_{s}\Phi_{\kappa_{0}}(z_{l},z_{m}) Q_{m}  \int_{\partial D_{l}} (S^{0}_{D_{l}})^{-1} (\cdot-z_{l})(s)\ d\sigma_{l}(s)\Big) + O(a^2+a^3 \norm{\phi_{l}}+\sum_{{m=1} \atop {m \neq l}}^{M} \frac{a^4}{d^{3}_{ml}} \norm{\phi_{m}} )\Big]
\\
&+\frac{1}{8 \pi} \kappa_{0}^{2} (\overline{z}_{l}-z_{l})_{p} (\lambda_{l}-\frac{1}{2})^{-1} \Big[R_l\cdot \sum_{{m=1} \atop {m \neq l}}^{M} \nabla_s\Phi_{\kappa_0}(z_l,z_m) Q_m\\
 &\qquad \qquad \qquad+\frac{\rho_{0}}{\rho_{l}-\rho_{0}}  \frac{i}{4\pi}(\kappa_0-\kappa_l)\big[Q_l+\sum_{m\neq\,l}Cap_l\Phi_{\kappa_0}(z_l,z_m)Q_m\big]\Big(8\pi\vert{D_l}\vert+\hat{A}_lCap_l\Big)\\
 &\qquad \qquad \qquad \qquad \qquad \qquad-O\left(a^4+a^3\sum_{{m=1}\atop {m\neq l}}^{M}\frac{a^3}{d_{ml}^3}\|\phi_m\|+a^5\|\phi_l\|+a^5\|\psi_l\|\right) \Big]\\
 &-\frac{1}{8 \pi} \kappa_{0}^{2} (\overline{z}_{l}-z_{l})_{p} (\lambda_{l}-\frac{1}{2})^{-1} \int_{\partial D_{l}} (A_{l}(s)-\hat{A}_{l}) \Big[\left(S^{{0}}_{D_{l}}\right)^{-1}\left(\frac{i(\kappa_0-\kappa_l)}{4\pi}Q_l\right)\\
&\quad +\left(S^{{0}}_{D_{l}}\right)^{-1}\left(\sum_{m\neq\;l}\left[\left(1-\frac{i\kappa_l}{4\pi}Cap_l\right) \Phi_{\kappa_0}(z_l,z_m)+(s-z_l)\cdot\nabla_s\Phi_{\kappa_0}(z_l,z_m)\right]Q_m\right)\\
&\quad +\left(S^{{0}}_{D_{l}}\right)^{-1}\left( \sum_{m\neq\;l}\nabla_t\Phi_{\kappa_0}(z_l,z_m)\cdot\,V_m\right)+\left(S^{{0}}_{D_{l}}\right)^{-1} u^{I}(z_l)+O\left(a+a^2\|\phi_l\|+\sum\limits_{m\neq\,l}\frac{a^3}{d^3}\|\phi_m\|\right) \Big],
\end{align*}
where the last error estimate mentioned above is in the sense of $L^2$, while the others are point-wise. Therefore we can write
\begin{align}\label{Kpsi}
-&\left(1-\frac{\rho_l}{\rho_0} \right)^{-1} K^{p}_{\psi_{l}}=-\frac{1}{8 \pi} \kappa_{0}^{2} \Big(1-\frac{\rho_{l}}{\rho_{0}}\Big)^{-1} (\overline{z}_{l}-z_{l}) (\lambda_{l}-\frac{1}{2})^{-1} \hat{A}_{l} Q_{l}\\
&-\frac{1}{8\pi} \kappa_{0}^{2} \Big(1-\frac{\rho_{l}}{\rho_{0}}\Big)^{-1} (\overline{z}_{l}-z_{l}) (\lambda_{l}-\frac{1}{2})^{-1} \hat{A}_{l} Cap_{l} \sum_{{m=1} \atop {m \neq l}}^{M} \Phi_{\kappa_{0}}(z_{l},z_{m}) Q_{m} \nonumber\\
&-\frac{1}{8\pi} \kappa_{0}^{2} \Big(1-\frac{\rho_{l}}{\rho_{0}}\Big)^{-1} (\overline{z}_{l}-z_{l}) (\lambda_{l}-\frac{1}{2})^{-1}  \int_{\partial D_{l}} (A_{l}(s)-\hat{A}_{l}) \Big(\sum_{{m=1} \atop {m \neq l}}^{M} \Phi_{\kappa_{0}}(z_{l},z_{m}) Q_{m} \Big) (S^{0}_{D_{l}})^{-1}(s) d\sigma_{l}(s)\nonumber\\
&+O(a^{5-\gamma} \norm{\psi_{l}})+O(a^{4} \norm{\psi_{l}})+O(a^{4-\gamma} \norm{\phi_{l}})
+O(\sum_{{m=1}\atop {m\neq l}}^{M}\frac{a^{5-\gamma}}{d_{ml}}\|\phi_m\|)\nonumber\\
&+O(a^{3-\gamma})
+O(\sum_{{m=1}\atop {m\neq l}}^{M}\frac{a^{5-\gamma}}{d_{ml}^2}\|\phi_m\|)
+O(\sum_{{m=1} \atop {m \neq l}}^{M} \frac{a^{6-\gamma}}{d^{3}_{ml}} \norm{\phi_{m}}).\nonumber
\end{align}
%
Finally, we deal with the term $B^p_{\phi_{l}}$ in the following manner. As in the previous cases, using $\overline{z}_{l}$, we first write
\begin{align}
B^p_{\phi_{l}}&= \sum_{{m=1} \atop {m \neq l}}^{M}\int_{\partial D_l}(s-z_l)_p\;[\lambda_l Id + (K^{{0}}_{D_{l}})^{*}]^{-1} \frac{\partial (S^{\kappa_{0}}_{D_{m}} \phi_{m})}{\partial \nu^l} d\sigma_{l}(s)\\
&= \underbrace{\sum_{{m=1} \atop {m \neq l}}^{M}\int_{\partial D_l}(s-\overline{z}_l)_p\;[\lambda_l Id + (K^{{0}}_{D_{l}})^{*}]^{-1} \frac{\partial (S^{\kappa_{0}}_{D_{m}} \phi_{m})}{\partial \nu^l} d\sigma_{l}(s)}_{B^{p,1}_{\phi_{l}}}\nonumber\\
&\qquad \qquad \qquad \qquad \qquad \qquad+ \underbrace{\sum_{{m=1} \atop {m \neq l}}^{M}\int_{\partial D_l}(\overline{z}_{l}-z_l)_p\;[\lambda_l Id + (K^{{0}}_{D_{l}})^{*}]^{-1} \frac{\partial (S^{\kappa_{0}}_{D_{m}} \phi_{m})}{\partial \nu^l} d\sigma_{l}(s)}_{B^{p,2}_{\phi_{l}}}.\nonumber
\end{align}
%
The first term can be dealt with as
\begin{align*}
B^{p,1}_{\phi_{l}}&={\sum_{{m=1} \atop {m \neq l}}^{M}\int_{\partial D_l}(s-\overline{z}_l)_p\;[\lambda_l Id + (K^{{0}}_{D_{l}})^{*}]^{-1} \frac{\partial (S^{\kappa_{0}}_{D_{m}} \phi_{m})}{\partial \nu^l}\ d\sigma_{l}(s)}\\
&=\sum_{{m=1} \atop {m \neq l}}^{M} \int_{\partial D_l}[\lambda_l Id+ K^{{0}}_{D_{l}}]^{-1}(s-\overline{z}_l)_p \cdot \frac{\partial (S^{\kappa_{0}}_{D_{m}} \phi_{m})}{\partial \nu^{l}} d\sigma_{l}(s)=\sum_{{m=1} \atop {m \neq l}}^{M} O\left(\frac{a^4}{d_{ml}^2}\|\phi_m\|\right),
\end{align*}
using the fact that $(s-\overline{z}_l)_p \in L^{2}_{0}$. To deal with the second term, we note that
\[\int_{\partial D_{l}} (\overline{z}_{l}-z_{l})_{p}\ [\lambda_{l} Id+(K^{0}_{D_{l}})^{*}]^{-1} \frac{\partial (S^{\kappa_{0}}_{D_{m}} \phi_{m})}{\partial \nu^{l}} d\sigma_{l}(s) =(\overline{z}_{l}-z_{l})_{p} \int_{\partial D_{l}} f_{3}(s) \ d\sigma_{l}(s),  \]
where $f_{3}$ satisfies \[[\lambda_{l} Id+(K^{0}_{D_{l}})^{*}] f_{3}=\frac{\partial (S^{\kappa_{0}}_{D_{m}} \phi_{m})}{\partial \nu^{l}}.\] 
From this we can deduce, as in the earlier cases and using \eqref{claim-normsingle-Lay}, that
\begin{align*}
\int_{\partial D_{l}} f_{3}(s) \ d\sigma_{l}(s)&=\Big(\lambda_{l} -\frac{1}{2} \Big)^{-1} \int_{\partial D_{l}} \frac{\partial (S^{\kappa_{0}}_{D_{m}}\phi_{m})}{\partial \nu^l} d\sigma_{l}(s)\\
&=\Big(\lambda_{l} -\frac{1}{2} \Big)^{-1} \Big[-\kappa_0^2\vert{ D_l}\vert \Big( \Phi_{\kappa_0}(z_l,z_m) Q_{m} +\nabla_t\Phi_{\kappa_0}(z_l,z_m)\cdot V_{m} \Big)\\
&\qquad \qquad \qquad \qquad \qquad -\kappa_{0}^{2}\nabla_x\Phi_{\kappa_{0}}(z_{l},z_m)\cdot\Big[ \int_{D_{l}}(x-z_l) dx\Big] Q_{m}-Err4_{m} \Big].
\end{align*}
Therefore
\begin{align}\label{Bp}
-B^{p}_{\phi_{l}}&=-\sum_{{m=1} \atop {m \neq l}}^{M} (\overline{z}_{l}-z_{l})_{p} \Big(\lambda_{l} -\frac{1}{2} \Big)^{-1} \Big[-\kappa_0^2\vert{ D_l}\vert \Big( \Phi_{\kappa_0}(z_l,z_m) Q_{m} +\nabla_t\Phi_{\kappa_0}(z_l,z_m)\cdot V_{m} \Big)\\
&\qquad \qquad \qquad \qquad \qquad -\kappa_{0}^{2}\nabla_x\Phi_{\kappa_{0}}(z_{l},z_m)\cdot\Big[ \int_{D_{l}}(x-z_l) dx\Big] Q_{m}-Err4_{m} \Big]+O\left(\frac{a^4}{d_{ml}^2}\|\phi_m\|\right)\nonumber\\
&=\sum_{{m=1} \atop {m \neq l}}^{M} (\overline{z}_{l}-z_{l}) \Big(\lambda_{l} -\frac{1}{2} \Big)^{-1} \kappa_0^2\vert{ D_l}\vert  \Phi_{\kappa_0}(z_l,z_m) Q_{m}\nonumber\\
&\qquad +\sum_{{m=1} \atop {m \neq l}}^{M} \Big[O(\frac{a^{5-\gamma}}{d^{2}_{ml}} \norm{\phi_{m}}) +O(\frac{a^{6-\gamma}}{d^{3}_{ml}} \norm{\phi_{m}}) +O\left(\frac{a^4}{d_{ml}^2}\|\phi_m\|\right)\Big].\nonumber
\end{align}
Using \eqref{Ep},\eqref{KpG},\eqref{Kphi}, \eqref{Kpsi}, \eqref{Bp} in \eqref{mult-obs-108r3}, we can write
\begin{align}\label{est-v-ball}
V_l&=\int_{\partial D_l}(s-z_l)\phi_{l}\;d\sigma_{l}(s) = V^{dom}_{l}+V^{rem}_{l},
\end{align}
where 
\begin{align*}
&V^{dom}_{l}:=\sum_{{m=1} \atop {m \neq l}}^{M} (\overline{z}_{l}-z_{l}) \Big(\lambda_{l} -\frac{1}{2} \Big)^{-1} \kappa_0^2\vert{ D_l}\vert  \Phi_{\kappa_0}(z_l,z_m) Q_{m}-\frac{1}{8 \pi} \kappa_{0}^{2} (\overline{z}_{l}-z_{l}) (\lambda_{l}-\frac{1}{2})^{-1} \hat{A}_{l} Q_{l} \\
&-\frac{1}{8 \pi} \kappa_{0}^{2} \Big(1-\frac{\rho_{l}}{\rho_{0}}\Big)^{-1} (\overline{z}_{l}-z_{l}) (\lambda_{l}-\frac{1}{2})^{-1} \hat{A}_{l} Q_{l}\\
&-\frac{1}{8\pi} \kappa_{0}^{2} \Big(1-\frac{\rho_{l}}{\rho_{0}}\Big)^{-1} (\overline{z}_{l}-z_{l}) (\lambda_{l}-\frac{1}{2})^{-1} \hat{A}_{l} Cap_{l} \sum_{{m=1} \atop {m \neq l}}^{M} \Phi_{\kappa_{0}}(z_{l},z_{m}) Q_{m} \\
&-\frac{1}{8\pi} \kappa_{0}^{2} \Big(1-\frac{\rho_{l}}{\rho_{0}}\Big)^{-1} (\overline{z}_{l}-z_{l}) (\lambda_{l}-\frac{1}{2})^{-1}  \int_{\partial D_{l}} (A_{l}(s)-\hat{A}_{l}) \Big(\sum_{{m=1} \atop {m \neq l}}^{M} \Phi_{\kappa_{0}}(z_{l},z_{m}) Q_{m} \Big) (S^{0}_{D_{l}})^{-1}(s) d\sigma_{l}(s)\\
&=V^{dom}_{l,1}+V^{dom}_{l,2},
\end{align*}
with
\begin{align*}
V^{dom}_{l,1}:=&-\frac{1}{8 \pi} \kappa_{0}^{2} (\overline{z}_{l}-z_{l}) (\lambda_{l}-\frac{1}{2})^{-1} \hat{A}_{l} Q_{l} -\frac{1}{8 \pi} \kappa_{0}^{2} \Big(1-\frac{\rho_{l}}{\rho_{0}}\Big)^{-1} (\overline{z}_{l}-z_{l}) (\lambda_{l}-\frac{1}{2})^{-1} \hat{A}_{l} Q_{l},\\
V^{dom}_{l,2}:=&\sum_{{m=1} \atop {m \neq l}}^{M} (\overline{z}_{l}-z_{l}) \Big(\lambda_{l} -\frac{1}{2} \Big)^{-1} \kappa_0^2\vert{ D_l}\vert  \Phi_{\kappa_0}(z_l,z_m) Q_{m}\\
&-\frac{1}{8\pi} \kappa_{0}^{2} \Big(1-\frac{\rho_{l}}{\rho_{0}}\Big)^{-1} (\overline{z}_{l}-z_{l}) (\lambda_{l}-\frac{1}{2})^{-1} \hat{A}_{l} Cap_{l} \sum_{{m=1} \atop {m \neq l}}^{M} \Phi_{\kappa_{0}}(z_{l},z_{m}) Q_{m}\\
&-\frac{1}{8\pi} \kappa_{0}^{2} \Big(1-\frac{\rho_{l}}{\rho_{0}}\Big)^{-1} (\overline{z}_{l}-z_{l}) (\lambda_{l}-\frac{1}{2})^{-1}  \int_{\partial D_{l}} (A_{l}(s)-\hat{A}_{l}) \Big(\sum_{{m=1} \atop {m \neq l}}^{M} \Phi_{\kappa_{0}}(z_{l},z_{m}) Q_{m} \Big) (S^{0}_{D_{l}})^{-1}(s) d\sigma_{l}(s),
\end{align*}
and by $V^{rem}_{l}$, we denote the rest of the terms. The remainder $V^{rem}_{l}$
satisfies the estimate
\begin{align*}
\vert V^{rem}_{l} \vert&=O(a^{3-\gamma})+O(a^{4-\gamma} \norm{\phi_{l}})+O(a^{3} \norm{\phi_{l}})\\
&\qquad +O\Big(\sum_{{m=1} \atop {m \neq l}}^{M}\frac{a^{5-\gamma}}{d^{2}_{ml}} \norm{\phi_{m}}\Big)+O\Big(\sum_{{m=1} \atop {m \neq l}}^{M}\frac{a^{5-\gamma}}{d_{ml}}\|\phi_m\|\Big) +O\Big(\sum_{{m=1} \atop {m \neq l}}^{M}\frac{a^{6-\gamma}}{d^{3}_{ml}} \norm{\phi_{m}}\Big) +O\Big(\sum_{{m=1} \atop {m \neq l}}^{M}\frac{a^4}{d_{ml}^2}\|\phi_m\|\Big)\nonumber\\
&=O(a^{3-\gamma})+O\Big( \Big(a^{3}+a^{4-\gamma}+\frac{a^{5-\gamma}}{d^2}+\frac{a^{5-\gamma}}{d^{3\alpha}}+\frac{a^4}{d^2}+\frac{a^4}{d^{3 \alpha}}+\frac{a^{6-\gamma}}{d^{3}}+\frac{a^{6-\gamma}}{d^{3 \alpha+1}} \Big) \norm{\phi} \Big),  \nonumber
\end{align*}
where we use the fact that $\rho_{l}\simeq a^{1+\gamma}$, with $\gamma \geq 0$ and $0\leq \alpha \leq 1$. Also note that in the above estimate, we have majorised $\norm{\phi_{i}}$ by $\norm{\phi}$.\\

\subsection{Proof of the invertibility of the algebraic system}\label{alg-sys} 
We rewrite \eqref{mult-obs-109r6-pr4} in the following compact form;
\begin{equation}\label{compacfrm1}
 (\mathbf{C_I}+\mathbf{B}+\mathbf{B^\prime}+\mathbf{R}_{1}+\mathbf{R}_{2})\mathbf{Q}=\mathbf{Y},
\end{equation}
\noindent
where $\mathbf{Q},\mathbf{Y} \in \mathbb{C}^{M\times 1}\mbox{ and } \mathbf{C_I},\mathbf{B},\mathbf{B^\prime},\mathbf{R}_{1}, \mathbf{R}_{2}\in\mathbb{C}^{M\times M}$ are defined as
\begin{eqnarray}
\mathbf{B}(l,m):=\left\{\begin{array}{ccccc}
  \Phi_{\kappa_0}(z_l,z_m),& \mbox{ if }&l\neq\,m\\
  0,& \mbox{ if }&l=m\\
   \end{array}\right.,\label{mainmatrix-acoustic-small}\\
\nonumber\\
\mathbf{B}^\prime(l,m):=\left\{\begin{array}{ccccc}
 \frac{F_l^\prime}{\kappa_l^2} \nabla_x\Phi_{\kappa_0}(z_l,z_m),& \mbox{ if }&l\neq\,m\\
  0,& \mbox{ if }&l=m\\
  \end{array}\right.,\label{mainmatrix-acoustic-small-prime}\\
  \nonumber\\
  \mathbf{C_I}(l,m):=\left\{\begin{array}{ccccc}
 0,& \mbox{ if }&l\neq\,m\\
  \mathbf{C}_l^{-1},& \mbox{ if }&l=m\\
  \end{array}\right.,\label{mainmatrix-acoustic-small-capinv}\\
  \nonumber\\
\mathbf{Q}:=\left(\begin{array}{cccc}
   {Q}_1 & {Q}_2 & \ldots  & {Q}_M
   \end{array}\right)^\top \text{ and } 
\mathrm{Y}:=\left(\begin{array}{cccc}
     Y_1 & Y_2& \ldots &  Y_M
   \end{array}\right)^\top,\label{coefficient-and-incidentvectors-acoustic-small}
\end{eqnarray}
and $\mathbf{R}_{1}=\mathbf{P}\mathbf{P}_{1}, \mathbf{R}_{2}=\mathbf{P}\mathbf{P}_{2}$ with the matrices $\mathbf{P}, \mathbf{P}_{1}, \mathbf{P}_{2} $ defined as
\begin{equation}
\mathbf{P}(l,m):=\left\{\begin{array}{ccccc}
  \nabla_{t}\Phi_{\kappa_0}(z_l,z_m),& \mbox{ if }&l\neq\,m\\
  0,& \mbox{ if }&l=m\\
   \end{array}\right.,
\label{matrix-P}
\end{equation}
\begin{equation}
\mathbf{P}_{1}(m,n):=\left\{\begin{array}{ccccc}
 \kappa_{0}^{2} (\overline{z}_{m}-z_{m})(\lambda_{m}-\frac{1}{2})^{-1}\Big[\vert D_{m}\vert\\
 \qquad \qquad -\frac{1}{8 \pi} (1-\frac{\rho_{m}}{\rho_{0}})^{-1} \int_{\partial D_{m}} A_{m}(s) (S^0_{D_{m}})^{-1}(s)d\sigma_{m}(s)\Big] \Phi_{\kappa_0}(z_m,z_n),& \mbox{ if }&m\neq\,n\\
  0,& \mbox{ if }&m=n\\
   \end{array}\right.,
\label{matrix-P1}
\end{equation}
\begin{equation}
\mathbf{P}_{2}(m,n):=\left\{\begin{array}{ccccc}
  0,& \mbox{ if }&m\neq\,n\\
  -\frac{1}{4 \pi} \kappa_{0}^{2} (\overline{z}_{m}-z_{m})(\lambda_{m}-\frac{1}{2})^{-1} \hat{A}_{m},& \mbox{ if }&m=n\\
   \end{array}\right..
\label{matrix-P2}
\end{equation}
Our strategy here is to follow the methodology in \cite{Challa-Sini-1}. In this direction, we multiply the system with $\mathbf{Q}^r $ and $\mathbf{Q}^i $, add the resulting identities and then use the fact that the matrices $\mathbf{C_I}, \mathbf{B}$ are self-adjoint to derive the inequality
\begin{align}\label{quadr-est}
&\langle\mathbf{C_I}^r\mathbf{Q}^r, \mathbf{Q}^r\rangle+\langle\mathbf{B}^r\mathbf{Q}^r,\mathbf{Q}^r \rangle+\langle\mathbf{B^\prime}^r\mathbf{Q}^r,\mathbf{Q}^r \rangle+\langle\mathbf{C_I}^r\mathbf{Q}^i, \mathbf{Q}^i\rangle+\langle\mathbf{B}^i\mathbf{Q}^i,\mathbf{Q}^i \rangle+\langle\mathbf{B^\prime}^r\mathbf{Q}^i,\mathbf{Q}^i \rangle\\
&+\langle\mathbf{B^\prime}^i\mathbf{Q}^r, \mathbf{Q}^i\rangle-\langle\mathbf{B^\prime}^i\mathbf{Q}^i, \mathbf{Q}^r\rangle+\langle\mathbf{R_{1}}^r\mathbf{Q}^r, \mathbf{Q}^r\rangle+\langle\mathbf{R_{1}}^r\mathbf{Q}^i, \mathbf{Q}^i\rangle+\langle\mathbf{R_{1}}^i\mathbf{Q}^r, \mathbf{Q}^i\rangle-\langle\mathbf{R_{1}}^i\mathbf{Q}^i, \mathbf{Q}^r\rangle \nonumber\\
&+\langle\mathbf{R_{2}}^r\mathbf{Q}^r, \mathbf{Q}^r\rangle+\langle\mathbf{R_{2}}^r\mathbf{Q}^i, \mathbf{Q}^i\rangle+\langle\mathbf{R_{2}}^i\mathbf{Q}^r, \mathbf{Q}^i\rangle-\langle\mathbf{R_{2}}^i\mathbf{Q}^i, \mathbf{Q}^r\rangle \nonumber\\
&=\langle\mathbf{Y}^r, \mathbf{Q}^r\rangle+\langle\mathbf{Y}^i,\mathbf{Q}^i \rangle\leq2\left(\sum_{m=1}^M \vert Y_m\vert^2\right)^\frac{1}{2}\left(\sum_{m=1}^M \vert Q_m\vert^2\right)^\frac{1}{2}.\nonumber
\end{align}
Let us first consider the case when $(\mathbf{C}_l^{-1})^{r}>0,\ \forall\ l=1,\dots,M $. 
In this case, proceeding as in \cite{Challa-Sini-1}, we can obtain 
\begin{equation}
\langle\mathbf{B}^r\mathbf{Q}^r,\mathbf{Q}^r \rangle+\langle\mathbf{B}^r\mathbf{Q}^i,\mathbf{Q}^i \rangle\geq -\frac{3\tau}{5\pi\,d}\sum_{m=1}^M \vert Q_m\vert^2,
\label{estBrQrquad}
\end{equation}
where $\tau:=\min_{1\leq j,m \leq M,\ j\neq m} cos(\kappa_{0}\vert z_{m}-z_{j} \vert)$ and is assumed to be non-negative. \\
We can also observe that
\begin{align}\label{estCrQrquad}
\langle\mathbf{C_I}^r\mathbf{Q}^r, \mathbf{Q}^r\rangle+\langle\mathbf{C_I}^r\mathbf{Q}^i, \mathbf{Q}^i\rangle&\geq \min\limits_{m} ({\mathbf{C}_m^{-1}})^r \sum_{m=1}^M \vert Q_m\vert^2 
\geq\frac{\min\limits_{1\leq{m}\leq{M}}({\mathbf{C}_m})^r}{(\max\limits_{1\leq{m}\leq{M}} \vert{\mathbf{C}_m}\vert)^2} \sum_{m=1}^M \vert Q_m\vert^2.
\end{align}
We would like to note that $(\mathbf{C}_m^{-1})^{r} $ and $({\mathbf{C}_m})^r $ have the same sign.\\
Next using Cauchy-Schwarz inequality, we can write
\begin{align*}
&\langle\mathbf{B^\prime}^r\mathbf{Q}^r, \mathbf{Q}^r\rangle+\langle\mathbf{B^\prime}^r\mathbf{Q}^i, \mathbf{Q}^i\rangle \geq -\|\mathbf{B^\prime}\|_2  \sum_{m=1}^M \vert Q_m\vert^2,\\
&\langle\mathbf{B^\prime}^i\mathbf{Q}^r, \mathbf{Q}^i\rangle \geq -\norm{\mathbf{B^\prime}^i}_{2} \norm{Q^r} \norm{Q^i} \geq - \norm{\mathbf{B^\prime}}_{2} \sum_{m=1}^{M} \vert Q_{m} \vert^2,\\
&-\langle\mathbf{B^\prime}^i\mathbf{Q}^i, \mathbf{Q}^r\rangle \geq -\norm{\mathbf{B^\prime}^i}_{2} \norm{Q^i} \norm{Q^r} \geq - \norm{\mathbf{B^\prime}}_{2} \sum_{m=1}^{M} \vert Q_{m} \vert^2,
\end{align*}
and therefore
\begin{align}\label{estBrprQrquad}
&\langle\mathbf{B^\prime}^r\mathbf{Q}^r, \mathbf{Q}^r\rangle+\langle\mathbf{B^\prime}^r\mathbf{Q}^i, \mathbf{Q}^i\rangle+\langle\mathbf{B^\prime}^i\mathbf{Q}^r, \mathbf{Q}^i\rangle-\langle\mathbf{B^\prime}^i\mathbf{Q}^i, \mathbf{Q}^r\rangle \\
&\geq-3\max\limits_{1\leq{m}\leq{M}} \left\vert\frac{F_m^\prime}{\kappa_m^2}\right\vert \frac{\kappa_0+1}{4\pi}  \left[\sum_{m,l=1\atop m\neq l}^M\frac{1}{d_{ml}^4}\right]^{\frac{1}{2}}\sum_{m=1}^M \vert Q_m\vert^2\nonumber\\
&\geq-3\max\limits_{1\leq{m}\leq{M}} \left\vert\frac{F_m^\prime}{\kappa_m^2}\right\vert \frac{\kappa_0+1}{4\pi} C \sqrt{M M_{max}} \left[\frac{1}{ d^4 }+\sum^{[d^{-\alpha}]}_{n=1}\hspace{-.1cm} [(2n+1)^3-(2n-1)^3]\frac{{1}}{ n^4d^{4\alpha}}\right]^\frac{1}{2}\sum_{m=1}^M \vert Q_m\vert^2\quad\nonumber\\
&\geq-\max\limits_{1\leq{m}\leq{M}} \left\vert\frac{F_m^\prime}{\kappa_m^2}\right\vert C \sqrt{M M_{max}}  \left[\frac{1}{ d^4 }+\frac{1}{ d^{5\alpha}}\right]^\frac{1}{2}\sum_{m=1}^M \vert Q_m\vert^2\quad,\nonumber
\end{align}
where $C$ denotes a generic constant that is bounded in terms of $a$. Similarly we derive
\begin{align}\label{est-R1}
&\langle\mathbf{R_{1}}^r\mathbf{Q}^r, \mathbf{Q}^r\rangle+\langle\mathbf{R_{1}}^r\mathbf{Q}^i, \mathbf{Q}^i\rangle+\langle\mathbf{R_{1}}^i\mathbf{Q}^r, \mathbf{Q}^i\rangle-\langle\mathbf{R_{1}}^i\mathbf{Q}^i, \mathbf{Q}^r\rangle \\
&\qquad \qquad \geq -C a^{3-\gamma} M M_{max} \left[\frac{1}{ d^2 }+\frac{1}{ d^{3\alpha}}\right]^\frac{1}{2}\left[\frac{1}{ d^4 }+\frac{1}{ d^{5\alpha}}\right]^\frac{1}{2} \sum_{m=1}^M \vert Q_m\vert^2\quad\nonumber
\end{align}
and
\begin{equation}
\begin{aligned}
&\langle\mathbf{R_{2}}^r\mathbf{Q}^r, \mathbf{Q}^r\rangle+\langle\mathbf{R_{2}}^r\mathbf{Q}^i, \mathbf{Q}^i\rangle+\langle\mathbf{R_{2}}^i\mathbf{Q}^r, \mathbf{Q}^i\rangle-\langle\mathbf{R_{2}}^i\mathbf{Q}^i, \mathbf{Q}^r\rangle \geq- C a^{2-\gamma} \sqrt{M M_{max}}  \left[\frac{1}{ d^4 }+\frac{1}{ d^{5\alpha}}\right]^\frac{1}{2}\sum_{m=1}^M \vert Q_m\vert^2\quad.
\end{aligned}
\label{est-R2}
\end{equation}
Making use of (\ref{estBrQrquad}-\ref{est-R2}) in \eqref{quadr-est}, the required estimate \eqref{mazya-fnlinvert-small-ela-2} follows.\\
Let us next consider the case when $(\mathbf{C}_l^{-1})^{r}<0,\ \forall\ l=1,\dots,M $. \\
In this case, we multiply the identity \eqref{quadr-est} with $-1$ and note that using Cauchy-Schwarz inequality, we can still write
\begin{equation}
-\langle\mathbf{Y}^r, \mathbf{Q}^r\rangle-\langle\mathbf{Y}^i,\mathbf{Q}^i \rangle\leq2\left(\sum_{m=1}^M \vert Y_m\vert^2\right)^\frac{1}{2}\left(\sum_{m=1}^M \vert Q_m\vert^2\right)^\frac{1}{2}.
\notag
\end{equation}
As in the previous case, we derive the estimates
\begin{align*}
-\langle\mathbf{B^\prime}^r\mathbf{Q}^r, \mathbf{Q}^r\rangle-\langle\mathbf{B^\prime}^r\mathbf{Q}^i, \mathbf{Q}^i\rangle-\langle\mathbf{B^\prime}^i\mathbf{Q}^r, \mathbf{Q}^i\rangle &+\langle\mathbf{B^\prime}^i\mathbf{Q}^i, \mathbf{Q}^r\rangle  \nonumber \\
&\geq-\max\limits_{1\leq{m}\leq{M}} \left\vert\frac{F_m^\prime}{\kappa_m^2}\right\vert C \sqrt{M M_{max}}  \left[\frac{1}{ d^4 }+\frac{1}{ d^{5\alpha}}\right]^\frac{1}{2}\sum_{m=1}^M \vert Q_m\vert^2,
\end{align*}
\begin{align*}
-\langle\mathbf{R_{1}}^r\mathbf{Q}^r, \mathbf{Q}^r\rangle-\langle\mathbf{R_{1}}^r\mathbf{Q}^i, \mathbf{Q}^i\rangle-\langle\mathbf{R_{1}}^i\mathbf{Q}^r, \mathbf{Q}^i\rangle &+\langle\mathbf{R_{1}}^i\mathbf{Q}^i, \mathbf{Q}^r\rangle \nonumber \\
& \geq -C a^{3-\gamma} M M_{max} \left[\frac{1}{ d^2 }+\frac{1}{ d^{3\alpha}}\right]^\frac{1}{2}\left[\frac{1}{ d^4 }+\frac{1}{ d^{5\alpha}}\right]^\frac{1}{2} \sum_{m=1}^M \vert Q_m\vert^2\quad
\end{align*}
and
\begin{align*}
-\langle\mathbf{R_{2}}^r\mathbf{Q}^r, \mathbf{Q}^r\rangle-\langle\mathbf{R_{2}}^r\mathbf{Q}^i, \mathbf{Q}^i\rangle-\langle\mathbf{R_{2}}^i\mathbf{Q}^r, \mathbf{Q}^i\rangle &+\langle\mathbf{R_{2}}^i\mathbf{Q}^i, \mathbf{Q}^r\rangle \nonumber \\
&\geq- C a^{2-\gamma} \sqrt{M M_{max}}  \left[\frac{1}{ d^4 }+\frac{1}{ d^{5\alpha}}\right]^\frac{1}{2}\sum_{m=1}^M \vert Q_m\vert^2\quad.
\end{align*}
To deal with the terms $-\langle\mathbf{B}^r\mathbf{Q}^r,\mathbf{Q}^r \rangle $ and $-\langle\mathbf{B}^r\mathbf{Q}^i,\mathbf{Q}^i \rangle $, in contrast to the earlier case, we use Cauchy-Schwarz inequality to derive
\begin{equation}
-\langle\mathbf{B}^r\mathbf{Q}^r,\mathbf{Q}^r \rangle-\langle\mathbf{B}^r\mathbf{Q}^i,\mathbf{Q}^i \rangle \geq -C \sqrt{M M_{max}} \left[\frac{1}{ d^2 }+\frac{1}{ d^{3\alpha}}\right]^\frac{1}{2}\sum_{m=1}^M \vert Q_m\vert^2.
\notag
\end{equation}
Also
\begin{align*}
-\langle\mathbf{C_I}^r\mathbf{Q}^r, \mathbf{Q}^r\rangle-\langle\mathbf{C_I}^r\mathbf{Q}^i, \mathbf{Q}^i\rangle\geq \min\limits_{1\leq m \leq M} (-({\mathbf{C}_m^{-1}})^r) \sum_{m=1}^M \vert Q_m\vert^2 
&= \min\limits_{1\leq{m}\leq{M}} \left(\frac{\vert \mathbf{C}^{r}_{m} \vert}{\vert \mathbf{C}_m \vert^2} \right)\sum_{m=1}^M \vert Q_m\vert^2\\
& \geq\frac{\min\limits_{1\leq{m}\leq{M}}\vert \mathbf{C}_m^r \vert}{(\max\limits_{1\leq{m}\leq{M}} \vert{\mathbf{C}_m}\vert)^2} \sum_{m=1}^M \vert Q_m\vert^2.
\end{align*}
Combining the above estimates, we can now conclude \eqref{mazya-fnlinvert-small-ela-3}. We would like to remark that the argument in this case can be applied also to the case when $(\mathbf{C}_l^{-1})^{r}>0,\ \forall\ l=1,\dots,M  $, but the estimate would be worse than that already achieved.\\
We note that in view of remark \ref{sign}, these two are the only possible cases and hence the proof is complete.
%
\subsection{Proof of proposition \ref{Prop-phi-estimate}}\label{phi-estimate}
First of all, we note that using \eqref{claim-single-Lay-dif-2},\eqref{def-Sg} and \eqref{def-SGer}, we can rewrite \eqref{mult-obs-108r-int} on $\partial D_l$ as
\begin{align}\label{mult-obs-108r}
 \left[\lambda_l Id+(K^{{\kappa_0}}_{D_{l}})^{*}\right] \phi_{l} &+\sum_{{m=1} \atop {m \neq l}}^{M} \frac{\partial (S^{\kappa_{0}}_{D_{m}} \phi_{m})}{\partial \nu^{l}}\Big\vert_{\partial D_{l}} +\frac{\partial u^{I}}{\partial \nu^{l}}\\
&=  - \Big(1-\frac{\rho_{l}}{\rho_{0}} \Big)^{-1}\left([(K^{\kappa_{l}}_{D_{l}})^{*}-(K^{\kappa_{0}}_{D_{l}})^{*}] \psi_{l} +[\frac{1}{2}Id+(K^{\kappa_{0}}_{D_{l}})^{*}] \left(S_{D_l}^{\kappa_0}\right)^{-1}\left(S_{D_l}^{\kappa_0}-S_{D_l}^{\kappa_l}\right) \psi_l\right)
\nonumber\\
&=  - \Big(1-\frac{\rho_{l}}{\rho_{0}} \Big)^{-1}\left([(K^{\kappa_{l}}_{D_{l}})^{*}-(K^{\kappa_{0}}_{D_{l}})^{*}] +[(K^{\kappa_{0}}_{D_{l}})^{*}-(K^{{0}}_{D_{l}})^{*}] \left(S_{D_l}^{\kappa_0}\right)^{-1}\left(S_{D_l}^{\kappa_0}-S_{D_l}^{\kappa_l}\right) \right)\psi_l
\nonumber\\
&\qquad - \Big(1-\frac{\rho_{l}}{\rho_{0}} \Big)^{-1}\left({[\frac{1}{2}Id+(K^{{0}}_{D_{l}})^{*}]} \left(S_{D_l}^{\kappa_0}\right)^{-1}\left(S_{D_l}^{\kappa_0}-S_{D_l}^{\kappa_l}\right) \psi_l\right)
\nonumber\\
&=  O(a^2\|\psi_l\|)+O(a^2\cdot a^{-1}\cdot a^2 \|\psi_l\|)+\underbrace{O(a^{-1}{\cdot}a^2\|\psi_l\|)}_{ \in L^2_0} =  O(a^2\|\psi_l\|)+\underbrace{O(a\|\psi_l\|)}_{\in L^2_0} \mbox{ in }L^2
\nonumber\\
&=O\left(a^2\left[1+\|\phi_l\|+\sum_{m\neq\,l}\frac{a}{d_{ml}}\|\phi_m\|+\sum_{m\neq\,l}\frac{a^2}{d^2_{ml}}\|\phi_m\|\right]\right)
\nonumber\\
&\qquad \qquad \qquad \qquad +\underbrace{O\left(a\left[1+\|\phi_l\|+\sum_{m\neq\,l}\frac{a}{d_{ml}}\|\phi_m\|+\sum_{m\neq\,l}\frac{a^2}{d^2_{ml}}\|\phi_m\|\right]\right)}_{ \in {L^2_0}} \mbox { in }L^2
\nonumber\\
&=E_{1,l}+E_{2,l}.\nonumber
\end{align}
Now let us consider the system
\begin{equation}\label{mult-obs-108pr} 
 [\lambda_l Id+(K^{\kappa_0}_{D_{l}})^{*}] \phi_{l}+\sum_{{m=1} \atop {m \neq l}}^{M} \frac{\partial (S^{\kappa_{0}}_{D_{m}} \phi_{m})}{\partial \nu^{l}}\Big\vert_{\partial D_{l}} =  -\frac{\partial u^{I}}{\partial \nu^{l}} \Big\vert_{\partial D_{l}}+E_{1,l}+E_{2,l}.
\end{equation}
We can further rewrite it as 
\begin{eqnarray}\label{cmp-1}
 (\bm{L}+\bm{K})\phi=-\partial_{\nu}u^{\bm{I}}+E_{1}+E_{2},
\end{eqnarray}
where
$\bm{L}:=(\bm{L}_{lm})_{l,m=1}^{M}$ and $\bm{K}:=(\bm{K}_{lm})_{l,m=1}^{M}$, with
\begin{eqnarray}\label{definition-L_K}
\bm{L}_{lm}=\left\{\begin{array}{ccc}
            [\lambda_l Id+(K^{\kappa_0}_{D_{l}})^{*}], & l=m\\
            0, & else
           \end{array}\right.,
&  \ &\bm{K}_{lm}=\left\{\begin{array}{ccc}
           \frac{\partial}{\partial \nu^l} S_{D_m}^{\kappa_0}, & l\neq m\\
            0, & else
           \end{array}\right., 
\end{eqnarray}
\begin{eqnarray}
\partial_{\nu}u^{I}:=\left(\frac{\partial u^I}{\partial \nu^1} \dots \frac{\partial u^I}{\partial \nu^M}\right)^T,\\
\phi:=\left(\phi_1 \dots \phi_M \right)^T,
 \end{eqnarray}
 and $E_{i}:=(E_{i,1}\ ...\ E_{i,M})^T,\ i=1,2 $.
Let us also set
\[\bm{\Phi}^{c}_{\kappa_0}\phi:=\left((\bm{\Phi}^{c}_{\kappa_0}\phi)_1 \dots (\bm{\Phi}^{c}_{\kappa_0}\phi)_M \right) \] where \[(\bm{\Phi}^{c}_{\kappa_0}\phi)_l(s):=\sum_{{m=1}\atop {m \neq l}}^{M} \frac{\partial}{\partial \nu^l} \Phi_{\kappa_0}(s,z_m) Q_m ,\]
\[\bm{\nabla^{1}\Phi}^{c}_{\kappa_0}\phi:=\left((\bm{\nabla^{1}\Phi}^{c}_{\kappa_0}\phi)_1 \dots (\bm{\nabla^{1}\Phi}^{c}_{\kappa_0}\phi)_M \right) \] where \[(\bm{\nabla^{1}\Phi}^{c}_{\kappa_0}\phi)_l(s):=\sum_{{m=1} \atop {m \neq l}}^{M} \nabla_{s} \nabla_{t} \Phi_{\kappa_{0}} (s,z_m)\cdot \nu^{l}(s) \cdot V_m ,\]
\[\bm{\nabla^{2}\Phi}^{c}_{\kappa_0}\phi:=\left((\bm{\nabla^{2}\Phi}^{c}_{\kappa_0}\phi)_1 \dots (\bm{\nabla^{2}\Phi}^{c}_{\kappa_0}\phi)_M \right)  \] where \[(\bm{\nabla^{2}\Phi}^{c}_{\kappa_0}\phi)_l (s):=\sum_{{m=1} \atop {m \neq l}}^{M}(s-z_l)\cdot \frac{\partial}{\partial \nu^l} \nabla_t \nabla_s \Phi_{k_0}(s,z_m) \cdot V_m ,\]
and 
\begin{equation}
\acute{\bm{K}}\phi:=\bm{K}\phi-\bm{\Phi}^{c}_{0}\phi-[\bm{\Phi}^{c}_{\kappa_{0}}-\bm{\Phi}^{c}_{0}]\phi-\bm{\nabla^{1}\Phi}^{c}_{0}\phi-[\bm{\nabla^{1}\Phi}^{c}_{\kappa_0}-\bm{\nabla^{1}\Phi}^{c}_{0}]\phi_z-[\bm{\nabla^{2}\Phi}^{c}_{\kappa_0}-\bm{\nabla^{2}\Phi}^{c}_{0}]\phi,
\notag
\end{equation}
with $(\bm{\nabla^{1}\Phi}^{c}_{0}\phi_z)_l:=(\bm{\nabla^{1}\Phi}^{c}_{\kappa_0}\phi)_l(z_l)=\sum_{m \neq l}\nabla_{s} \nabla_{t} \Phi_{\kappa_{0}} (z_l,z_m)\cdot \nu^{l}(s) \cdot V_m $ .\\
Using the above notations and the fact that the matrix $L$ is invertible, we can write  $\phi$ as
\begin{align} \label{phi-estimate-1}
\phi=&-\bm{L}^{-1}\partial_{\nu}{u}^{I}-\bm{L}^{-1}\bm{\Phi}^{c}_{0}\phi-\bm{L}^{-1}[\bm{\Phi}^{c}_{\kappa_{0}}-\bm{\Phi}^{c}_{0}]\phi-\bm{L}^{-1}\bm{\nabla^{1}\Phi}^{c}_{0}\phi-\bm{L}^{-1}[\bm{\nabla^{1}\Phi}^{c}_{\kappa_0}-\bm{\nabla^{1}\Phi}^{c}_{0}]\phi_z\\
&-\bm{L}^{-1}[\bm{\nabla^{2}\Phi}^{c}_{\kappa_0}-\bm{\nabla^{2}\Phi}^{c}_{0}]\phi-\bm{L}^{-1}\acute{\bm{K}}\phi+\bm{L}^{-1}E_1+\bm{L}^{-1}E_2.
\nonumber
 \end{align}
Now using the fact that $\bm{\Phi}^{c}_{0}\phi $, $\bm{\nabla^{1}\Phi}^{c}_{0}\phi $,$[\bm{\nabla^{1}\Phi}^{c}_{\kappa_0}-\bm{\nabla^{1}\Phi}^{c}_{0}]\phi_z$ and $E_2 $ are mean-free and $\bm{L}^{-1}$ doesn't scale while acting on mean-free vectors but in general $\norm{\bm{L}^{-1}}_{\mathcal{L}(L^2,L^2)}=O\left(\rho_{l}^{-1}\right) $, we obtain
\[\norm{(\bm{L}^{-1}\partial_{\nu}{u}^{I})_l}= O\left(\frac{a}{\vert{\rho_l}\vert}\right),\quad
\norm{(\bm{L}^{-1}\bm{\Phi}^{c}_{0}\phi+\bm{L}^{-1}[\bm{\Phi}^{c}_{\kappa_{0}}-\bm{\Phi}^{c}_{0}]\phi)_l}=O\left(\left[\frac{1}{d^2}+\frac{1}{\vert\rho_l\vert}\right]a\sum_{m \neq l}\vert{Q_m}\vert\right), \]
\[\norm{(\bm{L}^{-1}\bm{\nabla^{1}\Phi}^{c}_{0}\phi)_l}=O\left(\sum_{m\neq\,l}\frac{1}{d_{ml}^3}a\vert{V_m}\vert\right) ,\quad
\norm{(\bm{L}^{-1}[\bm{\nabla^{1}\Phi}^{c}_{\kappa_0}-\bm{\nabla^{1}\Phi}^{c}_{0}]\phi_z)_l}=O\left(\sum_{m\neq\,l} \frac{1}{d_{ml}}a\vert{V_m}\vert \right) ,\]
\[\norm{(\bm{L}^{-1}[\bm{\nabla^{2}\Phi}^{c}_{\kappa_0}-\bm{\nabla^{2}\Phi}^{c}_{0}]\phi)_l}=O\left(\sum_{m\neq\,l}\frac{1}{\vert\rho_l\vert} \frac{1}{d^2_{ml}}a^{2}\vert{V_m}\vert \right), \quad 
\norm{(\bm{L}^{-1}\acute{\bm{K}}\phi)_l}=O\left(\frac{1}{\vert \rho_l \vert}\max_l\sum_{m\neq\,l}\frac{a^4}{d_{ml}^4} \|\phi\|\right) ,\]
\[\norm{(\bm{L}^{-1}E_1)_l}=O\left(\frac{a^2}{\vert \rho_l \vert}\left[1+\|\phi_l\|+\sum_{m\neq\,l}\frac{a}{d_{ml}}\|\phi_m\|+\sum_{m\neq\,l}\frac{a^2}{d^2_{ml}}\|\phi_m\|\right]\right), \]
and
\[\norm{(\bm{L}^{-1}E_2)_l}=O\left(a\left[1+\|\phi_l\|+\sum_{m\neq\,l}\frac{a}{d_{ml}}\|\phi_m\|+\sum_{m\neq\,l}\frac{a^2}{d^2_{ml}}\|\phi_m\|\right]\right) .\]
Therefore we can write
 \begin{align} \label{phi-estimate-2}
 \|\phi_{l}\|  &=O\left(\frac{a}{\vert{\rho_l}\vert}\right)+O\left(\left[\frac{1}{d^2}+\frac{1}{\vert\rho_l\vert}\right]a\sum_{m \neq l}\vert{Q_m}\vert\right)+O\left(\sum_{m\neq\,l}\left(\frac{1}{d_{ml}^3}a\vert{V^{rem}_m}\vert+\frac{1}{\vert\rho_l\vert} \frac{1}{d^2_{ml}}a^{2}\vert{V^{rem}_m}\vert \right)\right)
 \\
 &+O\left(\sum_{m\neq\,l}\left(\frac{1}{d_{ml}^3}a\vert{V^{dom}_{m,1}}\vert+\frac{1}{\vert\rho_l\vert} \frac{1}{d^2_{ml}}a^{2}\vert{V^{dom}_{m,1}}\vert \right)\right)+O\left(\sum_{m\neq\,l}\left(\frac{1}{d_{ml}^3}a\vert{V^{dom}_{m,2}}\vert+\frac{1}{\vert\rho_l\vert} \frac{1}{d^2_{ml}}a^{2}\vert{V^{dom}_{m,2}}\vert \right)\right)
 \nonumber\\
 &+O\left(\frac{1}{\vert \rho_l \vert}\max_l\sum_{m\neq\,l}\frac{a^4}{d_{ml}^4} \|\phi\| \right)+O\left(\frac{a^2}{\vert \rho_l \vert}\left[1+\|\phi_l\|+\sum_{m\neq\,l}\frac{a}{d_{ml}}\|\phi_m\|+\sum_{m\neq\,l}\frac{a^2}{d^2_{ml}}\|\phi_m\|\right]\right)
 \nonumber\\
 &+O\left(a\left[1+\|\phi_l\|+\sum_{m\neq\,l}\frac{a}{d_{ml}}\|\phi_m\|+\sum_{m\neq\,l}\frac{a^2}{d^2_{ml}}\|\phi_m\|\right]\right),
 \nonumber
 \end{align}
 where we have ignored the contribution of the term $\bm{L}^{-1}[\bm{\nabla^{1}\Phi}^{c}_{\kappa_0}-\bm{\nabla^{1}\Phi}^{c}_{0}]\phi_{z} $ since the term $\bm{L}^{-1}\bm{\nabla^{1}\Phi}^{c}_{0}\phi $ is clearly more singular.\\
 To deal with the terms involving $V^{rem} $, we note that
\begin{align}\label{est-vterms-using roughestiv-1prp}
 \sum_{m\neq\,l}\frac{1}{d_{ml}^3}a\vert{V^{rem}_m}\vert
&=a\ O\Big(a^{3-\gamma}+ \Big[a^3+a^{4-\gamma}+\frac{a^{5-\gamma}}{d^2}+\frac{a^{5-\gamma}}{d^{3\alpha}}+\frac{a^4}{d^2}+\frac{a^4}{d^{3 \alpha}}+\frac{a^{6-\gamma}}{d^{3}}+\frac{a^{6-\gamma}}{d^{3 \alpha+1}} \Big] \norm{\phi} \Big) \sum_{m\neq\,l}\frac{1}{d_{ml}^3} \\
&=O\Big(a^{4-\gamma}+ \Big[a^4+a^{5-\gamma}+\frac{a^{6-\gamma}}{d^2}+\frac{a^{6-\gamma}}{d^{3\alpha}}+\frac{a^5}{d^2}+\frac{a^5}{d^{3 \alpha}}+\frac{a^{7-\gamma}}{d^{3}}+\frac{a^{7-\gamma}}{d^{3 \alpha+1}} \Big] \norm{\phi} \Big)O\left(\frac{1}{d^3}+\frac{1}{d^{3\alpha+1}}\right)
\nonumber\\
&=O\Big(\frac{a^{4-\gamma}}{d^3}+\frac{a^{4-\gamma}}{d^{3\alpha+1}}+\Big[\frac{a^4}{d^3}+\frac{a^4}{d^{3\alpha+1}}+\frac{a^{5-\gamma}}{d^3}+\frac{a^{5-\gamma}}{d^{3\alpha+1}}+\frac{a^{6-\gamma}}{d^5}+\frac{a^{6-\gamma}}{d^{3\alpha+3}}+\frac{a^{6-\gamma}}{d^{6\alpha+1}}
\nonumber\\
&\qquad \qquad \qquad \qquad \qquad \qquad +\frac{a^5}{d^5}+\frac{a^5}{d^{3\alpha+3}}+\frac{a^5}{d^{6\alpha+1}}+\frac{a^{7-\gamma}}{d^6}+\frac{a^{7-\gamma}}{d^{3\alpha+4}}+\frac{a^{7-\gamma}}{d^{6\alpha+2}} \Big]\|\phi\|\Big)
\nonumber\\
&=O\Big(a^{4-\gamma-3t}+a^{4-\gamma-s-t}+\Big[a^{4-3t}+a^{4-s-t}+a^{5-\gamma-3t}+a^{5-\gamma-s-t}+a^{6-\gamma-5t}+a^{6-\gamma-s-3t}+a^{6-\gamma-2s-t}
\nonumber\\
&\qquad \qquad  +a^{5-5t}+a^{5-s-3t}+a^{5-2s-t}+a^{7-\gamma-6t}+a^{7-\gamma-s-4t}+a^{7-\gamma-2s-2t} \Big]\|\phi\|\Big).\nonumber
\end{align}
%
\begin{align}\label{est-vterms-using roughestiv-2prp}
 &\sum_{m\neq\,l}\frac{1}{\vert\rho_l\vert} \frac{1}{d^2_{ml}}a^{2}\vert{V^{rem}_m}\vert
=\frac{a^2}{\vert\rho_l\vert} O\Big(a^{3-\gamma}+ \Big[a^{3}+a^{4-\gamma}+\frac{a^{5-\gamma}}{d^2}+\frac{a^{5-\gamma}}{d^{3\alpha}}+\frac{a^4}{d^2}+\frac{a^4}{d^{3 \alpha}}+\frac{a^{6-\gamma}}{d^{3}}+\frac{a^{6-\gamma}}{d^{3 \alpha+1}} \Big] \norm{\phi} \Big) \sum_{m\neq\,l}\frac{1}{d^2_{ml}} \\
&=\frac{1}{\vert\rho_l\vert} O\Big(a^{5-\gamma}+ \Big[a^{5}+a^{6-\gamma}+\frac{a^{7-\gamma}}{d^2}+\frac{a^{7-\gamma}}{d^{3\alpha}}+\frac{a^6}{d^2}+\frac{a^6}{d^{3 \alpha}}+\frac{a^{8-\gamma}}{d^{3}}+\frac{a^{8-\gamma}}{d^{3 \alpha+1}} \Big] \norm{\phi} \Big) O\left(\frac{1}{d^2}+\frac{1}{d^{3\alpha}}\right)
\nonumber\\
&=O\Big(a^{4-2\gamma}+ \Big[a^{4-\gamma}+a^{5-2\gamma}+\frac{a^{6-2\gamma}}{d^2}+\frac{a^{6-2\gamma}}{d^{3\alpha}}+\frac{a^{5-\gamma}}{d^2}+\frac{a^{5-\gamma}}{d^{3 \alpha}}+\frac{a^{7-2\gamma}}{d^{3}}+\frac{a^{7-2\gamma}}{d^{3 \alpha+1}} \Big] \norm{\phi} \Big) O\left(\frac{1}{d^2}+\frac{1}{d^{3\alpha}}\right)
\nonumber\\
&=O\Big(\frac{a^{4-2\gamma}}{d^2}+\frac{a^{4-2\gamma}}{d^{3\alpha}}+\Big[\frac{a^{4-\gamma}}{d^2}+\frac{4-\gamma}{d^{3\alpha}}+\frac{a^{5-2\gamma}}{d^2}+\frac{a^{5-2\gamma}}{d^{3\alpha}}+\frac{a^{6-2\gamma}}{d^4}+\frac{a^{6-2\gamma}}{d^{3\alpha+2}}+\frac{a^{6-2\gamma}}{d^{6\alpha}}
\nonumber\\
&\qquad \qquad \qquad \qquad \qquad \qquad +\frac{a^{5-\gamma}}{d^4}+\frac{a^{5-\gamma}}{d^{3\alpha+2}}+\frac{a^{5-\gamma}}{d^{6\alpha}}+\frac{a^{7-2\gamma}}{d^5}+\frac{a^{7-2\gamma}}{d^{3\alpha+3}}+\frac{a^{7-2\gamma}}{d^{6\alpha+1}} \Big]\|\phi\|\Big)
\nonumber\\
&=O\Big(a^{4-2\gamma-2t}+a^{4-2\gamma-s}+\Big[a^{4-\gamma-2t}+a^{4-\gamma-s}+a^{5-2\gamma-2t}+a^{5-2\gamma-s}+a^{6-2\gamma-4t}+a^{6-2\gamma-s-2t}+a^{6-2\gamma-2s}
\nonumber\\
&\qquad \qquad \qquad +a^{5-\gamma-4t}+a^{5-\gamma-s-2t}+a^{5-\gamma-2s}+a^{7-2\gamma-5t}+a^{7-2\gamma-s-3t}+a^{7-2\gamma-2s-t} \Big]\|\phi\|\Big).\nonumber
\end{align}
%
Now if we assume that $0\leq t < \frac{1}{2}, \; 0 \leq s \leq \frac{3}{2},\; 0 \leq \gamma \leq 1, \frac{s}{3}\leq t$ and $s+\gamma \leq 2$, then we can derive the estimate
\begin{equation}
\sum_{m\neq\,l}\left(\frac{1}{d_{ml}^3}a\vert{V^{rem}_m}\vert+\frac{1}{\vert\rho_l\vert} \frac{1}{d^2_{ml}}a^{2}\vert{V^{rem}_m}\vert \right)= O(a+a^{{\frac{3}{2}}} \norm{\phi}).
\label{est-vrem}
\end{equation}
For the terms involving $V^{dom}_{m,1} $, we can deduce that
\begin{equation}
\sum_{m\neq\,l}\frac{1}{d_{ml}^3}a\vert{V^{dom}_{m,1}}\vert
=a^{4-\gamma} \norm{\phi} O(\frac{1}{d^{3}}+\frac{1}{d^{3\alpha+1}})
=O([a^{4-\gamma-3t}+a^{4-\gamma-s-t}]\norm{\phi}),
\label{est-vdom-1-1}
\end{equation}
\begin{equation}
\sum_{m\neq\,l}\frac{1}{\vert\rho_l\vert} \frac{1}{d^2_{ml}}a^{2}\vert{V^{dom}_{m,1}}\vert
=a^{4-2\gamma} \norm{\phi} O(\frac{1}{d^{2}}+\frac{1}{d^{3\alpha}})
=O([a^{4-2\gamma-2t}+a^{4-2\gamma-s}]\norm{\phi}).
\label{est-vdom-1-2}
\end{equation}
Assuming that $0\leq t < \frac{1}{2}, \; 0 \leq s \leq \frac{3}{2},\; 0 \leq \gamma \leq 1,\; \frac{s}{3}\leq t$ and $s+\gamma \leq 2$, we can derive the estimate
\begin{equation}
\sum_{m\neq\,l}\left(\frac{1}{d_{ml}^3}a\vert{V^{dom}_{m,1}}\vert+ \frac{1}{\vert\rho_l\vert} \frac{1}{d^2_{ml}}a^{2}\vert{V^{dom}_{m,1}}\vert\right)
=O(a\norm{\phi}).
\label{est-vdom-1}
\end{equation}
Similarly for the terms involving $V^{dom}_{m,2} $, we can write
\begin{align}\label{est-vdom-2-1}
\sum_{m\neq\,l}\frac{1}{d_{ml}^3}a\vert{V^{dom}_{m,2}}\vert
&=O\left(a^{3-\gamma} a \Big(\sum_{m\neq l} \frac{1}{d^{3}_{ml}}\Big) \frac{1}{d} \sum_{n=1}^{M} \vert Q_{n} \vert \right)=O\left(a^{4-\gamma} \Big(\frac{1}{d^4}+\frac{1}{d^{3\alpha+2}} \Big)\sum_{n=1}^{M} \vert Q_{n} \vert  \right)\\
&=O([a^{4-\gamma-4t}+a^{4-\gamma-s-2t}]\sum_{n=1}^{M} \vert Q_{n} \vert),\nonumber
\end{align}
\begin{align}\label{est-vdom-2-2}
\sum_{m\neq\,l}\frac{1}{\vert\rho_l\vert} \frac{1}{d^2_{ml}}a^{2}\vert{V^{dom}_{m,2}}\vert
&=O\left(a^{-1-\gamma} a^{3-\gamma} a^2 \Big(\sum_{m\neq l} \frac{1}{d^{2}_{ml}} \Big) \frac{1}{d} \sum_{n=1}^{M} \vert Q_{n} \vert \right)=O\left(a^{4-2\gamma} (\frac{1}{d^3}+\frac{1}{d^{3\alpha+1}})\sum_{n=1}^{M} \vert Q_{n} \vert \right)\\
&=O([a^{4-2\gamma-3t}+a^{4-2\gamma-s-t}]\sum_{n=1}^{M} \vert Q_{n} \vert).\nonumber
\end{align} 
Again assuming that $0\leq t < \frac{1}{2}, \; 0 \leq s \leq \frac{3}{2},\; 0 \leq \gamma \leq 1,\; \frac{s}{3}\leq t$ and $s+\gamma \leq 2$, we can derive the estimate
\begin{equation}
\sum_{m\neq\,l}\left(\frac{1}{d_{ml}^3}a\vert{V^{dom}_{m,2}}\vert+\frac{1}{\vert\rho_l\vert} \frac{1}{d^2_{ml}}a^{2}\vert{V^{dom}_{m,2}}\vert\right)
=O(a^{\frac{1}{2}+}\sum_{n=1}^{M} \vert Q_{n} \vert).
\label{est-vdom-2}
\end{equation}
Using \eqref{alg-system-koneqkl-Q-generalshape}, \eqref{est-vrem}, \eqref{est-vdom-1}, \eqref{est-vdom-2} in \eqref{phi-estimate-2}, we can deduce that for $l=1,\dots, M $,
\begin{align*}
\norm{\phi_l}&=O(a^{-\gamma})+O\left([a^{1-2t}+a^{-\gamma}] M \max \vert \bm{C}_m\vert\right)\\
& +O\left([a^{1-2t}+a^{-\gamma}] M \max \vert \bm{C}_m \vert \left(a^2+a^{3-3t}+a^{3-s-t}+a^{4-2s} \right) \right)\norm{\phi}
\nonumber\\
& +O(a)+O(a^{\frac{3}{2}})\norm{\phi}+O(a)\norm{\phi}+O\left(a^{\frac{1}{2}+} M \max \vert \bm{C}_m \vert \left[1+a+\left(a^2+a^{3-3t}+a^{3-s-t}+a^{4-2s} \right) \norm{\phi} \right]\right)
\nonumber\\
& +O\left(a^{3-\gamma} \sum_{m \neq l} \frac{1}{d^{4}_{ml}} \right)\norm{\phi}+O\left(a^{1-\gamma} \left[1+\norm{\phi}+a\norm{\phi}\left(\sum_{m \neq l} \frac{1}{d_{ml}}\right)+a^{2} \norm{\phi} \left(\sum_{m \neq l} \frac{1}{d^{2}_{ml}}\right) \right] \right).\nonumber
\end{align*}
Now if $M \max \vert \bm{C}_m \vert=O(a^{-h})$, then we obtain
\begin{align*}
\norm{\phi_l}&=O(a^{-\gamma})+O\left(a^{-\gamma} \cdot a^{-h} \right)
+O\left(a^{-\gamma-h} \left(a^2+a^{3-3t}+a^{3-s-t}+a^{4-2s} \right) \right)\norm{\phi}\\
&\qquad +O(a)+O(a^{\frac{3}{2}})\norm{\phi}+O(a)\norm{\phi}+O\left(a^{-h+\frac{1}{2}+} \left[1+a+\left(a^2+a^{3-3t}+a^{3-s-t}+a^{4-2s} \right) \norm{\phi} \right]\right)\\
&\qquad +O\left(a^{3-\gamma} \sum_{m \neq l} \frac{1}{d^{4}_{ml}} \right)\norm{\phi}+O\left(a^{1-\gamma} \left[1+\norm{\phi}+a\norm{\phi}\left(\sum_{m \neq l} \frac{1}{d_{ml}}\right)+a^{2} \norm{\phi} \left(\sum_{m \neq l} \frac{1}{d^{2}_{ml}}\right) \right] \right) \\
&=O(a^{-\gamma})+O(a^{-\gamma-h})+O((a^{2-\gamma-h}+a^{3-3t-\gamma-h}+a^{3-s-t-\gamma-h}+a^{4-2s-\gamma-h}) \norm{\phi})\\
&\qquad +O(a^{1-2t}) \norm{\phi}+O(a^{1-\gamma}) \norm{\phi}+O(a^{2-\gamma-s})\norm{\phi}.
\end{align*}
Therefore provided $ h < \frac{1}{2} $,
\begin{align*}
\norm{\phi}&=O(a^{-\gamma})+O(a^{-\gamma-h})+O(a^{0+})\norm{\phi}+O(a^{1-2t})\norm{\phi}+O(a^{1-\gamma})\norm{\phi}+O(a^{2-\gamma-s})\norm{\phi},
\end{align*}
whence it follows that 
\[\norm{\phi}= O(a^{-\gamma})+ O(a^{-\gamma-h}).\]
Note that if $\gamma=1 $ or $\gamma+s=2 $, to deduce the last step we need to assume that the constant $C_{\rho} $ in \eqref{constant} is large enough.


\end{document}